\documentclass[10pt]{article}
\usepackage[utf8]{inputenc}
\usepackage[T1]{fontenc}
\usepackage{amsfonts}
\usepackage{amsmath,latexsym}
\usepackage{amssymb,amsthm}
\usepackage{graphicx}
\usepackage[english]{babel}
\usepackage{hyperref}

\usepackage{xcolor}

\usepackage{xspace}

\usepackage{mathrsfs}

\usepackage{tikz}
\usetikzlibrary{decorations.pathreplacing}
\usepackage{caption}

\usepackage{cases}
\usepackage[displaymath, mathlines]{lineno}

\usepackage[shortlabels]{enumitem}

\usepackage{footmisc}

\usepackage{geometry}
\usepackage[french]{layout}
\providecommand{\keywords}[1]{\textbf{Keywords:} #1}

\newtheorem{theorem}{Theorem}[section]
\newtheorem{lemma}[theorem]{Lemma}
\newtheorem{proposition}[theorem]{Proposition}

\theoremstyle{definition}
\newtheorem{definition}[theorem]{Definition}
\newtheorem{example}[theorem]{Example}

\theoremstyle{remark}
\newtheorem{remark}[theorem]{Remark}

\definecolor{darkgreen}{rgb}{0,0.5,0}

\def\ds{\displaystyle}

\newcommand{\tr}[1]{{#1}^{\mathrm{Tr}}}

\newcommand{\field}[1]{\ensuremath{\mathbb{#1}}}

\newcommand{\R}{\field{R}\xspace}
\newcommand{\N}{\field{N}\xspace}

\newcommand{\ens}[1]{ \left\{#1\right\} }

\newcommand\abs[1]{\left|#1\right|}
\def\norm#1{\left\|#1\right\|}

\newcommand\diag{\mathrm{diag} \xspace}

\newcommand\Topt{T_{\mathrm{unif}}(\Lambda)}

\newcommand\Id{\mathrm{Id}}

\renewcommand\epsilon{\varepsilon}

\newcommand\dds{\frac{d}{ds}}

\newcommand\pt[1]{\frac{\partial #1}{\partial t}}
\newcommand\px[1]{\frac{\partial #1}{\partial x}}
\newcommand\pxi[1]{\frac{\partial #1}{\partial \xi}}
\newcommand\pas[1]{\frac{\partial #1}{\partial s}}
\newcommand\pnu[1]{\frac{\partial #1}{\partial \nu}}

\newcommand\clos[1]{\overline{#1}}

\newcommand\Tau{\mathcal{T}}
\newcommand\Taub{\mathcal{P}}
\newcommand\Rect{\mathcal{R}}

\newcommand\discset{\mathcal{D}}

\newcommand\ssin{s^{\mathrm{in}}}
\newcommand\ssout{s^{\mathrm{out}}}
\newcommand\ssinb{\bar{s}^{\mathrm{in}}}

\newcommand\ssoutnu{s^{\mathrm{out},\nu}}

\newcommand\SSoutp{\Taub^{\mathrm{out},+}}
\newcommand\SSoutm{\Taub^{\mathrm{out},-}}
\newcommand\discsetout{\mathcal{D}^{\mathrm{out}}}

\newcommand\SSinp{\Taub^{\mathrm{in},+}}
\newcommand\SSinm{\Taub^{\mathrm{in},-}}
\newcommand\discsetin{\mathcal{D}^{\mathrm{in}}}

\newcommand\sinterin{\ssin}
\newcommand\sinterout{\ssout}

\newcommand\sdisc{s^{\mathrm{disc}}}

\newcommand{\ceil}[1]{\left\lceil #1 \right\rceil}

\date{\today}

\title{Boundary stabilization in finite time of one-dimensional linear hyperbolic balance laws with coefficients depending on time and space}

\AtBeginDocument{\def\MR#1{}}

\begin{document}

\author{Jean-Michel Coron\thanks{Sorbonne Universit\'{e}, Universit\'{e} de Paris, CNRS, INRIA,
	Laboratoire Jacques-Louis Lions, \'{e}quipe Cage, Paris, France. E-mail: \texttt{coron@ann.jussieu.fr}.},
Long Hu\thanks{School of Mathematics, Shandong University, Jinan, Shandong 250100, China.  E-mail: \texttt{hul@sdu.edu.cn}.},
Guillaume Olive\thanks{Institute of Mathematics, Jagiellonian University, Lojasiewicza 6, 30-348 Krakow, Poland. E-mail: \texttt{math.golive@gmail.com} or \texttt{guillaume.olive@uj.edu.pl}},
Peipei Shang\thanks{School of Mathematical Sciences, Tongji University, Shanghai 200092, China. E-mail: \texttt{peipeishang@hotmail.com}.}
      }

\maketitle

\begin{abstract}
In this article we are interested in the boundary stabilization in finite time of one-dimensional linear hyperbolic balance laws with coefficients depending on time and space.
We extend the so called ``backstepping method'' by introducing appropriate time-dependent integral transformations in order to map our initial system to a new one which has desired stability properties.
The kernels of the integral transformations involved are solutions to non standard multi-dimensional hyperbolic PDEs, where the time dependence introduces several new difficulties in the treatment of their well-posedness.
This work generalizes previous results of the literature, where only time-independent systems were considered.

\end{abstract}


\section{Introduction and main result}\label{sect intro}

\indent

In the present paper we are interested in the one-sided boundary stabilization in finite time of one-dimensional linear hyperbolic balance laws when the coupling coefficients of the system depend on both time and space variables.
To investigate this stabilization property we use the by now so-called ``backstepping method'', a method that consists in transforming our initial system into another system - called target system - for which the stabilization properties are simpler to study. In finite dimension it relies on a recursive design procedure, which in the case of partial differential equations leads to Volterra transformations of the second kind.

The idea of the possibility to transform a control system into another one in order to study its controllability or stabilization properties already goes back to the development of the control theory for linear finite-dimensional systems in the late 60's, notably with the celebrated work \cite{Bru70} where the author introduced the so-called ``control canonical form''.
Concerning infinite-dimensional systems, such as systems modeled by partial differential equations (PDEs), this approach is much more complicated.
The first attempt in this direction seems to be \cite{Rus78-JMAA}, where the author was interested in the spectral determination (i.e. pole placement) of a particular $2 \times 2$ first-order hyperbolic system.
The difficult task in this approach is, in general, to find an invertible transformation that allows to pass from one system to another and, to the best of our knowledge,  there is no general theory for infinite-dimensional systems so far (if possible).
In \cite{Rus78-JMAA}, the author proposed to use a Volterra transformation of the second kind to pass from what he called the ``control normal form'' to the control canonical form of his hyperbolic system and, in this way, easily solved his spectral determination problem.
In that paper, the use of such a transformation was justified by the analogy with finite-dimensional systems when using transformations of the simple form $\Id+K$ with $K$ being a triangular matrix (while for Volterra transformations of the second kind, $K$ is an integral operator whose kernel is supported in a triangular domain).
The use of a Volterra transformation of the second kind to transform a PDE into another one was also introduced at almost the same time in \cite{Col77}.
Therein, the author showed that a one-dimensional perturbed heat equation, with a time and space dependent perturbation, can be transformed into the classical heat equation by means of a Volterra transformation of the second kind whose kernel has to satisfy some PDE posed on a non-standard domain which is triangular.
The equation that the kernel has to satisfy is now commonly referred to as the ``kernel equation'' and the method was then referred by the author of \cite{Col77} to as the ``method of integral operators''.
The result of \cite{Col77} was notably applied in \cite{Sei84} to deduce the boundary null-controllability in one space dimension of the perturbed heat equation from that of the classical heat equation.

In the 90's a method with similar spirit appeared under the name of ``backstepping method''.
This method was primarily designed to transform, thanks to a recursive procedure, finite-dimensional control systems, which may be nonlinear, into control systems which can be stabilized by means of simple feedback laws.
This method was later on extended to linear PDEs. The first result in this direction is in \cite{ANC98} for a beam equation; see also\cite{LK00} for a Burgers' equation. However, the main breakthrough for the PDEs case are in \cite{BKL01,2002-Balogh-Krstic-EJC, Liu03}, which deal with 1-D heat equations and where Volterra transformations of the second kind are introduced or used. In particular in \cite{2002-Balogh-Krstic-EJC} the backstepping recursive procedure in finite dimension is applied to the semi-discretized finite difference approximation of these equations and it
is proved that, as the spatial step size tends to 0, the backstepping transformation at the finite dimensional level is converging to a Volterra transformation of the second kind. The fact that the transformation which appears with this approach is a Volterra transformation of the second kind comes from the recursive procedure of the backstepping method.
With this method the authors, directly inspired by the backstepping in finite dimension, independently arrived at the use of exactly the same transformation as in the two above mentioned pioneering references \cite{Rus78-JMAA} and \cite{Col77}.
This is the reason of the use of the terminology ``backstepping'' for the construction of stabilizing feedback laws relying on the use of Volterra transformations of the second kind to transform a given control PDE to another control PDE (called the target system) which can be easily stabilized (usually with the null feedback law).

The use of Volterra transformations of the second kind also matches very well with the boundary stabilization of one-dimensional systems since this transformation somehow removes the undesirable terms (or adds desirable ones) of the equation by ``bringing'' them to the part of the boundary where the feedback is acting (through the kernel equations).
This approach rapidly turned out to be very successful in the study of the boundary stabilization of various important PDEs such as heat equations, wave equations, Schr\"{o}dinger equations, Korteweg-de Vries equations, Kuramoto-Sivashinsky equations, etc. and it eventually leads to the by now reference book \cite{KS08} on this subject.
This method is nowadays systematically used as a standard tool to analyze the boundary stabilization for (mainly one-dimensional) PDEs.
This method has also received some recent developments.
Notably, the use of Volterra transformations of the second kind has started to show some serious limitation for some problems and it has been replaced by more general integral transformations such as Fredholm integral transformations (see e.g. \cite{CL14,CL15,BAK15,CHO16,CHO17,CGM18}) or other kind of integral transformations (see e.g. \cite{SGK09}).
In these cases the transformation on the state does not have any special structure and the method is no longer related to the finite dimensional backstepping approach. It is related to the older notion of feedback equivalence, as initiated in \cite{Bru70}; see also \cite{1972-Kalman-in-Weiss}, \cite[Section 5.7]{1985-Wonham-book}, and \cite[Section 5.2]{1998-Sontag-book}.

Concerning more specifically systems of hyperbolic equations and the finite-time stabilization property, which is the focus of this article, the first result was obtained in \cite{CVKB13}.
In this paper, the authors developed the original backstepping method to prove the boundary stabilization of a $2 \times 2$ hyperbolic system in finite time, with the best time that can be achieved.
The generalization of the result of \cite{CVKB13} to $n \times n$ systems was a non-trivial task which was eventually solved in \cite{HDMVK16,HVDMK19} using the ideas introduced previously in \cite{HDM15} for $3 \times 3$ systems.
The key point was to add additional constraints on the kernel to obtain a specific structure of the coupling parameter in the target system.
The time of stabilization found in \cite{HDMVK16,HVDMK19} was then improved in \cite{ADM16,CHO17}, using two different target systems.

The goal of the present article is to extend the results of the previously mentioned references to time-dependent systems.
For the finite-time stabilization of non-autonomous hyperbolic systems, the only works that we are aware of are \cite{DJK16} and \cite{AA18} which concerned a single equation with constant speed.
Therefore, the non-autonomous case for systems was still left without investigation.
The introduction of the time variable in the coupling coefficients obviously complicates the whole situation.
As in \cite{Col77,DJK16,AA18} we need to introduce integral transformations with time-dependent kernels, resulting in much more complex kernel equations to solve.
Finally, in addition to the previous references, we would also like to mention the work \cite{Wan06} on time-dependent quasilinear hyperbolic systems concerning the related notion of controllability and the works \cite{SK05,KD19}, with the references therein, concerning the stabilization of time-dependent parabolic systems (where strong regularity conditions are required to make the backstepping method work, because of the result of \cite{Kan90}).

The rest of this paper is organized as follows.
In the remaining part of Section \ref{sect intro} we present in details the class of hyperbolic systems that we consider and we state our main result.
In Section \ref{sect syst transfo} we perform several transformations to show that our initial system can be mapped to a target system which is finite-time stable with desired settling time.
In Section \ref{sect exist kern} we prove the existence and regularity of the kernels of the integral transformations that were used in the previous section.
Finally, we gathered in Appendices \ref{sect broad sol}, \ref{app prop} and \ref{app lemma} some auxiliary results.

\subsection{System description}\label{sec:system}

In this article, we focus on the following general $n\times n$ linear hyperbolic systems, which appear for instance in the linearized Saint-Venant equations,  plug flow chemical reactors equations, heat exchangers equations and many other physical models of balance laws (see e.g. \cite[Chapter 1]{BC16}) around time-varying trajectories:
\begin{equation}\label{syst init}
\left\{\begin{array}{l}
\ds \pt{y}(t,x)+\Lambda(t,x) \px{y}(t,x)=M(t,x) y(t,x), \\
y_-(t,1)=u(t), \quad y_+(t,0)=Q(t)y_-(t,0),  \\
y(t^0,x)=y^0(x).
\end{array}\right.
\end{equation}
In \eqref{syst init}, $t>t^0 \geq 0$ and $x \in (0,1)$, $y(t,\cdot)$ is the state at time $t$, $y^0$ is the initial data at time $t^0$ and $u(t)$ is the control at time $t$.
The matrix $M$ couples the equations of the system inside the domain and the matrix $Q$ couples the equations of the system on the boundary $x=0$.
We assume that the matrix $\Lambda$ is diagonal:
\begin{equation}\label{Lambda diag}
\Lambda=\diag(\lambda_1,\ldots,,\lambda_n).
\end{equation}
We denote by $m\in\ens{1,\ldots,n-1}$ the number of equations with negative speeds and by $p=n-m \in\ens{1,\ldots,n-1}$ the number of equations with positive speeds (all along this work we assume that $n \geq 2$, see Remark \ref{rem m=n} below for the case $m=n \geq 1$).
We assume that there exists some $\epsilon>0$ such that, for every $t \geq 0$ and $x \in [0,1]$, we have
\begin{equation}\label{hyp speeds}
\lambda_1(t,x)<\cdots <\lambda_m(t,x)<-\epsilon <0< \epsilon<\lambda_{m+1}(t,x) < \cdots < \lambda_n(t,x),
\end{equation}
and, for every $i \in \ens{1,\ldots,n-1}$,
\begin{equation}\label{hyp speeds bis}
\lambda_{i+1}(t,x)-\lambda_i(t,x)>\epsilon.
\end{equation}
Assumptions \eqref{hyp speeds} and \eqref{hyp speeds bis} will be commented, respectively, in Remarks \ref{rem hyp 2} and \ref{rem hyp bis} below.

All along this paper, for a vector (or vector-valued function) $v \in \R^n$ and a matrix (or matrix-valued function) $A \in \R^{n \times n}$, we use the notation
$$
v=\begin{pmatrix} v_- \\ v_+ \end{pmatrix},
\qquad
A=\begin{pmatrix} A_{--} & A_{-+} \\ A_{+-} & A_{++} \end{pmatrix},
$$
where $v_- \in \R^m$, $v_+ \in \R^{p}$ and $A_{--} \in \R^{m \times m}$, $A_{-+} \in \R^{m \times p}$, $A_{+-} \in \R^{p \times m},$ $A_{++} \in \R^{p \times p}$.

We will always assume the following regularities for the parameters involved in the system \eqref{syst init}:
\begin{equation}\label{reg parameters}
\begin{array}{c}
\Lambda \in C^1([0,+\infty)\times [0,1])^{n \times n},
\qquad
M \in C^0([0,+\infty)\times[0,1])^{n \times n},
\qquad
Q \in C^0([0,+\infty))^{p \times m}, \\
\ds \Lambda, \px{\Lambda}, M \in L^{\infty}((0,+\infty)\times(0,1))^{n \times n},
\qquad
Q \in L^{\infty}(0,+\infty)^{p \times m}.
\end{array}
\end{equation}

In this article, we use the notion of ``solution along the characteristics'' or ``broad solution'' for the system \eqref{syst init}.
The necessary background on this notion is given in Appendix \ref{sect broad sol} (see also \cite[Section 3.4]{Bre00} for more information).
For the moment we only need to know that, for every $F \in L^{\infty}((0,+\infty)\times(0,1))^{m \times n}$, $t^0\geq 0$ and $y^0 \in L^2(0,1)^n$, there exists a unique (broad) solution $y \in  C^0([t^0,+\infty);L^2(0,1)^n)$ to the system \eqref{syst init} with
\begin{equation}\label{feedback law}
u(t)=\int_0^1 F(t,\xi) y(t,\xi) \, d\xi.
\end{equation}
The relation \eqref{feedback law} will be called the feedback law and the function $F$ will be called the state-feedback gain function.

Let us now give the notion of stability that we are interested in this article (see, for example, \cite[Definition]{BB98}, \cite[Section 3.2]{2005-Bacciotti-Rosier-book} and \cite[Definitions 11.11 and 11.27]{2007-Coron-book} for time-varying systems in finite dimension).

\begin{definition}\label{def stab FT}
Let $T>0$.
We say that the system \eqref{syst init} with feedback law \eqref{feedback law} is finite-time stable with settling time $T$ if the following two properties hold:
\begin{enumerate}[(i)]
\item\label{def stabi}\textbf{Finite-time global attractor.} For every $t^0\geq 0$ and $y^0\in L^2(0,1)^n$,
\begin{equation}\label{attracT}
y(t^0+T,\cdot)=0.
\end{equation}

\item\label{def stabl}\textbf{Uniform stability.}
For every $\varepsilon>0$, there exists $\delta>0$ such that, for
every $t^0\geq 0$ and $y^0\in L^2(0,1)^n$,
\begin{equation}\label{FT cond}
\left(\norm{y^0}_{L^2(0,1)^n}\leq \delta\right)
\quad \Longrightarrow \quad
\left( \norm{y(t,\cdot)}_{L^2(0,1)^n}\leq \varepsilon,\quad \forall t \geq t^0 \right).
\end{equation}
\end{enumerate}
\end{definition}

\begin{remark}
The property \eqref{FT cond} guarantees that, inside any time interval of the form $[t^0,t^0+T]$, the solution is controlled solely by its value at the initial time $t^0$, even if this time $t^0$ is very large.
For our system \eqref{syst init} this property is in fact a consequence of the first property \eqref{attracT} and that the state-feedback gain function $F$ is in $L^\infty((0,+\infty) \times (0,1))^{m\times n}$ (see Remark \ref{rem FTGA implies US}). 
Such an implication is in general not true for time-dependent hyperbolic systems.
A simple example is the following transport equation:
$$
\left\{\begin{array}{l}
\ds \pt{y}(t,x)- \px{y}(t,x)=0, \\
\ds y(t,1)=f(t)\int_0^1 y(t,\xi)d\xi, \\
\ds y(t^0,x)=y^0(x),
\end{array}\right.
$$
where $f \in C^\infty([0,+\infty))$ is such that, for every $k \in \N$,
$$
\left\{\begin{array}{l}
\ds f(t)=0, \quad \forall t\in [2k,2k+1], \\
\ds f(t)=t, \quad \forall t\in \left[2k+\frac{5}{4},2k+\frac{7}{4}\right],
\end{array}\right.
$$
(note that $f \not\in L^{\infty}(0,+\infty)$).
Then the finite-time global attractor property holds (with $T=3$) but the uniform stability property does not hold (consider the sequences $y^0_{\delta}(x)=\delta$ for every $x \in (0,1)$ and $t^0_\delta=2\ceil{\frac{1}{\delta}}+\frac{5}{4}$, where $\ceil{\cdot}$ denotes the ceiling function).
\end{remark}

\begin{remark}
As we are trying to find a state-feedback gain function $F$ so that \eqref{syst init} with feedback law \eqref{feedback law} is finite-time stable, let us first point out that, in general, $F=0$ does not work.
A simple example is provided by the $2 \times 2$ system with constant coefficients ($t^0=0$ to simplify)
$$
\left\{\begin{array}{l}
\ds \pt{y_-}(t,x)-\px{y_-}(t,x)=-c y_+(t,x), \\
\ds \pt{y_+}(t,x)+\px{y_+}(t,x)=-c y_-(t,x), \\
y_-(t,1)=0, \quad y_+(t,0)=y_-(t,0),  \\
y(0,x)=y^0(x),
\end{array}\right.
$$
which is exponentially unstable for $c>\pi$ (see e.g. \cite[Proposition 5.12]{BC16} with $y_-(t,x)=S_1(t,1-x)$ and $y_+(t,x)=S_2(t,1-x)$), and thus not finite-time stable.
\end{remark}

\subsection{The characteristics}
\label{subsection:char}

To state the main result of this paper we need to introduce the characteristic curves associated with system \eqref{syst init}.
To this end, it is convenient to first extend $\Lambda$ to a function of $\R^2$ (still denoted by $\Lambda$).

\begin{remark}\label{rem ext Lambda}
This extension procedure can be done in such a way that the properties \eqref{Lambda diag}, \eqref{hyp speeds}, \eqref{hyp speeds bis} and \eqref{reg parameters} remain valid on $\R^2$.
We can take for instance
$$
\bar{\lambda}_i(t,x)=
\left\{\begin{array}{ll}
\ds \lambda_i(t,x) & \mbox{ if } t \geq 0,\\
\ds \lambda_i(0,x)+\delta\left(\lambda_i(0,x)-\lambda_i\left(1-e^{t/\delta},x\right)\right) & \mbox{ if } t<0,
\end{array}\right.
$$
where $\delta>0$ is small enough so that $-\epsilon+4\delta\max_i \norm{\lambda_i}_{L^{\infty}((0,1)\times(0,1))}<-\epsilon/2$ to guarantee the properties \eqref{hyp speeds} and \eqref{hyp speeds bis} with $\epsilon/2$ in place of $\epsilon$.
This extends the function to $\R\times[0,1]$.
We can use a similar procedure to then extend it to $\R^2$.
We can check that the results of this paper do not depend on such a choice of extension (all the important data are uniquely determined on the domain of interest $(0,+\infty)\times(0,1)$).
\end{remark}

\subsubsection{The flow}

For every $i\in \{1,\cdots,n\}$,  let $\chi_i$ be the flow associated with $\lambda_i$, i.e. for every $(t,x) \in \R \times \R$, the function $s \mapsto \chi_i(s;t,x)$ is the solution to the ODE
\begin{equation}\label{ODE xi}
\left\{\begin{array}{l}
\ds \pas{\chi_i}(s;t,x)=\lambda_i(s,\chi_i(s;t,x)), \quad \forall s \in \R, \\
\ds \chi_i(t;t,x)=x.
\end{array}\right.
\end{equation}
The existence and uniqueness of the solution to the ODE \eqref{ODE xi} follows from the (local) Cauchy-Lipschitz theorem and this solution is global since $\lambda_i$ is bounded (by the finite time blow-up theorem, see e.g. \cite[Theorem II.3.1]{Har02}).
The uniqueness of the solution to the ODE \eqref{ODE xi} also yields the group property
\begin{equation}\label{chi invariance}
\chi_i\left(\sigma;s,\chi_i(s;t,x)\right)=\chi_i(\sigma;t,x), \quad \forall \sigma \in \R.
\end{equation}

\noindent
By classical regularity results on ODEs (see e.g. \cite[Theorem V.3.1]{Har02}), $\chi_i$ has the regularity
\begin{equation}\label{reg xi}
\chi_i \in C^1(\R^3),
\end{equation}
and, for every $s,t,x \in \R$, we have
\begin{equation}\label{deriv chi}
\pt{\chi_i}(s;t,x)=-\lambda_i(t,x)e^{\int_t^s \px{\lambda_i}(\theta,\chi_i(\theta;t,x)) \, d\theta},
\qquad
\px{\chi_i}(s;t,x)=e^{\int_t^s \px{\lambda_i}(\theta,\chi_i(\theta;t,x)) \, d\theta}.
\end{equation}

\noindent
Note in particular that
\begin{equation}\label{chi increases}
\left\{\begin{array}{cl}
\ds \pt{\chi_i}(s;t,x)>0 & \mbox{ if } i \in \ens{1,\ldots,m}, \\
\ds \pt{\chi_i}(s;t,x)<0 & \mbox{ if } i \in \ens{m+1,\ldots,n}, \\
\ds \px{\chi_i}(s;t,x)>0. & \\
\end{array}\right.
\end{equation}

\subsubsection{The entry and exit times}

For every $i\in \{1,\cdots,n\}$, $t\in \R$ and $x \in [0,1]$, let $\ssin_i(t,x),\ssout_i(t,x) \in \R$ be the entry and exit times of the flow $\chi_i(\cdot;t,x)$ inside the domain $[0,1]$, i.e. the respective unique solutions to
\begin{equation}\label{def sin sout}
\left\{\begin{array}{l}
\chi_i(\ssin_i(t,x);t,x)=1, \qquad \chi_i(\ssout_i(t,x); t, x)=0, \qquad \mbox{ if } i \in \ens{1,\ldots,m}, \\
\chi_i(\ssin_i(t,x);t,x)=0, \qquad \chi_i(\ssout_i(t,x);t,x)=1, \qquad  \mbox{ if } i \in \ens{m+1,\ldots,n}.
\end{array}\right.
\end{equation}
The existence and uniqueness of $\ssout_i(t,x)$ and $\ssin_i(t,x)$ are guaranteed by the assumption \eqref{hyp speeds}.
Note that we always have
\begin{equation}\label{tentredeux}
\ssin_i(t,x) \leq t \leq \ssout_i(t,x)
\end{equation}
and the cases of equalities are given by
\begin{equation}\label{sin is zero}
\left\{\begin{array}{l}
\ssin_i(t,x)=t \quad \Longleftrightarrow \quad x=1, \qquad \ssout_i(t,x)=t \quad \Longleftrightarrow \quad x=0, \qquad \mbox{ if } i \in \ens{1,\ldots,m}, \\
\ssin_i(t,x)=t \quad \Longleftrightarrow \quad x=0, \qquad \ssout_i(t,x)=t \quad \Longleftrightarrow \quad x=1, \qquad  \mbox{ if } i \in \ens{m+1,\ldots,n}.
\end{array}\right.
\end{equation}

\noindent
It readily follows from \eqref{chi invariance} and the uniqueness of $\ssin_i,\ssout_i$ that, for every $s \in [\ssin_i(t,x),\ssout_i(t,x)]$,
\begin{equation}\label{invariance}
\ssin_i\left(s,\chi_i(s;t,x)\right)=\ssin_i(t,x),
\quad
\ssout_i\left(s,\chi_i(s;t,x)\right)=\ssout_i(t,x).
\end{equation}

\noindent
From \eqref{reg xi} and by the implicit function theorem, we have
\begin{equation}\label{reg sin sout}
\ssin_i, \ssout_i \in C^1(\R \times [0,1]).
\end{equation}
Moreover, integrating the ODE \eqref{ODE xi} and using the assumption \eqref{hyp speeds}, we have the following bounds, valid for every $t \in \R$ and $x \in [0,1]$,
\begin{equation}\label{bounds sin sout}
t-\ssin_i(t,x)<\frac{1}{\epsilon}, \quad \ssout_i(t,x)-t<\frac{1}{\epsilon}.
\end{equation}
On the other hand, differentiating \eqref{def sin sout} and using \eqref{chi increases} with \eqref{hyp speeds}, we see that, for every $t \in \R$ and $x \in [0,1]$, we have
\begin{equation}\label{increases}
\left\{\begin{array}{cl}
\ds \pt{\ssin_i}(t,x)>0, \quad \pt{\ssout_i}(t,x)>0, & \\
\ds \px{\ssin_i}(t,x)>0, \quad \px{\ssout_i}(t,x)>0 & \mbox{ if } i \in \ens{1,\ldots,m}, \\
\ds \px{\ssin_i}(t,x)<0, \quad \px{\ssout_i}(t,x)<0 & \mbox{ if } i \in \ens{m+1,\ldots,n}.
\end{array}\right.
\end{equation}

\noindent
Finally, from the assumption \eqref{hyp speeds} and classical results on comparison for ODEs (see e.g. \cite[Corollary III.4.2]{Har02}), we have, for every $t \in \R$ and $x \in [0,1]$,
\begin{equation}\label{comparison ssout}
\left\{\begin{array}{c}
\ssin_m(t,x)<\ldots<\ssin_1(t,x) \quad \mbox{ if } x \neq 1,
\quad
\ssout_1(t,x)<\ldots<\ssout_m(t,x) \quad \mbox{ if } x \neq 0,
\\
\ssin_{m+1}(t,x)<\ldots<\ssin_n(t,x) \quad \mbox{ if } x \neq 0,
\quad
\ssout_n(t,x)<\ldots<\ssout_{m+1}(t,x) \quad \mbox{ if } x \neq 1.
\end{array}\right.
\end{equation}

\subsection{Main result and comments}

We are  now in position to state the main result of this paper:
\begin{theorem}\label{main thm}
Let $\Lambda$, $M$ and $Q$ satisfy \eqref{Lambda diag}, \eqref{hyp speeds}, \eqref{hyp speeds bis} and \eqref{reg parameters}.
Then, there exists a state-feedback gain function $F\in L^\infty((0,+\infty) \times (0,1))^{m\times n}$ such that the system \eqref{syst init} with feedback law \eqref{feedback law} is finite-time stable with settling time $\Topt$ defined by
\begin{equation}\label{def Topt}
\Topt=\sup_{t^0 \geq 0} \, \ssout_{m+1} \left(\ssout_m\left(t^0,1\right),0\right)-t^0.
\end{equation}
Moreover, if for some $\tau>0$, $\Lambda$, $M$ and $Q$ are $\tau$-periodic with respect to time (that is $\Lambda(t+\tau,x)=\Lambda(t,x)$ for every $t \geq 0$ and $x \in [0,1]$, same for $M$ and $Q$) then one can also impose to $F$ to be $\tau$-periodic with respect to time (almost everywhere).
\end{theorem}

Let us remark that, thanks to \eqref{tentredeux}, \eqref{sin is zero} and \eqref{bounds sin sout}, we always have
$$0<\Topt<\frac{2}{\epsilon}.$$
Note as well, thanks to \eqref{comparison ssout} and the first line in \eqref{increases}, that we have
$$
\Topt=\max_{j \in \ens{m+1,\ldots,n}} \, \max_{i \in \ens{1,\ldots,m}} \, \sup_{t^0 \geq 0} \, \ssout_j\left(\ssout_i\left(t^0,1\right),0\right)-t^0
.$$

\begin{example}\label{exex}
Theorem \ref{main thm} applies for instance to the following coupled $2 \times 2$ system:
\begin{equation}\label{syst ex}
\left\{\begin{array}{l}
\ds \pt{y_1}(t,x)-\px{y_1}(t,x)=m_{11}(t,x)y_1(t,x)+m_{12}(t,x)y_2(t,x), \\
\ds \pt{y_2}(t,x)+\left(1+\frac{1}{1+t}\right)\px{y_2}(t,x)=m_{21}(t,x)y_1(t,x)+m_{22}(t,x)y_2(t,x), \\
y_1(t,1)=u(t), \quad y_2(t,0)=q(t)y_1(t,0),  \\
y(t^0,x)=y^0(x),
\end{array}\right.
\end{equation}
where $M={(m_{ij})}_{1 \leq i,j \leq 2}$ and $Q=(q)$ are any parameters with the regularity \eqref{reg parameters}.
Let us show how to compute $\Topt$ for this example.
First of all, it is clear that $\chi_1(s;t,x)=-s+t+x$, so that $\ssout_1(t^0,1)=t^0+1$.
On the other hand, we have $\chi_2(s;t,x)=s+\ln(1+s)-t-\ln(1+t)+x$.
Therefore, $h(t^0)=\ssout_2(t^0+1,0)-t^0$ solves $\Psi(h(t^0),t^0)=0$, where
$$\Psi(h,t^0)=h+\ln(1+h+t^0)-2-\ln(2+t^0).$$
Taking the derivative of the relation $\Psi(h(t^0),t^0)=0$ and using the fact that $h \geq 1$ by \eqref{tentredeux}, we see that $h'(t^0) \geq 0$, so that $h$ is non-decreasing.
Since $h \leq 2$ by \eqref{bounds sin sout}, the function $h$ is thus a bounded non-decreasing function and, consequently, $\lim_{t^0 \to +\infty} h(t^0)$ exists and is equal to $\sup_{t^0 \geq 0} h(t^0)=\Topt$.
Writing the relation $\Psi(h(t^0),t^0)=0$ as follows for $t^0>0$
$$h(t^0)+\ln\left(\frac{1}{t^0}+\frac{h(t^0)}{t^0}+1\right)-2-\ln\left(\frac{2}{t^0}+1\right)=0,$$
and letting $t^0 \to +\infty$ we obtain the value $\Topt=2$.
\end{example}

\begin{remark}\label{rem Topt}
Observe that the time $\Topt$ does not depend on the parameters $M$ and $Q$.
It depends only on  $\Lambda$ on $[0,+\infty)\times (0,1)$.
Moreover, this is the best time one can obtain, uniformly with respect to all the possible choices of $M$ and $Q$ (this explains our notation ``$\Topt$'').
More precisely,
$$\Topt=\min E,$$
where $E$ is the set of $T>0$ such that, for every $M$ and $Q$ with the regularity \eqref{reg parameters}, there exists a state-feedback gain function $F \in L^\infty((0,+\infty)\times(0,1))^{m\times n}$ so that the system \eqref{syst init} with feedback law \eqref{feedback law} is finite-time stable with settling time $T$.
Indeed, Theorem \ref{main thm} establishes that $\Topt \in E$, so that $E \neq \emptyset$.
On the other hand, taking $M=0$ and the constant matrix
$$
\begin{aligned}
Q=\begin{array}{c@{\hspace{-5pt}}l}\left(\begin{array}{ccc|c}
 & 0 &&1\\ \hline
 & 0 && 0
\end{array}\right)
& \ \begin{array}{l}\rule{-2mm}{0mm}\Big\}1\\
\rule{-2mm}{0mm}\Big\}p-1
\end{array}\\[-15pt]
\ \begin{array}{c}\underbrace{\rule{12mm}{0mm}}_{m-1}\underbrace{\rule{1mm}{0mm}}_{1}
\end{array}
\end{array}.
\end{aligned}
$$
we can check from the very definition of broad solution (see Definition \ref{def broad sol}) that, if $T<\Topt$, then there exist $t^0 \geq 0$ and $y^0 \in L^2(0,1)^n$ such that the corresponding solution to \eqref{syst init} satisfies $y(t^0+T,\cdot) \neq 0$, whatever $u \in L^\infty(t^0,t^0+T)^m$ is.

Of course, for particular choices of $M$ and $Q$ one may obtain a better settling time (a trivial example being $M=0$ and $Q=0$).
In the case of time-independent systems, the minimal time in which one can achieve the stabilization and related controllability properties has been recently discussed in \cite{CN19} and \cite{HO19} (see also the references therein).
\end{remark}

\begin{remark}
If the speeds do not depend on time, i.e. $\lambda_{\ell}(t,x)=\lambda_{\ell}(x)$ for every $t \geq 0$, then we have a more explicit formula for the time $\Topt$, namely:
\begin{equation}\label{Topt Tindep}
\Topt=\int_0^1 \frac{1}{-\lambda_{m}(\xi)} \,d\xi +\int_0^1 \frac{1}{\lambda_{m+1}(\xi)} \,d\xi.
\end{equation}
The value \eqref{Topt Tindep} is obtained by integrating over $\xi \in [0,1]$ the differential equation satisfied by the inverse functions $\xi \longmapsto \chi_m^{-1}(\xi;t,1)$ and $\xi \longmapsto \chi_{m+1}^{-1}(\xi;t,0)$.
\end{remark}

\begin{remark}\label{rem hyp 2}
The assumption \eqref{hyp speeds} that the negative (resp. positive) speeds are uniformly bounded from above (resp. below), despite not being necessary for the existence of a solution to \eqref{syst init}, is to be expected for the system \eqref{syst init} to be finite-time stable.
This is an issue that is not specific to systems and that already occurs for a single equation.
Indeed, let us consider for instance the equation with speed $\lambda(t)=-e^{-t}$ (and $t^0=0$ to simplify):
$$
\left\{\begin{array}{l}
\ds \pt{y}(t,x)-e^{-t} \px{y}(t,x)=0,\\
y(t,1)=u(t), \\
y(0,x)=y^0(x).
\end{array}\right.
$$
Then, whatever $y^0 \in L^2(0,1)$ and $u \in L^\infty(0,+\infty)$ are, if $y^0 \neq 0$ in a neighborhood of $1$ we have
$$y(T,\cdot) \neq 0, \quad \forall T>0.$$
This is easily seen thanks to the explicit representation of the solution (obtained by the characteristic method):
$$
y(t,x)=
\left\{\begin{array}{cl}
y^0\left(1-(e^{-t}-x)\right) & \mbox{ if } 0<x<e^{-t}, \\
\ds u\left(\ln \left(\frac{1}{1+e^{-t}-x}\right)\right) & \mbox{ if } e^{-t}<x<1.
\end{array}\right.
$$
\end{remark}

\begin{remark}\label{rem hyp bis}
Contrary to \eqref{hyp speeds}, the assumption \eqref{hyp speeds bis} is mainly technical.
This assumption is needed because we will have to divide in the sequel by the quantities $\lambda_j-\lambda_i$ (see in particular \eqref{def R} below) and we will need this inverse function to be bounded.
However, this condition is clearly not necessary for some systems \eqref{syst init} to be finite-time stable.
Indeed, consider for instance the following $3 \times 3$ system:
\begin{equation}\label{syst rem hyp bis}
\left\{\begin{array}{l}
\ds \pt{y_1}(t,x)-\px{y_1}(t,x)=y_2(t,x), \\
\ds \pt{y_2}(t,x)-\left(1-\frac{e^{-t}}{2}\right)\px{y_2}(t,x)=0, \\
\ds \pt{y_3}(t,x)+\px{y_3}(t,x)=0, \\
y_1(t,1)=u_1(t),
\quad
y_2(t,1)=u_2(t),
\quad y_3(t,0)=y_2(t,0),  \\
y(t^0,x)=y^0(x).
\end{array}\right.
\end{equation}
Then, using the characteristic method it is not difficult to see that the system \eqref{syst rem hyp bis} with $u_1=u_2=0$ is finite-time stable with settling time $T+1$, where $T$ is the unique positive solution to the equation $T+\frac{e^{-T}}{2}=\frac{3}{2}$.

\end{remark}

\begin{remark}\label{rem m=n}
The case $m=n \geq 1$ (no boundary conditions at $x=0$) is easier and does not require the techniques presented in this paper.
Indeed, it can be checked using for instance the constructive method of \cite{LR03, Wan06} that in this case the system \eqref{syst init} with $u=0$ is finite-time stable with settling time equal to $\sup_{t^0 \geq 0} \ssout_m(t^0,1)-t^0$.
\end{remark}

\section{System transformations}\label{sect syst transfo}

The goal of this section is to show that we can use several invertible transformations in order to remove or transform some coupling terms in the initial system \eqref{syst init} and to obtain in the end a system for which we can directly establish that it is finite-time stable with settling time $\Topt$.
The plan of this section is as follows:

\begin{enumerate}[1)]
\item
In Section \ref{sect remov diag}, we use a diagonal transformation to remove the diagonal terms in $M$.

\item
Next, in Section \ref{reduc to triang}, inspired by the seminal works \cite{Col77, Rus78-JMAA, 2002-Balogh-Krstic-EJC} for equations and \cite{CVKB13,HDM15,HDMVK16,HVDMK19} for hyperbolic systems, we use a Volterra transformation of the second kind to transform the system obtained by the previous step into a new system in the so-called ``control normal form'' and with an additional triangular structure for the couplings.

\item
Finally, in Section \ref{sect stabi S2}, inspired by the work \cite{CHO17} for time-independent systems, we use an invertible Fredholm integral transformation to transform the system obtained by the previous step into a new system with a very simple coupling structure that allows us to readily see that it is finite-time stable with settling time $\Topt$.

\end{enumerate}

In Section \ref{sect syst transfo} only the properties of the transformations and new systems are discussed.
The existence of the transformations is the main technical point of this paper and will be proved in Section \ref{sect exist kern} below for the sake of the presentation.

Finally, because of the nature of the transformations that we will use in the sequel, we are led to consider a class of systems that is slightly more general than \eqref{syst init}.
All the systems of this paper will have the following form:
\begin{equation}\label{system general}
\left\{\begin{array}{l}
\ds \pt{y}(t,x)+\Lambda(t,x) \px{y}(t,x)=M(t,x) y(t,x)+G(t,x)y(t,0), \\
y_-(t,1)=\ds \int_0^1 F(t,\xi) y(t,\xi) \, d\xi, \quad y_+(t,0)=Q(t)y_-(t,0),  \\
y(t^0,x)=y^0(x),
\end{array}\right.
\end{equation}
where $M$ and $Q$ will have at least the regularity \eqref{reg parameters}, $F \in L^\infty((0,+\infty) \times (0,1))^{m\times n}$ and
$$G \in C^0([0,+\infty)\times[0,1])^{n \times n} \cap L^{\infty}((0,+\infty)\times(0,1))^{n \times n}.$$
Therefore, \eqref{system general} is similar to \eqref{syst init} but has the extra term with $G$.
In what follows, we will also refer to a system of the form \eqref{system general} as
$$(M,G,F,Q).$$

Hyperbolic equations similar to $(0,G,F,Q)$ were called in ``control normal form'' in the pioneering work \cite[p. 212]{Rus78-JMAA} for the similarity with the finite-dimensional setting (see also the earlier paper \cite{Bru70}).

\subsection{Removal of the diagonal terms}\label{sect remov diag}

In this section we just perform a simple preliminary transformation in order to remove the diagonal terms in $M$.
This is only a technical step, which is nevertheless necessary in view of the existence of the transformation that we will use in the next section, see Remark \ref{rem CN solv kern} below.
This step is sometimes called ``exponential pre-transformation'' in the case of time-independent systems (see Remark \ref{rem exp pre} below).
More precisely, the goal of this section is to establish the following result:

\begin{proposition}\label{prop transfo 1}
There exists $M^1={(m^1_{ij})}_{1 \leq i,j \leq n} \in C^0([0,+\infty)\times[0,1])^{n \times n} \cap L^{\infty}((0,+\infty)\times (0,1))^{n \times n}$ with diagonal terms equal to zero:
\begin{equation}\label{diag zero}
m^1_{ii}=0, \quad \forall i \in \ens{1,\ldots,n},
\end{equation}
and there exists $Q^1 \in C^0([0,+\infty))^{p \times m} \cap L^{\infty}(0,+\infty)^{p \times m}$ such that, for every $F^1 \in L^{\infty}((0,+\infty)\times(0,1))^{m \times n}$, there exists $F \in L^{\infty}((0,+\infty)\times(0,1))^{m \times n}$ such that the following property holds for every $T>0$:
\begin{multline}\label{stab 10}
(M^1,0,F^1,Q^1) \text{ is finite-time stable with settling time } T
\\
\Longrightarrow \quad (M,0,F,Q) \text{ is finite-time stable with settling time } T.
\end{multline}
\end{proposition}

\subsubsection{Formal computations}

To prove Proposition \ref{prop transfo 1}, the idea is to show that, for every $F^1$, there exists $F$ such that we can transform a solution of $(M,0,F,Q)$ into a solution of $(M^1,0,F^1,Q^1)$.
Let then $y$ be the solution to the system $(M,0,F,Q)$ with state-feedback gain function $F$ to be determined below and initial data $y^0$.
Let $\Phi:[0,+\infty)\times[0,1] \longrightarrow \R^{n \times n}$ be a smooth matrix-valued function and set
\begin{equation}\label{transfo diag}
w(t,x)=\Phi(t,x) y(t,x).
\end{equation}
Let us now perform some formal computations in order to see what $w$ can solve.
Using the equation satisfied by $y$, we have
$$
\pt{w}+\Lambda \px{w}=
\left(\pt{\Phi}+\Phi M
+\Lambda\px{\Phi}\right)y+(-\Phi\Lambda+\Lambda \Phi)\px{y}.
$$
On the other hand, using the boundary condition satisfied by $y$ at $x=0$, we have
\begin{multline*}
w_+(t,0)-Q^1(t)w_-(t,0)
\\
=
\left(\Phi_{+-}(t,0)+\Phi_{++}(t,0)Q(t)-Q^1(t)\Phi_{--}(t,0)-Q^1(t)\Phi_{-+}(t,0)Q(t)\right)y_-(t,0).
\end{multline*}
Finally, at $x=1$, we have
$$
w_-(t,1)-\int_0^1 F^1(t,\xi) w(t,\xi) \, d\xi
=
\int_0^1 \left(\Phi_{--}(t,1)F(t,\xi)-F^1(t,\xi)\Phi(t,\xi)\right)y(t,\xi) \, d\xi + \Phi_{-+}(t,1)y_+(t,1).
$$
Thus, we see that $w$ satisfies at $x=1$ the boundary condition $w_-(t,1)=\int_0^1 F^1(t,\xi) w(t,\xi) \, d\xi$ if $\Phi_{-+}(t,1)=0$ and
\begin{equation}\label{def F}
F(t,\xi)=\Phi_{--}(t,1)^{-1}F^1(t,\xi)\Phi(t,\xi),
\end{equation}
provided that $\Phi_{--}(t,1)$ is also invertible.
Moreover, note that $F$ belongs to $L^{\infty}((0,+\infty)\times (0,1))^{m \times n}$ provided that $F^1$ belongs to this space as well and
\begin{equation}\label{min Phi}
\exists C>0,
\quad
\norm{\Phi_{--}(\cdot,1)^{-1}}_{L^{\infty}(0,+\infty)^{m \times m}} \leq C.
\end{equation}

In summary, $w$ defined by \eqref{transfo diag} is the solution of $(M^1,0,F^1,Q^1)$ with state-feedback gain function $F^1$ (which is assumed to be known) and initial data $w^0(\cdot)=\Phi(0,\cdot)y^0(\cdot)$ if we have the following four properties:
\begin{enumerate}[(i)]
\item\label{Phi comm}
$\Lambda(t,x)\Phi(t,x)=\Phi(t,x)\Lambda(t,x)$ for every $t \geq 0$ and $x \in [0,1]$.

\item\label{Phi inv}
The matrices $\Phi(t,x)$ and $\Phi_{--}(t,0)+\Phi_{-+}(t,0)Q(t)$ are invertible for every $t \geq 0$ and $x \in [0,1]$.

\item\label{Phi feed}
$\Phi_{-+}(t,1)=0$ for every $t \geq 0$ (it then follows with \ref{Phi inv} that $\Phi_{--}(t,1)$ is invertible).

\item
$M^1$ and $Q^1$ are defined by
\begin{equation}\label{def M1}
\left\{\begin{array}{l}
\ds M^1(t,x)=\left(\pt{\Phi}(t,x)+\Lambda(t,x)\px{\Phi}(t,x)+\Phi(t,x)M(t,x)\right)\Phi(t,x)^{-1}, \\
\ds Q^1(t)=\left(\Phi_{+-}(t,0)+\Phi_{++}(t,0)Q(t)\right)\left(\Phi_{--}(t,0)+\Phi_{-+}(t,0)Q(t)\right)^{-1}.
\end{array}\right.
\end{equation}
\end{enumerate}

Finally, it is not difficult to check that the stability property \eqref{stab 10} is indeed satisfied since the state-feedback gain function $F$ is solely determined by the state-feedback gain function $F^1$ and, at every fixed $t \geq 0$, the transformation \eqref{transfo diag} defines an injective (in fact, invertible) map of $L^2(0,1)^n$.

\subsubsection{Existence of the transformation}

Let us now prove the existence of a function $\Phi$ with the properties listed above and which in addition ensures that the condition \eqref{diag zero} on $M^1$ holds.

\begin{proposition}\label{exist transfophi}
There exists $\Phi$ with $\Phi, \pt{\Phi}+\Lambda\px{\Phi} \in C^0([0,+\infty)\times[0,1])^{n \times n} \cap L^{\infty}((0,+\infty)\times (0,1))^{n \times n}$ such that the properties \ref{Phi comm}, \ref{Phi inv}, \ref{Phi feed} and \eqref{min Phi} are satisfied and such that the matrix-valued function $M^1$ defined in \eqref{def M1} satisfies \eqref{diag zero}.
\end{proposition}

\begin{proof}

Let $\Phi$ be the diagonal matrix-valued function defined for every $t \geq 0$ and $x \in [0,1]$ by
$$\Phi(t,x)=\diag(\phi_1(t,x),\ldots,\phi_n(t,x)),$$
where, for every $i \in \ens{1,\ldots,n}$,
\begin{equation}\label{sol equ phi_i}
\phi_i(t,x)=e^{-\int_{\ssin_i(t,x)}^t m_{ii}\left(\sigma,\chi_i(\sigma;t,x)\right) \, d\sigma},
\end{equation}
where $m_{ii}$ is extended to negative times by an arbitrary function that keeps the regularity \eqref{reg parameters}.
Clearly, $\phi_i\in C^0([0,+\infty)\times[0,1])$ and it follows from \eqref{bounds sin sout} that $\phi_i \in L^{\infty}((0,+\infty)\times (0,1))$.

It is clear that the first property \ref{Phi comm} holds since $\Lambda$ and $\Phi$ are both diagonal matrices.
Since $\Phi_{-+}=0$, the third property \ref{Phi feed} is automatically satisfied.
It also follows that, to check the second property \ref{Phi inv}, we only need to show that $\Phi(t,x)$ is invertible, which readily follows from the explicit expression of $\phi_i$.
The estimate \eqref{min Phi} is obviously true since $\Phi_{--}(t,1)=\Id_{\R^{m \times m}}$ (recall \eqref{sin is zero}).
Finally, $M^1$ defined in \eqref{def M1} satisfies \eqref{diag zero} since $\phi_i$ satisfies the following linear hyperbolic equation:
$$\pt{\phi_i}(t,x)+\lambda_i(t,x)\px{\phi_i}(t,x)+m_{ii}(t,x)\phi_i(t,x)=0.$$

\end{proof}

\begin{remark}\label{rem exp pre}
There are obviously other possible choices for $\phi_i$, for instance in the time-independent case we can take the slightly simpler function $\phi_i(t,x)=e^{-\int_0^x \frac{m_{ii}(\xi)}{\lambda_i(\xi)} \, d\xi}$ (which coincides with \eqref{sol equ phi_i} only for $i \in \ens{m+1,\ldots,n}$).
\end{remark}

\subsection{Volterra transformation}\label{reduc to triang}

In this section we perform a second transformation to remove some coupling terms of the system.
The system will then have a triangular coupling structure, which is the key point to show later on (Section \ref{sect stabi S2} below) that this system is finite-time stable with settling time $\Topt$.
More precisely, the goal of this section is to establish the following result:

\begin{proposition}\label{prop transfo 2}
There exists a strictly lower triangular matrix $G^2_{--}={(g^2_{ij})}_{1 \leq i,j \leq m} \in C^0([0,+\infty)\times[0,1])^{m \times m} \cap L^\infty((0,+\infty)\times(0,1))^{m \times m}$:
\begin{equation}\label{Gamma cond zero}
g^2_{ij}=0, \quad \forall \, 1 \leq i \leq j \leq m,
\end{equation}
and there exists $G^2_{+-} \in C^0([0,+\infty)\times[0,1])^{p \times m} \cap L^\infty((0,+\infty)\times(0,1))^{p \times m}$ such that, for every $F^2 \in L^{\infty}((0,+\infty)\times(0,1))^{m \times n}$, there exists $F^1 \in L^{\infty}((0,+\infty)\times(0,1))^{m \times n}$ such that the following property holds for every $T>0$:
\begin{multline}\label{stab 11}
(0,G^2,F^2,Q^1) \text{ is finite-time stable with settling time } T
\\
\Longrightarrow \quad (M^1,0,F^1,Q^1) \text{ is finite-time stable with settling time } T,
\end{multline}
where
\begin{equation}\label{syst S2}
G^2=\begin{pmatrix} G^2_{--} & 0 \\ G^2_{+-} & 0 \end{pmatrix}.
\end{equation}
\end{proposition}

\begin{remark}\label{rem S2 stable}
Thanks to the triangular structure \eqref{Gamma cond zero} and \eqref{syst S2} of $G^2$, we can check from the very definition of broad solution (see Definition \ref{def broad sol}) that the system provided by Proposition \ref{prop transfo 2} with state-feedback gain function equal to zero, i.e. $(0,G^2,0,Q^1)$, is finite-time stable with settling time $T(0,G^2,0,Q^1)$ defined by
$$T(0,G^2,0,Q^1)=\sup_{t^0 \geq 0} \, \ssout_{m+1}\left(T_m(t^0),0\right)-t^0,$$
where
$$
\left\{\begin{array}{l}
T_1(t^0)=\ssout_1(t^0,1), \\
T_i(t^0)=\ssout_i(T_{i-1}(t^0),1), \quad \forall i \in \ens{2,\ldots,m}.
\end{array}\right.
$$
We do not detail this point here because it is not needed, and we refer to the arguments used in the proof of Proposition \ref{final targ syst stab} below for an idea of the proof of this assertion.
As a result, the combination of Proposition \ref{prop transfo 2} with Proposition \ref{prop transfo 1} already shows that our initial system $(M,0,F,Q)$ is finite-time stable for some $F$, with settling time $T(0,G^2,0,Q^1)$.
However, this time $T(0,G^2,0,Q^1)$ is always strictly larger than the time $\Topt$ given in Theorem \ref{main thm} (as long as $m>1$).
In the case of time-independent systems, the time $T(0,G^2,0,Q^1)$ is the time obtained in \cite{HDMVK16, HVDMK19} and it has the more explicit expression
$$T(0,G^2,0,Q^1)=\sum_{i=1}^m \int_0^1 \frac{1}{-\lambda_i(\xi)} \, d\xi + \int_0^1 \frac{1}{\lambda_{m+1}(\xi)} \, d\xi.$$
\end{remark}

\subsubsection{Formal computations}

Let us now show how to establish Proposition \ref{prop transfo 2}.
As before, the goal is to show that, for every $F^2$, there exists $F^1$ such that we can transform a solution of $(M^1,0,F^1,Q^1)$ into a solution of $(0,G^2,F^2,Q^1)$.
Let then $w$ be the solution to the system $(M^1,0,F^1,Q^1)$ with state-feedback gain function $F^1$ to be determined below and initial data $w^0$.
Inspired by the works mentioned at the beginning of Section \ref{sect syst transfo}, we use a Volterra transformation of the second kind as follows:
\begin{equation}\label{transfo volt}
\gamma(t,x)=w(t,x)-\int_0^x K(t,x,\xi)w(t,\xi) \, d\xi,
\end{equation}
where we suppose for the moment that the kernel $K$ is smooth on $\clos{\Tau}$, where $\Tau$ is the infinite triangular prism defined by
$$\Tau=\ens{(t,x,\xi) \in (0,+\infty)\times(0,1)\times(0,1), \qquad \xi<x}.$$
Let us now perform some formal computations to see what $\gamma$ can solve.
We have
\begin{multline*}
\pt{\gamma}(t,x)+\Lambda(t,x) \px{\gamma}(t,x)=
\pt{w}(t,x)+\Lambda(t,x) \px{w}(t,x)
\\
-\int_0^x \pt{K}(t,x,\xi)w(t,\xi) \, d\xi
-\int_0^x K(t,x,\xi)\pt{w}(t,\xi) \, d\xi
\\
-\Lambda(t,x)K(t,x,x)w(t,x)-\Lambda(t,x)\int_0^x \px{K}(t,x,\xi)w(t,\xi) \, d\xi.
\end{multline*}
Using the equation satisfied by $w$, we obtain
\begin{multline*}
\pt{\gamma}(t,x)+\Lambda(t,x) \px{\gamma}(t,x)=
M^1(t,x)w(t,x)
\\
-\int_0^x \pt{K}(t,x,\xi)w(t,\xi) \, d\xi
-\int_0^x K(t,x,\xi)\left(-\Lambda(t,\xi)\pxi{w}(t,\xi)+M^1(t,\xi)w(t,\xi)\right) \, d\xi
\\
-\Lambda(t,x)K(t,x,x)w(t,x)-\Lambda(t,x)\int_0^x \px{K}(t,x,\xi)w(t,\xi) \, d\xi.
\end{multline*}
Integrating by parts the third term of the right hand side and using the boundary condition $w_+(t,0)=Q^1(t)w_-(t,0)$, we finally obtain

\begin{multline*}
\pt{\gamma}(t,x)+\Lambda(t,x) \px{\gamma}(t,x)
=\int_0^x \Bigg(-\pt{K}(t,x,\xi)
-\pxi{K}(t,x,\xi)\Lambda(t,\xi)
-K(t,x,\xi)\pxi{\Lambda}(t,\xi)
\\
-K(t,x,\xi)M^1(t,\xi)
-\Lambda(t,x)\px{K}(t,x,\xi)
\Bigg)w(t,\xi) \, d\xi
\\
+\left(M^1(t,x)+K(t,x,x)\Lambda(t,x)-\Lambda(t,x)K(t,x,x)\right)w(t,x)
-K(t,x,0)\Lambda(t,0)\begin{pmatrix} \Id_{\R^{m \times m}} \\ Q^1(t) \end{pmatrix} w_-(t,0).
\end{multline*}

\noindent
On the other hand, since $\gamma(t,0)=w(t,0)$, $\gamma$ satisfies the same boundary condition as $w$ at $x=0$:
$$
\gamma_+(t,0)-Q^1(t)\gamma_-(t,0)
=w_+(t,0)-Q^1(t)w_-(t,0)
=0.
$$
Finally, at $x=1$, we have
\begin{multline*}
\gamma_-(t,1)-\int_0^1 F^2(t,\xi) \gamma(t,\xi) \, d\xi
=
\\
\int_0^1 \left(F^1(t,\xi)-K_-(t,1,\xi)-F^2(t,\xi)+\int_\xi^1 F^2(t,\zeta) K(t,\zeta,\xi) \, d\zeta\right) w(t,\xi) \, d\xi,
\end{multline*}
where $K_-$ denotes the $m \times n$ sub-matrix of $K$ formed by its first $m$ rows.
Thus, we see that $\gamma$ satisfies at $x=1$ the boundary condition $\gamma_-(t,1)=\int_0^1 F^2(t,\xi) \gamma(t,\xi) \, d\xi$ if we take
\begin{equation}\label{def F1}
F^1(t,\xi)=K_-(t,1,\xi)+F^2(t,\xi)-\int_\xi^1 F^2(t,\zeta) K(t,\zeta,\xi) \, d\zeta.
\end{equation}
Note that $F^1$ belongs to $L^{\infty}((0,+\infty)\times(0,1))^{m \times n}$ provided that $F^2$ belongs to this space as well and
$$K \in L^\infty(\Tau)^{n \times n}.$$

In summary, $\gamma$ defined by \eqref{transfo volt} is the solution to $(0,G^2,F^2,Q^1)$ with initial data $\gamma^0(x)=w^0(x)-\int_0^x K(t^0,x,\xi)w^0(\xi) \,d\xi$ if we have the following two properties:

\begin{enumerate}[(i)]
\item\label{equ K volt}
For every $(t,x,\xi) \in \Tau$,
\begin{equation}\label{kern equ}
\left\{\begin{array}{l}
\ds \pt{K}(t,x,\xi)
+\Lambda(t,x)\px{K}(t,x,\xi)
+\pxi{K}(t,x,\xi)\Lambda(t,\xi)
\\
\ds \qquad \qquad \qquad \qquad \qquad \qquad \qquad \qquad \qquad \qquad +K(t,x,\xi)\left(\pxi{\Lambda}(t,\xi)+M^1(t,\xi)\right)=0,
\\
\ds K(t,x,x)\Lambda(t,x)-\Lambda(t,x)K(t,x,x)=-M^1(t,x).
\end{array}\right.
\end{equation}

\item
$G^2$ is defined by
$$
G^2=
\begin{pmatrix}
G^2_{--} & 0 \\
G^2_{+-} & 0
\end{pmatrix},
$$
with
\begin{equation}\label{def G2 2}
\left\{\begin{array}{l}
G^2_{--}(t,x)=-K_{--}(t,x,0)\Lambda_{--}(t,0)-K_{-+}(t,x,0)\Lambda_{++}(t,0)Q^1(t), \\
G^2_{+-}(t,x)=-K_{+-}(t,x,0)\Lambda_{--}(t,0)-K_{++}(t,x,0)\Lambda_{++}(t,0)Q^1(t).
\end{array}\right.
\end{equation}

\end{enumerate}

Finally, the stability property \eqref{stab 11} is clearly satisfied since, at every fixed $t \geq 0$, the Volterra transformation \eqref{transfo volt} defines an injective map of $L^2(0,1)^n$ (see e.g. \cite[Theorem 2.6]{Hor73}).

\subsubsection{The kernel equations}

We can prove that there exists $K \in C^0(\clos{\Tau})^{n \times n} \cap L^\infty(\Tau)^{n \times n}$ that satisfies the so-called ``kernel equations'' \eqref{kern equ} in the sense of broad solutions.
However, it is in general not enough to deduce stability results for the initial system $(M,0,F,Q)$ since the investigation of the stability properties of the system $(0,G^2,F^2,Q^1)$ is not an easier task without knowing any more information about it.
The breakthrough idea of the conference paper \cite{HDM15} in the time-independent case (see also \cite{HDMVK16, HVDMK19})  was to construct a solution $K$ to the kernel equations which, in addition, yields a simpler structure for the matrix $G^2_{--}$ defined in \eqref{def G2 2}.
This is the key point to prove stability results for the system $(0,G^2,F^2,Q^1)$ (see Remark \ref{rem S2 stable} and Section \ref{sect stabi S2} below).
Such a construction is possible by adding some conditions for $K_{--}$ at $(t,x,0)$ (see \eqref{def G2 2}) but the price to pay is that it introduces discontinuities for $K_{--}$, so that $K$ will not be globally $C^0$ anymore but only piecewise $C^0$ in general.
We will prove the following result:

\begin{theorem}\label{thm exist kern}
There exists a $n \times n$ matrix-valued function $K={(k_{ij})}_{1 \leq i,j \leq n}$ such that:
\begin{enumerate}[(i)]
\item
$K \in L^\infty(\Tau)^{n \times n}$.

\item\label{smooth part}
For every $i,j \in \ens{1,\ldots,n}$ with $j \not\in \ens{i+1,\ldots,m}$, we have $k_{ij} \in C^0(\clos{\Tau})$.

\item\label{reg K disc}
For every $i,j \in \ens{1,\ldots,m}$ with $i<j$, we have $k_{ij} \in C^0(\clos{\Tau_{ij}^-}) \cap C^0(\clos{\Tau_{ij}^+})$, where (see Figure \ref{figure2D})
\begin{align*}
& \Tau^-_{ij}=\ens{(t,x,\xi) \in \Tau, \quad \xi<\psi_{ij}(t,x)}, \\
& \Tau^+_{ij}=\ens{(t,x,\xi) \in \Tau, \quad \xi>\psi_{ij}(t,x)},
\end{align*}
where $\psi_{ij} \in C^1([0,+\infty)\times[0,1])$ satisfies the following semi-linear hyperbolic equation for every $t \geq 0$ and $x \in [0,1]$:
\begin{equation}\label{equ psi}
\left\{\begin{array}{l}
\ds \pt{\psi_{ij}}(t,x)+\lambda_i(t,x)\px{\psi_{ij}}(t,x)-\lambda_j(t,\psi_{ij}(t,x))=0, \\
\psi_{ij}(t,0)=0.
\end{array}\right.
\end{equation}

\item
$K$ is a broad solution of \eqref{kern equ} in $\Tau$ (the exact meaning of this statement will be detailed during the proof of the theorem, in Section \ref{sect exist kern T2} below).

\item\label{Gm triang}
For every $t \geq 0$ and $x \in [0,1]$, the matrix $G^2_{--}(t,x)$ defined in \eqref{def G2 2} is strictly lower triangular, i.e. it satisfies \eqref{Gamma cond zero} (it then follows from \ref{smooth part} that $G^2_{--} \in C^0([0,+\infty)\times[0,1])^{m \times m}$).
\end{enumerate}

\end{theorem}

The proof of Theorem \ref{thm exist kern} is one of the main technical difficulties of this article and it is postponed to Section \ref{sect exist kern T2} below for the sake of the presentation.
We conclude this section with some important remarks.

\begin{remark}\label{rem CN solv kern}
Let us rewrite the second condition of \eqref{kern equ} component-wise:
\begin{equation}\label{equ remov diag}
\left(\lambda_j(t,x)-\lambda_i(t,x)\right)k_{ij}(t,x,x)=-m^1_{ij}(t,x).
\end{equation}
Therefore, we see that for $i=j$ we shall necessarily have $m^1_{ii}=0$ and it explains why we had to perform a preliminary transformation in Section \ref{sect remov diag} to remove these terms (otherwise the equation \eqref{equ remov diag}, and thus the kernel equations \eqref{kern equ}, have no solution).
\end{remark}

\begin{remark}\label{rem G2+- too}
It is in general not possible to solve \eqref{kern equ} with $G^2_{--}=0$, unless $m=1$.
\end{remark}

\begin{remark}\label{rem broad reg int term}
Observe that, with the regularity stated in Theorem \ref{thm exist kern}, we have in particular that, for every $w \in C^0([t^0,+\infty);L^2(0,1)^n)$, $t^0 \geq 0$,
$$
(t,x) \mapsto \int_0^x K(t,x,\xi)w(t,\xi) \,d\xi
\in C^0([t^0,+\infty)\times[0,1])^n.
$$
This follows from Lebesgue's dominated convergence theorem.
This shows that $\gamma$ defined by \eqref{transfo volt} has the good regularity to be a broad solution (see Definition \ref{def broad sol}), if so has $w$.
\end{remark}

\begin{remark}\label{rem discont}
Observe that the condition \ref{reg K disc} shows that the kernel has possible discontinuities on $\xi=\psi_{ij}(t,x)$ for $i<j \leq m$.
Besides, these discontinuities also depend on the component of the kernel that we consider.
The appearance of such discontinuities is explained by the requirement of the last condition \ref{Gm triang} because we somehow force two boundary conditions at the points $(t,0,0)$, one by the condition already required in \eqref{kern equ} (which concerns $i \neq j$, see Remark \ref{rem CN solv kern}) and another one by \ref{Gm triang} (which only concerns $i \leq j \leq m$).
This results in discontinuities along the characteristics passing through these points.
Note that this also complicates the justification of formal computations that we performed above since regularity problems will occur during the computation of the following term (when $i<j \leq m$):
$$
\pt{}\left(\int_0^x k_{ij}(t,x,\xi)w_j(t,\xi)\, d\xi\right)
+\lambda_i(t,x)\px{}\left(\int_0^x k_{ij}(t,x,\xi)w_j(t,\xi) \, d\xi\right).
$$
More precisely, writing $\int_0^x=\int_0^{\psi_{ij}(t,x)}+\int_{\psi_{ij}(t,x)}^x$ and using integration by parts, we see that the following jump terms notably appear:
\begin{multline*}
\left(\pt{\psi_{ij}}(t,x)-\lambda_j(t,\psi_{ij}(t,x))+\lambda_i(t,x)\px{\psi_{ij}}(t,x)\right)
\\
\times \left(k_{ij}^-(t,x,\psi_{ij}(t,x))-k_{ij}^+(t,x,\psi_{ij}(t,x))\right)w_j(t,\psi_{ij}(t,x)),
\end{multline*}
where $k_{ij}^-$ (resp. $k_{ij}^+$) denotes the trace of the restriction of $k_{ij}$ to $\partial\Tau_{ij}^-$ (resp $\partial\Tau_{ij}^+$).
This is why it is crucial to precise that $\psi_{ij}$ solves the first equation in \eqref{equ psi} so that such undesired terms vanish in the end.
In the case of time-independent systems, we have in fact
$$\psi_{ij}(t,x)=\phi_j^{-1}\left(\phi_i(x)\right),$$
where we introduced $\phi_\ell(x)=\int_0^x \frac{1}{-\lambda_\ell(\xi)} \, d\xi$ for $\ell \in \ens{1,\ldots,m}$ ($\psi_{ij}$ is well defined because $i<j$).
This is the same function as in \cite[(A.1)]{HVDMK19}.
\end{remark}

\begin{center}
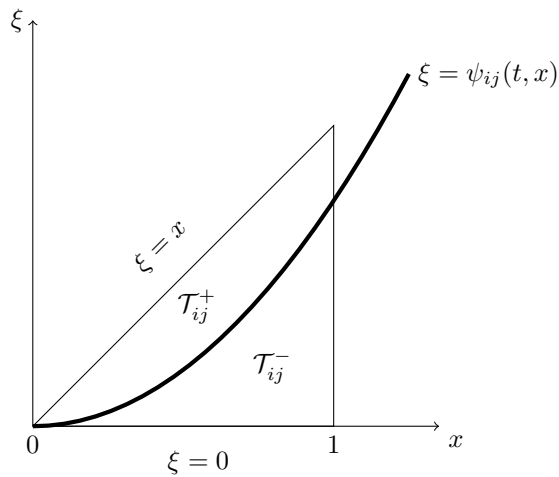

\scalebox{1}{
\begin{tikzpicture}[scale=4]
\draw [<->] (0,1.35) node [left]  {$\xi$} -- (0,0) -- (1.35,0) node [below right] {$x$};
\draw (0,0) node [left, below] {$0$};
\draw (1,0) node [below] {$1$};
\draw (0,0) -- (1,0) -- (1,1) -- (0,0);
\draw[ultra thick, domain=0:1.25] plot (\x,{3*\x^2/4});
\draw (1.25,1.17) node [right] {$\xi=\psi_{ij}(t,x)$};
\draw (0.45,0.4) node [right] {$\Tau_{ij}^+$};
\draw (0.7,0.2) node [right] {$\Tau_{ij}^-$};
\draw (0.55,-0.05) node [below] {$\xi=0$};
\draw (0.425,0.6) node[rotate=45] {$\xi=x$};
      \end{tikzpicture}
      }
\captionof{figure}{2D cross-section of the domain $\Tau$ at a fixed $t$}\label{figure2D}
\end{center}

\subsection{Fredholm integral transformation}\label{sect stabi S2}

We recall that at the moment we already know that the system $(0,G^2,F^2,Q^1)$ of Proposition \ref{prop transfo 2} is finite-time stable if we take $F^2=0$, but only with a settling time which is strictly larger than $\Topt$ (unless $m=1$), see Remark \ref{rem S2 stable}.
In this section, we perform a third and last transformation to remove the coupling term $G^2_{--}$ in the system $(0,G^2,F^2,Q^1)$ and we show that the resulting system has the desired stability properties.
More precisely, the goal of this section is to establish the two following results:

\begin{proposition}\label{prop transfo 3}

There exists $F^2 \in L^{\infty}((0,+\infty)\times(0,1))^{m \times n}$ such that the following property holds for every $T>0$:
\begin{multline}\label{stab 12}
(0,G^3,0,Q^1) \text{ is finite-time stable with settling time } T
\\
\Longrightarrow \quad (0,G^2,F^2,Q^1) \text{ is finite-time stable with settling time } T,
\end{multline}
where
\begin{equation}\label{def G3}
G^3=\begin{pmatrix} 0 & 0 \\ G^2_{+-} & 0 \end{pmatrix}.
\end{equation}
\end{proposition}

\begin{proposition}\label{final targ syst stab}
The system $(0,G^3,0,Q^1)$ is finite-time stable with settling time $\Topt$ defined by \eqref{def Topt}.
\end{proposition}

Note that the proof of our main result -- Theorem \ref{main thm} -- will then be complete   (recall Propositions \ref{prop transfo 1} and \ref{prop transfo 2}), except for the $\tau$-periodicity statement which will be studied later on in Section~\ref{sec.per}.

\subsubsection{Finite-time stability of the system $(0,G^3,0,Q^1)$}
In this section we prove Proposition \ref{final targ syst stab} in four steps.

\begin{proof}[\nopunct]
\begin{enumerate}[1)]
\item
Let $t^0 \geq 0$ be fixed.
From the very definition of broad solution (see Definition \ref{def broad sol}) and the simple structure of $G^3$, we see that the first $m$ components of the system vanish at time $t^0+\Topt$ if (recall that the feedback is equal to zero)
\begin{equation}\label{m first compo zero}
\ssin_i(t^0+\Topt,x)>t^0, \quad \forall x \in [0,1], \quad \forall i \in \ens{1,\ldots,m},
\end{equation}
and the remaining $p$ components of the system vanish at time $t^0+\Topt$ if
\begin{equation}\label{p last compo zero}
\left\{\begin{array}{l}
\ssin_i(t^0+\Topt,x)>t^0, \quad \forall x \in [0,1], \quad \forall i \in \ens{m+1,\ldots,n}
\\
\ssin_j\left(\ssin_i(t^0+\Topt,x),0\right)>t^0, \quad \forall x \in [0,1), \quad \forall j \in \ens{1,\ldots,m},
\quad \forall i \in \ens{m+1,\ldots,n}.
\end{array}\right.
\end{equation}

\item
First of all, observe that, from \eqref{increases}, \eqref{invariance} and \eqref{sin is zero} we have the following inverse formula for every $t,\bar{t} \in \R$:
\begin{equation}\label{inv sin}
\left\{\begin{array}{l}
\ssin_i(t,0)>\bar{t} \quad \Longleftrightarrow \quad t>\ssout_i(\bar{t},1), \qquad \mbox{ if } i \in \ens{1,\ldots,m}, \\
\ssin_i(t,1) \geq \bar{t} \quad \Longleftrightarrow \quad t \geq \ssout_i(\bar{t},0), \qquad  \mbox{ if } i \in \ens{m+1,\ldots,n}.
\end{array}\right.
\end{equation}

\item
Let us establish \eqref{m first compo zero}.
Let then $i \in \ens{1,\ldots,m}$ be fixed.
We have:
$$
\begin{array}{rcl}
\ds \ssin_i(t^0+\Topt,x)>t^0, \quad \forall x \in [0,1] &\ds \Longleftrightarrow &\ds \ssin_i(t^0+\Topt,0)>t^0, \quad \mbox{ (by \eqref{increases}),} \\
    &\ds \Longleftrightarrow &\ds t^0+\Topt>\ssout_i\left(t^0,1\right), \quad \mbox{ (by \eqref{inv sin}),}
\end{array}
$$
and this last statement holds true since, by definition of $\Topt$ and \eqref{tentredeux}-\eqref{sin is zero}, we have, for an arbitrary $j \in \ens{m+1,\ldots,n}$,
$$t^0+\Topt \geq \ssout_j(\ssout_i(t^0,1),0)>\ssout_i(t^0,1).$$

\item
Let us now establish \eqref{p last compo zero}.
We focus on the second inequality since the first one is obtained similarly to \eqref{m first compo zero}.
Let then $i \in \ens{m+1,\ldots,n}$ and $j \in \ens{1,\ldots,m}$ be fixed.
We have:
\begin{multline*}
\ssin_j\left(\ssin_i(t^0+\Topt,x),0\right)>t^0, \quad \forall x \in [0,1)
\\
\begin{array}{cl}
\ds \Longleftrightarrow &\ds \ssin_i(t^0+\Topt,x)>\ssout_j\left(t^0,1\right), \quad \forall x \in [0,1), \quad \mbox{ (by \eqref{inv sin}),}
\\
\ds \Longleftrightarrow &\ds \ssin_i(t^0+\Topt,1) \geq \ssout_j\left(t^0,1\right), \quad \mbox{ (by \eqref{increases}),}
\\
\ds \Longleftrightarrow &\ds t^0+\Topt \geq \ssout_i\left(\ssout_j\left(t^0,1\right),0\right), \quad \mbox{ (by \eqref{inv sin}),}
\end{array}
\end{multline*}
and this last statement holds true by definition of $\Topt$.

\end{enumerate}

\end{proof}

\subsubsection{Proof of Proposition~\ref{prop transfo 3}}\label{sect form compu}
\begin{proof}[\nopunct]
We start the proof with some computations.
We will show that we can transform a solution of $(0,G^3,0,Q^1)$ into a solution of $(0,G^2,F^2,Q^1)$ (note the difference in the order of the transformation with respect to the previous sections and see Remark \ref{rem ordre transfo} below for the reason).
Let then $z$ be the solution to the system $(0,G^3,0,Q^1)$ with initial data $z^0$.
Inspired by the work \cite{CHO17} mentioned before (for time-independent systems), we propose to use a Fredholm integral transformation as follows:
\begin{equation}\label{transfo fred}
\gamma(t,x)=z(t,x)-\int_0^1 H(t,x,\xi)z(t,\xi) \, d\xi,
\end{equation}
where we suppose for the moment that the kernel $H$ is smooth on $\clos{\Rect}$, where $\Rect$ is the infinite rectangular prism defined by
$$\Rect=(0,+\infty)\times(0,1)\times(0,1).$$
Let us now perform some formal computations and see what $\gamma$ can solve.
We have
\begin{multline*}
\pt{\gamma}(t,x)+\Lambda(t,x) \px{\gamma}(t,x)=
\pt{z}(t,x)+\Lambda(t,x) \px{z}(t,x)
\\
-\int_0^1 \pt{H}(t,x,\xi)z(t,\xi) \, d\xi
-\int_0^1 H(t,x,\xi)\pt{z}(t,\xi) \, d\xi
-\Lambda(t,x)\int_0^1 \px{H}(t,x,\xi)z(t,\xi) \, d\xi.
\end{multline*}
Using the equation satisfied by $z$, we obtain
\begin{multline*}
\pt{\gamma}(t,x)+\Lambda(t,x) \px{\gamma}(t,x)=
G^3(t,x)z(t,0)
-\int_0^1 \pt{H}(t,x,\xi)z(t,\xi) \, d\xi
\\
-\int_0^1 H(t,x,\xi)\left(-\Lambda(t,\xi)\pxi{z}(t,\xi)+G^3(t,\xi)z(t,0)\right) \, d\xi
-\Lambda(t,x)\int_0^1 \px{H}(t,x,\xi)z(t,\xi) \, d\xi.
\end{multline*}
Integrating by parts the third term of the right hand side, we obtain

\begin{multline*}
\pt{\gamma}(t,x)+\Lambda(t,x) \px{\gamma}(t,x)
=\int_0^1 \Bigg(-\pt{H}(t,x,\xi)
-\pxi{H}(t,x,\xi)\Lambda(t,\xi)
-H(t,x,\xi)\pxi{\Lambda}(t,\xi)
\\
-\Lambda(t,x)\px{H}(t,x,\xi)
\Bigg)z(t,\xi) \, d\xi
\\
+H(t,x,1)\Lambda(t,1)z(t,1)
+\left(G^3(t,x)-H(t,x,0)\Lambda(t,0)
-\int_0^1 H(t,x,\xi)G^3(t,\xi) \, d\xi\right)z(t,0).
\end{multline*}
Using the formula \eqref{transfo fred} with $x=0$ we finally obtain
\begin{multline*}
\pt{\gamma}(t,x)+\Lambda(t,x) \px{\gamma}(t,x)
=
\int_0^1 \Bigg(-\pt{H}(t,x,\xi)
-\pxi{H}(t,x,\xi)\Lambda(t,\xi)
-H(t,x,\xi)\pxi{\Lambda}(t,\xi)
\\
-\Lambda(t,x)\px{H}(t,x,\xi)
+\left(G^3(t,x)-H(t,x,0)\Lambda(t,0)
-\int_0^1 H(t,x,\zeta)G^3(t,\zeta) \, d\zeta\right) H(t,0,\xi)
\Bigg)z(t,\xi) \, d\xi
\\
+H(t,x,1)\Lambda(t,1)z(t,1)
+\left(G^3(t,x)-H(t,x,0)\Lambda(t,0)
-\int_0^1 H(t,x,\xi)G^3(t,\xi) \, d\xi\right)\gamma(t,0).
\end{multline*}
Since
\begin{equation}\label{no feedback}
z_-(t,1)=0,
\end{equation}
the boundary term $H(t,x,1)\Lambda(t,1)z(t,1)$ vanishes if we require that $H$ satisfies
$$H_{-+}(t,x,1)=H_{++}(t,x,1)=0.$$
On the other hand, $\gamma$ and $z$ satisfy the same boundary condition at $x=0$ provided that
$$H(t,0,\xi)=0.$$
Finally, at $x=1$, we have (recall \eqref{no feedback})
$$
\gamma_-(t,1)-\int_0^1 F^2(t,\xi) \gamma(t,\xi) \, d\xi
=
\int_0^1 \left(-H_-(t,1,\xi)-F^2(t,\xi)+\int_0^1 F^2(t,\zeta) H(t,\zeta,\xi) \, d\zeta\right) z(t,\xi) \, d\xi,
$$
where $H_-$ denotes again the $m \times n$ sub-matrix of $H$ formed by its first $m$ rows.
Thus,  we see that $\gamma$ satisfies at $x=1$ the boundary condition $\gamma_-(t,1)=\int_0^1 F^2(t,\xi) \gamma(t,\xi) \, d\xi$ if $F^2(t,\cdot)$ satisfies the following Fredholm integral equation (at $t$ fixed):
\begin{equation}\label{def F2}
F^2(t,\xi)-\int_0^1 F^2(t,\zeta) H(t,\zeta,\xi) \, d\zeta=-H_-(t,1,\xi).
\end{equation}

In summary, $\gamma$ defined by \eqref{transfo fred} is the solution of $(0,G^2,F^2,Q^1)$ with state-feedback gain function $F^2$ satisfying \eqref{def F2} (whenever it exists) and initial data $\gamma^0(x)=z^0(x)-\int_0^1 H(t^0,x,\xi)z^0(\xi) \,d\xi$ if we have the following two properties:

\begin{enumerate}[(i)]
\item\label{equ K fred}
For every $(t,x,\xi) \in \Rect$,
\begin{equation}\label{kern equ T30}
\left\{\begin{array}{l}
\ds \pt{H}(t,x,\xi)
+\Lambda(t,x)\px{H}(t,x,\xi)
+\pxi{H}(t,x,\xi)\Lambda(t,\xi)
+H(t,x,\xi)\pxi{\Lambda}(t,\xi)
=0, \\
\ds H_{-+}(t,x,1)=H_{++}(t,x,1)=H(t,0,\xi)=0.
\end{array}\right.
\end{equation}

\item
$G^3$ satisfies the Fredholm integral equation
\begin{equation}\label{fred equ G3}
G^3(t,x)-\int_0^1 H(t,x,\xi)G^3(t,\xi) \, d\xi=
G^2(t,x)
+H(t,x,0)\Lambda(t,0).
\end{equation}

\end{enumerate}
 Finally, the stability property \eqref{stab 12} is clearly satisfied if, for every $t \geq 0$, the Fredholm transformation \eqref{transfo fred} defines a surjective map of  $L^2(0,1)^n$.

It remains to prove the existence of  $F^2$ and $H$ satisfying the above properties and so that the Fredholm transformation \eqref{transfo fred} is invertible (let us recall that, unlike Volterra transformations of the second kind, Fredholm transformations are not always invertible).
Note that $H=0$ is a solution of \eqref{kern equ T30}.
Taking into account the very particular structure \eqref{syst S2} of $G^2$, this motivates our attempt to look for a kernel $H$ with the following simple structure:
\begin{equation}\label{def ker H}
H=\begin{pmatrix} H_{--} & 0 \\ 0 & 0 \end{pmatrix}.
\end{equation}

\noindent
This structure implies that the Fredholm equation \eqref{fred equ G3} is equivalent to
$$
\left\{\begin{array}{l}
\ds G^3_{--}(t,x)-\int_0^1 H_{--}(t,x,\xi)G^3_{--}(t,\xi) \, d\xi=G^2_{--}(t,x)+H_{--}(t,x,0)\Lambda_{--}(t,0), \\
\ds G^3_{-+}(t,x)-\int_0^1 H_{--}(t,x,\xi)G^3_{-+}(t,\xi) \, d\xi=0, \\
\ds G^3_{+-}(t,x)=G^2_{+-}(t,x), \\
\ds G^3_{++}(t,x)=0. \\
\end{array}\right.
$$
These equations are easily solved by taking $G^3_{--}=G^3_{-+}=0$ (so that $G^3$ is indeed given by \eqref{def G3}) if we impose the following condition for $H_{--}$ at $(t,x,0)$:
$$H_{--}(t,x,0)=-G^2_{--}(t,x)\Lambda_{--}(t,0)^{-1}.$$

\noindent
We point out that this last condition may introduce discontinuities in the kernel, because of possible compatibility conditions at $(t,0,0)$ with the previous requirement that $H_{--}(t,0,\xi)=0$.

Finally, since $G^2_{--}$ is in fact strictly lower triangular \eqref{Gamma cond zero}, we also look for $H_{--}$ with the same structure.
Note that this structure in particular ensures that the Fredholm transformation \eqref{transfo fred} is invertible and that the Fredholm equation \eqref{def F2} always has a unique solution $F^2 \in L^{\infty}((0,+\infty)\times(0,1))^{m \times n}$, provided that $H_{--} \in L^{\infty}(\Rect)^{m \times m}$ (see for instance \cite[Appendix]{CHO17} for more details).
This property was a priori not guaranteed without additional information (we emphasize again that Fredholm transformations are not always invertible).

The final step is to prove the existence of  $H_{--}$ that satisfies all the properties mentioned above.
This is the goal of the following theorem, the proof of which is given in Section \ref{sect exist kern T3} below.

\begin{theorem}\label{thm exist kern T3}
There exists a $m \times m$ matrix-valued function $H_{--}={(h_{ij})}_{1 \leq i,j \leq m}$ such that:
\begin{enumerate}[(i)]
\item
$H_{--} \in L^{\infty}(\Rect)^{m \times m}$.

\item
For $i \leq j$, we have $h_{ij}=0$ (i.e. $H_{--}$ is strictly lower triangular).

\item
For $i>j$, we have $h_{ij} \in C^0(\clos{\Rect_{ij}^-}) \cap C^0(\clos{\Rect_{ij}^+})$, where
\begin{align*}
& \Rect^-_{ij}=\ens{(t,x,\xi) \in \Rect, \quad \xi<\psi_{ij}(t,x)}, \\
& \Rect^+_{ij}=\ens{(t,x,\xi) \in \Rect, \quad \xi>\psi_{ij}(t,x)},
\end{align*}
where $\psi_{ij} \in C^1([0,+\infty)\times[0,1])$ satisfies the semi-linear hyperbolic equation \eqref{equ psi}.

\item
$H_{--}$ is the unique broad solution in $\Rect$ of the system
\begin{equation}\label{kern equ T3}
\left\{\begin{array}{l}
\ds \pt{H_{--}}(t,x,\xi)+\Lambda_{--}(t,x)\px{H_{--}}(t,x,\xi)+\pxi{H_{--}}(t,x,\xi)\Lambda_{--}(t,\xi)
\\
\ds \qquad \qquad \qquad \qquad \qquad \qquad \qquad \qquad \qquad \qquad \qquad \qquad
+H_{--}(t,x,\xi)\pxi{\Lambda_{--}}(t,\xi)=0, \\
\ds H_{--}(t,0,\xi)=0, \\
\ds H_{--}(t,x,0)=-G^2_{--}(t,x)\Lambda_{--}(t,0)^{-1},
\end{array}\right.
\end{equation}
(once again, the exact meaning of this statement will be detailed during the proof of the theorem, in Section \ref{sect exist kern T3} below).

\end{enumerate}

\end{theorem}

This concludes the proof of Proposition~\ref{prop transfo 3}.

\end{proof}

\begin{remark}
Observe once again that the kernel is discontinuous.
This introduces some additional boundary terms along these discontinuities in the formal computations performed above but, as mentioned before in Remark \ref{rem discont}, these terms cancel each other out thanks to the equation satisfied by $\psi_{ij}$ in \eqref{equ psi}.
As in the time-independent case (\cite[Section 3]{CHO17}), the system \eqref{kern equ T3} is easy to solve and its solution is even explicit (see \eqref{kij Fred}-\eqref{explicit bij} below).
\end{remark}

\begin{remark}\label{rem ordre transfo}
If we prefer to use the inverse transformation
\begin{equation}\label{transfo fred inv}
z(t,x)=\gamma(t,x)-\int_0^1 L(t,x,\xi)\gamma(t,\xi) \, d\xi,
\end{equation}
where $L$ has the same structure as $H$ (i.e. only $L_{--}$ is not zero), then the corresponding kernel equations are
$$
\left\{\begin{array}{l}
\ds \pt{L_{--}}(t,x,\xi)+\Lambda_{--}(t,x)\px{L_{--}}(t,x,\xi)+\pxi{L_{--}}(t,x,\xi)\Lambda_{--}(t,\xi)
\\
\ds \qquad \qquad \qquad \qquad \qquad
+L_{--}(t,x,\xi)\pxi{\Lambda_{--}}(t,\xi)
-L_{--}(t,x,1)\Lambda_{--}(t,1)L_{--}(t,1,\xi)=0, \\
\ds L_{--}(t,0,\xi)=0, \\
\ds L_{--}(t,x,0)=\left(G^2_{--}(t,x)-\int_0^1L_{--}(t,x,\xi)G^2_{--}(t,\xi)d\xi\right)\Lambda_{--}(t,0)^{-1}.
\end{array}\right.
$$

We see that these equations are slightly more complicated than \eqref{kern equ T3} since there is a nonlinear and nonlocal term.
This explains why we had a preference for the transformation \eqref{transfo fred} over \eqref{transfo fred inv} but there is no obstruction to work with \eqref{transfo fred inv}.
\end{remark}

\section{Existence of a solution to the kernel equations}\label{sect exist kern}

In this section we prove Theorem \ref{thm exist kern} and Theorem \ref{thm exist kern T3}, which are the two key results for the present article, and we describe in Section \ref{sec.per} how to obtain a time-periodic feedback.
We propose to start with the proof of Theorem \ref{thm exist kern T3} because it is far more simpler (in particular, no fixed-point argument is needed).

\subsection{Kernel for the Fredholm transformation}\label{sect exist kern T3}
In this section we prove Theorem \ref{thm exist kern T3}, that is we prove the existence of a suitably smooth matrix-valued function $H_{--}=(h_{ij})_{1 \leq i,j \leq m}$ which is strictly lower triangular and satisfies \eqref{kern equ T3} (in some sense).

Writing \eqref{kern equ T3} component-wise, this gives
\begin{equation}\label{kern T3 cpw}
\left\{\begin{array}{l}
\ds \pt{h_{ij}}(t,x,\xi)+\lambda_i(t,x)\px{h_{ij}}(t,x,\xi)+\lambda_j(t,\xi)\pxi{h_{ij}}(t,x,\xi)
+\pxi{\lambda_j}(t,\xi) h_{ij}(t,x,\xi)=0, \\
\ds h_{ij}(t,0,\xi)=0, \\
\ds h_{ij}(t,x,0)=-\frac{g^2_{ij}(t,x)}{\lambda_j(t,0)}.
\end{array}\right.
\end{equation}
Since we see that the equations are uncoupled, we can fix the indices $i,j$ for the remainder of Section \ref{sect exist kern T3}:
$$i,j \in \ens{1,\ldots,m} \text{ are fixed.}$$

\subsubsection{The characteristics of \eqref{kern T3 cpw}}

For each $(t,x,\xi) \in \R^3$ fixed, we introduce the characteristic curve $\chi_{ij}(\cdot;t,x)$ associated with the hyperbolic equation \eqref{kern T3 cpw} passing through the point $(t,x,\xi)$, i.e.
$$\chi_{ij}(s;t,x,\xi)=(s,\chi_i(s;t,x),\chi_j(s;t,\xi)), \quad \forall s \in \R,$$
where we recall that $\chi_i$ and $\chi_j$ are defined in \eqref{ODE xi}.
For every $(t,x,\xi) \in \Rect$, we have
$$\chi_{ij}(s;t,x,\xi) \in \Rect, \quad \forall s \in (\ssin_{ij}(t,x,\xi),\ssout_{ij}(t,x,\xi)),$$
where we introduced
$$
\ssin_{ij}(t,x,\xi)=\max\ens{0,\ssin_i(t,x),\ssin_j(t,\xi)}>0,
\qquad
\ssout_{ij}(t,x,\xi)=\min\ens{\ssout_i(t,x),\ssout_j(t,\xi)}.
$$

Since the speeds $\lambda_i,\lambda_j$ are negative ($i,j \leq m$), when $s$ is increasing, $s \mapsto \chi_i(s;t,x), s \mapsto \chi_j(s;t,\xi)$ are decreasing.
Therefore, the associated characteristic $\chi_{ij}(\cdot;t,x,\xi)$ will exit the domain $\Rect$ through the planes $x=0$ or $\xi=0$.
This is why we can impose boundary conditions at $(t,0,\xi)$ and $(t,x,0)$ (see \eqref{kern T3 cpw}) and this is why it is enough to (uniquely) determine a solution on $\Rect$.
To be more precise, we can split $\Rect$ into three disjoint subsets:
$$\Rect=\Rect^+_{ij} \cup \Rect^-_{ij} \cup \discset_{ij},$$
where
\begin{align*}
\Rect^+_{ij}=&\ens{(t,x,\xi) \in \Rect, \quad \ssout_i(t,x) < \ssout_j(t,\xi)}, \\
\Rect^-_{ij}=&\ens{(t,x,\xi) \in \Rect, \quad \ssout_i(t,x) > \ssout_j(t,\xi)}, \\
\discset_{ij}=&\ens{(t,x,\xi) \in \Rect, \quad \ssout_i(t,x)=\ssout_j(t,\xi)}.
\end{align*}
With these notations, the characteristic $\chi_{ij}(\cdot;t,x,\xi)$ will either exit the domain $\Rect$ through the plane $x=0$ if $(t,x,\xi) \in \Rect^+_{ij}$ or through the plane $\xi=0$ if $(t,x,\xi) \in \Rect^-_{ij}$:

\begin{proposition}\label{prop caract out}
~
\begin{enumerate}[(i)]
\item
For every $(t,x,\xi) \in \Rect^+_{ij}$, we have $\chi_{ij}(s;t,x,\xi) \in \Rect^+_{ij}$ for every $s \in (t,\ssout_i(t,x))$.

\item
For every $(t,x,\xi) \in \Rect^-_{ij}$, we have $\chi_{ij}(s;t,x,\xi) \in \Rect^-_{ij}$ for every $s \in (t,\ssout_j(t,\xi))$.

\item\label{point iii}
For every $(t,x,\xi) \in \discset_{ij}$, we have $\chi_{ij}(s;t,x,\xi) \in \discset_{ij}$ for every $s \in (t,\ssout_i(t,x))=(t,\ssout_j(t,\xi))$.
\end{enumerate}

\end{proposition}

These three points directly follow from \eqref{invariance}.

\subsubsection{Existence and regularity of a solution to \eqref{kern T3 cpw}}\label{sect exist Fred}

Writing the solution of \eqref{kern T3 cpw} along the characteristic curve $\chi_{ij}(s;t,x,\xi)$ for $s \in [\ssin_{ij}(t,x,\xi),\ssout_{ij}(t,x,\xi)]$ and using the boundary conditions, we obtain the following ODE:
\begin{equation}\label{ODE kerh}
\left\{\begin{array}{l}
\ds \dds h_{ij}(\chi_{ij}(s;t,x,\xi))=-\pxi{\lambda_j}(s,\chi_j(s;t,\xi))h_{ij}(\chi_{ij}(s;t,x,\xi)),\\
\ds h_{ij}(\chi_{ij}(\ssout_{ij}(t,x,\xi);t,x,\xi))=b_{ij}(t,x,\xi),
\end{array}\right.
\end{equation}
where
\begin{equation}\label{explicit bij}
b_{ij}(t,x,\xi)
=
\left\{\begin{array}{cl}
\ds 0 & \mbox{ if } (t,x,\xi) \in \Rect^+_{ij}, \\
\ds -\frac{g^2_{ij}(\ssout_j(t,\xi),\chi_i(\ssout_j(t,\xi);t,x))}{\lambda_j(\ssout_j(t,\xi),0)} & \mbox{ if } (t,x,\xi) \in \Rect^-_{ij}.
\end{array}\right.
\end{equation}
Integrating this ODE over $[t,\ssout_{ij}(t,x,\xi)]$ yields the integral equation
\begin{equation}\label{valuehij}
h_{ij}(t,x,\xi)=b_{ij}(t,x,\xi)
+\int_t^{\ssout_{ij}(t,x,\xi)} \pxi{\lambda_j}(s,\chi_j(s;t,\xi))h_{ij}(\chi_{ij}(s;t,x,\xi)) \, ds.
\end{equation}

\noindent
In this case, this integral equation is very easily solved by taking (as it is in fact directly seen from the ODE \eqref{ODE kerh}):
\begin{equation}\label{kij Fred}
h_{ij}(t,x,\xi)
=b_{ij}(t,x,\xi)e^{\int_t^{\ssout_{ij}(t,x,\xi)} \pxi{\lambda_j}(s,\chi_j(s;t,\xi)) \, ds}.
\end{equation}
Clearly,  $h_{ij}=0$ for $i \leq j$ (i.e. $H_{--}$ is indeed strictly lower triangular) since $g^2_{ij}=0$ for such indices (see \eqref{Gamma cond zero}).
Obviously, $h_{ij} \in C^0(\clos{\Rect^+_{ij}}) \cap L^\infty(\Rect^+_{ij})$.
On the other hand, thanks in particular to the regularities \eqref{reg xi}, \eqref{reg sin sout}, the bounds \eqref{bounds sin sout} and the assumption $\px{\Lambda} \in L^{\infty}((0,+\infty)\times(0,1))^{n \times n}$, we can check that
$$h_{ij} \in C^0(\clos{\Rect^-_{ij}}) \cap L^\infty(\Rect^-_{ij}).$$

\subsubsection{Characterization of $\Rect^{\pm}_{ij}$ and $\discset_{ij}$}\label{sect caract surf}

Let us now show that
\begin{equation}\label{caract S1 S2}
\begin{array}{l}
\Rect^-_{ij}=\ens{(t,x,\xi) \in \Rect, \quad \xi<\psi_{ij}(t,x)}, \\
\Rect^+_{ij}=\ens{(t,x,\xi) \in \Rect, \quad \xi>\psi_{ij}(t,x)},
\end{array}
\end{equation}
where $\psi_{ij} \in C^1([0,+\infty)\times[0,1])$ satisfies the semi-linear hyperbolic equation \eqref{equ psi}.
First of all, it follows from \eqref{increases} and the implicit function theorem that there exists a function $\psi_{ij}\in C^1([0,+\infty)\times[0,1])$, $0 \leq \psi_{ij} \leq 1$, such that
\begin{equation}\label{IFT}
\ssout_i(t,x)=\ssout_j(t,\xi)
\qquad \Longleftrightarrow \qquad
\xi=\psi_{ij}(t,x).
\end{equation}
This shows that
\begin{equation}\label{caract interface}
\discset_{ij}=\ens{(t,x,\xi)\in \Rect, \quad \xi=\psi_{ij}(t,x)}.
\end{equation}

\noindent
On the other hand, thanks to \eqref{increases} and \eqref{IFT} we have
$$
\xi>\psi_{ij}(t,x)
\qquad \Longleftrightarrow \qquad
\ssout_j(t,\xi)>\ssout_j(t,\psi_{ij}(t,x))=\ssout_i(t,x).
$$
This shows the equality \eqref{caract S1 S2} for $\Rect^+_{ij}$.
The equality for $\Rect^-_{ij}$ can be proved similarly.

It remains to show that $\psi_{ij}$ satisfies the semi-linear hyperbolic equation \eqref{equ psi}.
This in fact follows from \eqref{caract interface} and \ref{point iii} of Proposition \ref{prop caract out}.
Indeed, thanks to these results, we have
$$
\chi_j(s;t,\psi_{ij}(t,x))=\psi_{ij}(s,\chi_i(s;t,x)), \quad \forall s \in (t,\ssout_i(t,x)) =(t,\ssout_j(t,\psi_{ij}(t,x))).
$$
Taking the derivative of this identity at $s=t^+$, we immediately obtain the equation in \eqref{equ psi}.
On the other hand, letting $s \to \ssout_i(t,x)^-=\ssout_j(t,\psi_{ij}(t,x))^-$ and then letting $x \to 0^+$, we obtain the second condition $\psi_{ij}(t,0)=0$.

\subsection{Kernel for the Volterra transformation}\label{sect exist kern T2}

In this section we prove Theorem \ref{thm exist kern}, that is we prove the existence of a suitably smooth matrix-valued function $K=(k_{ij})_{1 \leq i,j \leq n}$ such that
\begin{equation}\label{kern equ bis}
\left\{\begin{array}{l}
\ds \pt{K}(t,x,\xi)
+\Lambda(t,x)\px{K}(t,x,\xi)
+\pxi{K}(t,x,\xi)\Lambda(t,\xi)
+K(t,x,\xi)\widetilde{M}^1(t,\xi)=0, \\
\ds K(t,x,x)\Lambda(t,x)-\Lambda(t,x)K(t,x,x)=-M^1(t,x),
\end{array}\right.
\end{equation}
where we introduced the notation
\begin{equation}\label{def Mtu}
\widetilde{M}^1(t,\xi)=\pxi{\Lambda}(t,\xi)+M^1(t,\xi).
\end{equation}
Note in particular that $\widetilde{M}^1 \in L^{\infty}((0,+\infty)\times(0,1))^{n \times n}$ thanks to the assumption \eqref{reg parameters}.
Besides, noticing \eqref{def G2 2}, we also want the matrix
\begin{equation}\label{G2-- triang}
-K_{--}(t,x,0)\Lambda_{--}(t,0)-K_{-+}(t,x,0)\Lambda_{++}(t,0)Q^1(t) \qquad \text{ to be strictly lower triangular.}
\end{equation}

\subsubsection{Preliminaries}

Let us rewrite \eqref{kern equ bis} by block.
It is equivalent to the following four sub-systems:
$$
\left\{\begin{array}{l}
\ds \pt{K_{--}}(t,x,\xi)
+\Lambda_{--}(t,x)\px{K_{--}}(t,x,\xi)
+\pxi{K_{--}}(t,x,\xi)\Lambda_{--}(t,\xi)
\\
\quad \quad \quad \quad \quad \quad \quad \quad \quad \quad \quad \quad \quad \quad \quad \quad \quad
+K_{--}(t,x,\xi)\widetilde{M}_{--}^1(t,\xi)
+K_{-+}(t,x,\xi)\widetilde{M}_{+-}^1(t,\xi)
=0,
\\
\ds K_{--}(t,x,x)\Lambda_{--}(t,x)-\Lambda_{--}(t,x)K_{--}(t,x,x)=-M^1_{--}(t,x),
\end{array}\right.
$$

$$
\left\{\begin{array}{l}
\ds \pt{K_{-+}}(t,x,\xi)
+\Lambda_{--}(t,x)\px{K_{-+}}(t,x,\xi)
+\pxi{K_{-+}}(t,x,\xi)\Lambda_{++}(t,\xi)
\\
\quad \quad \quad \quad \quad \quad \quad \quad \quad \quad \quad \quad \quad \quad \quad \quad \quad
+K_{--}(t,x,\xi)\widetilde{M}_{-+}^1(t,\xi)
+K_{-+}(t,x,\xi)\widetilde{M}_{++}^1(t,\xi)
=0,
\\
\ds K_{-+}(t,x,x)\Lambda_{++}(t,x)-\Lambda_{--}(t,x)K_{-+}(t,x,x)=-M^1_{-+}(t,x),
\end{array}\right.
$$

$$
\left\{\begin{array}{l}
\ds \pt{K_{+-}}(t,x,\xi)
+\Lambda_{++}(t,x)\px{K_{+-}}(t,x,\xi)
+\pxi{K_{+-}}(t,x,\xi)\Lambda_{--}(t,\xi)
\\
\quad \quad \quad \quad \quad \quad \quad \quad \quad \quad \quad \quad \quad \quad \quad \quad \quad
+K_{+-}(t,x,\xi)\widetilde{M}_{--}^1(t,\xi)
+K_{++}(t,x,\xi)\widetilde{M}_{+-}^1(t,\xi)
=0,
\\
\ds K_{+-}(t,x,x)\Lambda_{--}(t,x)-\Lambda_{++}(t,x)K_{+-}(t,x,x)=-M^1_{+-}(t,x),
\end{array}\right.
$$

$$
\left\{\begin{array}{l}
\ds \pt{K_{++}}(t,x,\xi)
+\Lambda_{++}(t,x)\px{K_{++}}(t,x,\xi)
+\pxi{K_{++}}(t,x,\xi)\Lambda_{++}(t,\xi)
\\
\quad \quad \quad \quad \quad \quad \quad \quad \quad \quad \quad \quad \quad \quad \quad \quad \quad
+K_{+-}(t,x,\xi)\widetilde{M}_{-+}^1(t,\xi)
+K_{++}(t,x,\xi)\widetilde{M}_{++}^1(t,\xi)
=0,
\\
\ds K_{++}(t,x,x)\Lambda_{++}(t,x)-\Lambda_{++}(t,x)K_{++}(t,x,x)=-M^1_{++}(t,x),
\end{array}\right.
$$

\begin{remark}
We see that $K_{--}$ is coupled only with $K_{-+}$ and that $K_{+-}$ is coupled only with $K_{++}$.
Moreover, the systems satisfied by $(K_{--},K_{-+})$ and by $(K_{+-},K_{++})$ are similar.
Therefore, from now on we only focus on the system satisfied by $(K_{--},K_{-+})$ (note that the extra condition \eqref{G2-- triang} only concerns this system).
In addition, because of the nature of the coupling terms inside the domain (namely, matrix multiplication by the right), we see that the entries from different rows are not coupled.
Therefore, for the rest of Section \ref{sect exist kern T2}, we assume that
$$i \in \ens{1,\ldots,m} \text{ is fixed.}$$
\end{remark}

Let us now rewrite the equations for $K_{--}$ and $K_{-+}$ component-wise.
For convenience, we introduce
\begin{equation}\label{def R}
r_{ij}(t,x)=\frac{-m^1_{ij}(t,x)}{\lambda_j(t,x)-\lambda_i(t,x)} \qquad (j \neq i).
\end{equation}
Note that $r_{ij} \in C^0([0,+\infty)\times[0,1])$.
Moreover, $r_{ij} \in L^{\infty}((0,+\infty)\times(0,1))$ thanks to \eqref{hyp speeds bis}.

\noindent
We have:
\begin{enumerate}[1)]
\item
If $j \neq i$, then
\begin{equation}\label{kern equ A}
\left\{\begin{array}{l}
\ds \pt{k_{ij}}(t,x,\xi)+\lambda_i(t,x) \px{k_{ij}}(t,x,\xi) +\lambda_j(t,\xi) \pxi{k_{ij}}(t,x,\xi)
+\sum_{\ell=1}^n  k_{i \ell}(t,x,\xi)\widetilde{m}^1_{\ell j}(t,\xi)
=0, \\
\ds k_{ij}(t,x,x)=r_{ij}(t,x).
\end{array}\right.
\end{equation}

\item
If $j=i$, then
\begin{equation}\label{kern equ C}
\pt{k_{ii}}(t,x,\xi)+\lambda_i(t,x) \px{k_{ii}}(t,x,\xi) +\lambda_i(t,\xi) \pxi{k_{ii}}(t,x,\xi)
+\sum_{\ell=1}^n  k_{i \ell}(t,x,\xi)\widetilde{m}^1_{\ell i}(t,\xi)
=0.
\end{equation}

\end{enumerate}

The geometric situation of the characteristics is more complicated than in Section \ref{sect exist kern T3}, it is detailed in Section \ref{sect caract Volt} below.
For the moment, let us just point out that we will have to consider parameters $s<t$ (compare with Section \ref{sect exist kern T3}) and, consequently, we should also add an artificial boundary condition at $t=0$ (the value of $k_{ij}$ at a point $(t,x,\xi) \in\Tau$ for sufficiently small $t$ can not be obtained from its values on the planes $\xi=x$ or $x=1$).
To avoid imposing such a condition we can equivalently study \eqref{kern equ A}-\eqref{kern equ C} on the domain extended in time
$$\Taub=\ens{(t,x,\xi) \in \R\times(0,1)\times(0,1), \quad \xi<x}.$$
Therefore, we need the values of $\widetilde{m}_{\ell j}^1$ and $r_{ij}$ for negative $t$.
We also need the values of $q^1_{\ell j}$ for negative $t$ since we want to consider the property \eqref{G2-- triang}.
To this end we extend $M$ to $\R \times[0,1]$ (recall that its diagonal elements were already extended in the proof of Proposition \ref{exist transfophi}) and we extend $Q$ to $\R$ in such a way that the property \eqref{reg parameters} is preserved.
This extends $\widetilde{m}_{\ell j}^1$ and $r_{ij}$ to $\R \times[0,1]$ and $q^1_{\ell j}$ to $\R$ through the formula \eqref{def Mtu}, \eqref{def R} and \eqref{def M1}, \eqref{sol equ phi_i}, with
$$
\widetilde{m}_{\ell j}^1, r_{ij} \in C^0(\mathbb{R}\times[0,1]) \cap L^{\infty}(\R\times(0,1)),
\qquad q^1_{\ell j} \in C^0(\R) \cap L^{\infty}(\R).
$$

\subsubsection{The characteristics of \eqref{kern equ A}-\eqref{kern equ C}}\label{sect caract Volt}

For each $(t,x,\xi) \in \R^3$ fixed, we still denote by $\chi_{ij}(\cdot;t,x,\xi)$ the characteristic curve associated with the hyperbolic system \eqref{kern equ A}-\eqref{kern equ C} passing through the point $(t,x,\xi)$, i.e.
$$\chi_{ij}(s;t,x,\xi)=(s,\chi_i(s;t,x),\chi_j(s;t,\xi)), \quad \forall s \in \R.$$

We now need to find for which parameters $s$ the characteristic $\chi_{ij}(s;t,x,\xi)$ stays in the domain $\Taub$ when $(t,x,\xi) \in \Taub$.
To this end, we introduce the following sets for $j \in \ens{1,\ldots,m}$:
\begin{align*}
& \SSinp_{ij}=\ens{(t,x,\xi) \in \Taub, \quad \ssin_i(t,x) < \ssin_j(t,\xi)}, \\
& \SSinm_{ij}=\ens{(t,x,\xi) \in \Taub, \quad \ssin_i(t,x) > \ssin_j(t,\xi)}, \\
& \discsetin_{ij}=\ens{(t,x,\xi) \in \Taub, \quad \ssin_i(t,x)=\ssin_j(t,\xi)},
\end{align*}
and
\begin{align*}
& \SSoutp_{ij}=\ens{(t,x,\xi) \in \Taub, \quad \ssout_i(t,x) < \ssout_j(t,\xi)}, \\
& \SSoutm_{ij}=\ens{(t,x,\xi) \in \Taub, \quad \ssout_i(t,x) > \ssout_j(t,\xi)}, \\
& \discsetout_{ij}=\ens{(t,x,\xi) \in \Taub, \quad \ssout_i(t,x)=\ssout_j(t,\xi)}.
\end{align*}
As in Section \ref{sect caract surf} we can show that $\SSoutp_{ij} \cap \Tau=\Tau^+_{ij}$ and $\SSoutm_{ij} \cap \Tau=\Tau^-_{ij}$ (we recall that $\Tau^+_{ij}$ and $\Tau^-_{ij}$ are defined in the statement of Theorem \ref{thm exist kern}).

The following proposition gives precise information about the exit of the characteristics from the domain $\Taub$ (the proof is postponed to Appendix \ref{app prop} for the sake of the presentation; we refer to Figures \ref{figure prop caract 1}, \ref{figure prop caract 2}, \ref{figure prop caract 3} and \ref{figure prop caract 4} for a clarification of the geometric situation at a fixed $t$):

\begin{proposition}\label{prop caract exit}
~
\begin{enumerate}[(i)]
\item\label{exit 1}
For every $j \in \ens{1,\ldots,i-1}$, there exists a unique $\sinterin_{ij} \in C^0(\clos{\Taub})$ with $(t,x,\xi) \mapsto t-\sinterin_{ij}(t,x,\xi) \in L^{\infty}(\Taub)$ such that, for every $t \in \R$ and $0 \leq \xi<x<1$, we have $\sinterin_{ij}(t,x,\xi)<t$ (and $\sinterin_{ij}(t,x,\xi)=t$ otherwise) with
$$
\chi_{ij}(s;t,x,\xi) \in \Taub, \quad \forall s \in \left(\sinterin_{ij}(t,x,\xi),t\right),
$$
and
$$
\left\{\begin{array}{ll}
\chi_j(\sinterin_{ij}(t,x,\xi);t,\xi)=\chi_i(\sinterin_{ij}(t,x,\xi);t,x)
& \mbox{ if } (t,x,\xi) \in \clos{\SSinp_{ij}}, \\
\chi_i(\sinterin_{ij}(t,x,\xi);t,x)=1
& \mbox{ if } (t,x,\xi) \in \clos{\SSinm_{ij}}.
\end{array}\right.
$$

\item\label{exit 2}
For $j=i$, there exists a unique $\sinterout_{ii} \in C^0(\clos{\Taub})$ with $(t,x,\xi) \mapsto \sinterout_{ii}(t,x,\xi)-t \in L^{\infty}(\Taub)$ such that, for every $t \in \R$ and $0<\xi \leq x \leq 1$, we have $\sinterout_{ii}(t,x,\xi)>t$ (and $\sinterout_{ii}(t,x,\xi)=t$ otherwise) and, if in addition $\xi<x$, then we have
$$\chi_{ii}(s;t,x,\xi) \in \Taub, \quad \forall s \in \left(t, \sinterout_{ii}(t,x,\xi)\right),$$
and
$$
\chi_i(\sinterout_{ii}(t,x,\xi);t,\xi)=0.
$$

\item\label{exit 3}
For every $j \in \ens{i+1,\ldots,m}$, there exists a unique $\sinterout_{ij} \in C^0(\clos{\Taub})$ with $(t,x,\xi) \mapsto \sinterout_{ij}(t,x,\xi)-t \in L^{\infty}(\Taub)$ such that, for every $t \in \R$ and $0<\xi<x \leq 1$, we have $\sinterout_{ij}(t,x,\xi)>t$ (and $\sinterout_{ij}(t,x,\xi)=t$ otherwise) with
$$\chi_{ij}(s;t,x,\xi) \in \Taub, \quad \forall s \in \left(t, \sinterout_{ij}(t,x,\xi)\right),$$
and
$$
\left\{\begin{array}{ll}
\chi_j(\sinterout_{ij}(t,x,\xi);t,\xi)=\chi_i(\sinterout_{ij}(t,x,\xi);t,x)
& \mbox{ if } (t,x,\xi) \in \clos{\SSoutp_{ij}}, \\
\chi_j(\sinterout_{ij}(t,x,\xi);t,\xi)=0
& \mbox{ if } (t,x,\xi) \in \clos{\SSoutm_{ij}}.
\end{array}\right.
$$

\item\label{exit 4}
For every $j \in \ens{m+1,\ldots,n}$, there exists a unique $\sinterout_{ij} \in C^0(\clos{\Taub})$ with $(t,x,\xi) \mapsto \sinterout_{ij}(t,x,\xi)-t \in L^{\infty}(\Taub)$ such that, for every $t \in \R$ and $0 \leq \xi<x \leq 1$, we have $\sinterout_{ij}(t,x,\xi)>t$ (and $\sinterout_{ij}(t,x,\xi)=t$ otherwise) with
$$\chi_{ij}(s;t,x,\xi) \in \Taub, \quad \forall s \in \left(t, \sinterout_{ij}(t,x,\xi)\right),$$
and
$$
\chi_j(\sinterout_{ij}(t,x,\xi);t,\xi)=\chi_i(\sinterout_{ij}(t,x,\xi);t,x).
$$
\end{enumerate}
\end{proposition}

\begin{minipage}[c]{.46\linewidth}
\centering

\scalebox{0.75}{
\begin{tikzpicture}[scale=2]
\draw [<->] (0,2.35) node [left]  {$\xi$} -- (0,0) -- (2.35,0) node [below right] {$x$};
\draw (0,0) node [left, below] {$0$};
\draw (2,0) node [below] {$1$};
\draw (0,0) -- (2,0) -- (2,2) -- (0,0);
\draw[ultra thick, domain=0.95:2.05] plot (\x,{2*\x-2});

\draw[dashed, domain=1.1:1.6] plot (\x,{1/2+2*(\x-1)});
\draw (1.1,0.7) node {$\times$};
\draw (1,0.55) node {$s=t$};
\draw (1.5,1.5) node {$\times$} node[left] {$s=\sinterin_{ij}$};

\draw[dashed, domain=1.6:2.1] plot (\x,{1/2+2*(\x-3/2)});
\draw (1.6,0.7) node {$\times$};
\draw (1.7,0.55) node {$s=t$};
\draw (2,1.5) node {$\times$} node[right] {$s=\sinterin_{ij}$};

\draw (0.75,0.15) node {$\SSinp_{ij}$};
\draw (1.4,0.15) node {$\SSinm_{ij}$};

\draw (1.1,-0.1) node [below] {$\xi=0$};
\draw (2.1,0.9) node [right] {$x=1$};
\draw (0.65,0.85) node[rotate=45] {$\xi=x$};

      \end{tikzpicture}
      }
\captionof{figure}{Definition of $\sinterin_{ij}$}\label{figure prop caract 1}
\end{minipage}
\begin{minipage}[c]{.46\linewidth}
\centering
\scalebox{0.75}{
\begin{tikzpicture}[scale=2]
\draw [<->] (0,2.35) node [left]  {$\xi$} -- (0,0) -- (2.35,0) node [below right] {$x$};
\draw (0,0) node [left, below] {$0$};
\draw (2,0) node [below] {$1$};
\draw (0,0) -- (2,0) -- (2,2) -- (0,0);

\draw[dashed, domain=0.4:1.6] plot (\x,{1.1+(\x-1.6)});
\draw (1.6,1.1) node {$\times$};
\draw (1.75,0.95) node {$s=t$};
\draw (0.5,0) node {$\times$} node[below] {$s=\sinterout_{ii}$};

\draw (1.1,-0.1) node [below] {$\xi=0$};
\draw (2.1,0.9) node [right] {$x=1$};
\draw (0.65,0.85) node[rotate=45] {$\xi=x$};

      \end{tikzpicture}
      }
\captionof{figure}{Definition of $\sinterout_{ii}$}\label{figure prop caract 2}
\end{minipage}

\begin{minipage}[c]{.46\linewidth}
\centering

\scalebox{0.75}{
\begin{tikzpicture}[scale=2]
\draw [<->] (0,2.35) node [left]  {$\xi$} -- (0,0) -- (2.35,0) node [below right] {$x$};
\draw (0,0) node [left, below] {$0$};
\draw (2,0) node [below] {$1$};
\draw (0,0) -- (2,0) -- (2,2) -- (0,0);

\draw (1.1,-0.1) node [below] {$\xi=0$};
\draw (2.1,0.9) node [right] {$x=1$};
\draw (0.65,0.85) node[rotate=45] {$\xi=x$};

\draw[ultra thick, domain=-0.05:2.05] plot (\x,{\x/2});
\draw (1.75,1.1) node {$\SSoutp_{ij}$};
\draw (1.75,0.5) node {$\SSoutm_{ij}$};

\draw[dashed, domain=0.4:1] plot (\x,{1/4+(\x-1)/2});
\draw (0.5,0) node {$\times$} node[below] {$s=\sinterout_{ij}$};
\draw (1,0.25) node {$\times$};
\draw (1,0.25) node[right] {$s=t$};

\draw[dashed, domain=0.2:1.1] plot (\x,{0.7+(\x-1.1)/2});
\draw (0.3,0.3) node {$\times$};
\draw (0.175,0.5) node {$s=\sinterout_{ij}$};
\draw (1.1,0.7) node {$\times$};
\draw (1.1,0.7) node[above] {$s=t$};

      \end{tikzpicture}
      }
\captionof{figure}{Definition of $\sinterout_{ij}$ when $i<j \leq m$}\label{figure prop caract 3}

\end{minipage}
\begin{minipage}[c]{.46\linewidth}
\centering

\scalebox{0.75}{
\begin{tikzpicture}[scale=2]
\draw [<->] (0,2.35) node [left]  {$\xi$} -- (0,0) -- (2.35,0) node [below right] {$x$};
\draw (0,0) node [left, below] {$0$};
\draw (2,0) node [below] {$1$};
\draw (0,0) -- (2,0) -- (2,2) -- (0,0);

\draw[dashed, domain=1.0875:1.5] plot (\x,{1/4-3*(\x-1.5)});
\draw (1.5,0.25) node {$\times$} node[below] {$s=t$};
\draw (1.1875,1.1875) node {$\times$} node[left] {$s=\sinterout_{ij}$};

\draw (1.1,-0.1) node [below] {$\xi=0$};
\draw (2.1,0.9) node [right] {$x=1$};
\draw (0.65,0.85) node[rotate=45] {$\xi=x$};

      \end{tikzpicture}
      }
\captionof{figure}{Definition of $\sinterout_{ij}$ when $j>m$}\label{figure prop caract 4}
\end{minipage}

\vspace{1cm}

In order to show that the system \eqref{kern equ A}-\eqref{kern equ C} is well-posed, we see from Proposition \ref{prop caract exit} that we need to add some conditions:

\begin{enumerate}[1)]
\item
when $j \in \ens{1,\ldots,i-1}$, we will consider the following artificial boundary condition at $x=1$:
$$k_{ij}(t,1,\xi)=a_{ij}(t,\xi), \quad \forall j \in \ens{1,\ldots,i-1},$$
where $a_{ij} \in C^0(\R\times[0,1]) \cap L^{\infty}(\R\times(0,1))$ is any function that satisfies the corresponding $C^0$-compatibility conditions at $(t,x,\xi)=(t,1,1)$, namely:
\begin{equation}\label{compa cond ra}
a_{ij}(t,1)=r_{ij}(t,1), \quad \forall t \in \R.
\end{equation}

\item
when $j \in \ens{i,\ldots,m}$, we have some freedom for the boundary condition.
We choose to consider the following one in order to obtain \eqref{G2-- triang}:
$$k_{ij}(t,x,0)=\sum_{\ell=1}^{p} k_{i, m+\ell}(t,x,0)\widetilde{q}^1_{\ell j}(t), \quad \forall j \in \ens{i,\ldots,m},$$
where we set
$$\widetilde{q}^1_{\ell j}(t)=-\frac{1}{\lambda_j(t,0)} \lambda_{m+\ell}(t,0) q^1_{\ell j}(t).$$
Note that $\widetilde{q}^1_{\ell j} \in C^0(\R) \cap L^{\infty}(\R)$.

\end{enumerate}

In summary, we are going to solve the following coupled hyperbolic system:
\begin{enumerate}[1)]
\item
If $j \in \ens{1,\ldots,i-1}$, then
\begin{equation}\label{kern equ ij 1}
\left\{\begin{array}{l}
\ds \pt{k_{ij}}(t,x,\xi)+\lambda_i(t,x) \px{k_{ij}}(t,x,\xi) +\lambda_j(t,\xi) \pxi{k_{ij}}(t,x,\xi)
+\sum_{\ell=1}^n  k_{i \ell}(t,x,\xi)\widetilde{m}^1_{\ell j}(t,\xi)
=0, \\
\ds k_{ij}(t,x,x)=r_{ij}(t,x), \\
\ds k_{ij}(t,1,\xi)=a_{ij}(t,\xi).
\end{array}\right.
\end{equation}

\item
If $j=i$, then
\begin{equation}\label{kern equ ij 2}
\left\{\begin{array}{l}
\ds \pt{k_{ii}}(t,x,\xi)+\lambda_i(t,x) \px{k_{ii}}(t,x,\xi) +\lambda_i(t,\xi) \pxi{k_{ii}}(t,x,\xi)
+\sum_{\ell=1}^n  k_{i \ell}(t,x,\xi)\widetilde{m}^1_{\ell i}(t,\xi)
=0,
\\
\ds k_{ii}(t,x,0)=\sum_{\ell=1}^{p} k_{i, m+\ell}(t,x,0)\widetilde{q}^1_{\ell i}(t).
\end{array}\right.
\end{equation}

\item
If $j \in \ens{i+1,\ldots,m}$, then
\begin{equation}\label{kern equ ij extr1}
\left\{\begin{array}{l}
\ds \pt{k_{ij}}(t,x,\xi)+\lambda_i(t,x) \px{k_{ij}}(t,x,\xi) +\lambda_j(t,\xi) \pxi{k_{ij}}(t,x,\xi)
+\sum_{\ell=1}^n  k_{i \ell}(t,x,\xi)\widetilde{m}^1_{\ell j}(t,\xi)
=0, \\
\ds k_{ij}(t,x,x)=r_{ij}(t,x),
\\
\ds k_{ij}(t,x,0)=\sum_{\ell=1}^{p} k_{i, m+\ell}(t,x,0)\widetilde{q}^1_{\ell j}(t).\\
\end{array}\right.
\end{equation}

\item
If $j \in \ens{m+1,\ldots,n}$, then
\begin{equation}\label{kern equ ij 3}
\left\{\begin{array}{l}
\ds \pt{k_{ij}}(t,x,\xi)+\lambda_i(t,x) \px{k_{ij}}(t,x,\xi)+\lambda_j(t,\xi) \pxi{k_{ij}}(t,x,\xi)
+\sum_{\ell=1}^n  k_{i \ell}(t,x,\xi)\widetilde{m}^1_{\ell j}(t,\xi)
=0, \\
\ds k_{ij}(t,x,x)=r_{ij}(t,x).
\end{array}\right.
\end{equation}
\end{enumerate}

\subsubsection{Transformation into integral equations}\label{charact j>i}

To prove the existence and uniqueness of the solution to the kernel equations \eqref{kern equ ij 1}-\eqref{kern equ ij 3} on $\Taub$,  we use the classical strategy that consists in transforming these hyperbolic equations into integral equations.
Then, in the next subsection, we will prove that this system of integral equations has a unique solution by using a fixed-point argument and appropriate estimates.

Let us introduce
$$
\widetilde{k}^0_{ij}(t,x,\xi)=
\left\{\begin{array}{ll}
\ds r_{ij}\left(\sinterin_{ij}(t,x,\xi),\chi_i\left(\sinterin_{ij}(t,x,\xi);t,x\right)\right) & \quad \mbox{ if } j \in \ens{1,\ldots,i-1} \mbox{ and } (t,x,\xi) \in \clos{\SSinp_{ij}},
\\
\ds a_{ij}\left(\sinterin_{ij}(t,x,\xi),\chi_j\left(\sinterin_{ij}(t,x,\xi);t,\xi\right)\right) & \quad \mbox{ if } j \in \ens{1,\ldots,i-1} \mbox{ and }  (t,x,\xi) \in \clos{\SSinm_{ij}}, \\
\ds r_{ij}\left(\sinterout_{ij}(t,x,\xi),\chi_i\left(\sinterout_{ij}(t,x,\xi);t,x\right)\right) & \quad \mbox{ if } j \in \ens{i+1,\ldots,m} \mbox{ and } (t,x,\xi) \in \clos{\SSoutp_{ij}},
\\
\ds r_{ij}\left(\sinterout_{ij}(t,x,\xi),\chi_i\left(\sinterout_{ij}(t,x,\xi);t,x\right)\right) & \quad \mbox{ if } j \in \ens{m+1,\ldots,n}.
\end{array}\right.
$$
Thanks to the $C^0$-compatibility condition \eqref{compa cond ra}, note in particular that
\begin{equation}\label{reg hatvarphi}
\widetilde{k}^0_{ij} \in C^0(\clos{\Taub}), \quad \forall j \in \ens{1,\ldots,i-1}.
\end{equation}

\noindent
Using now Proposition \ref{prop caract exit}, we can obtain that

\begin{enumerate}[1)]
\item
For $j \in \ens{1,\ldots,i-1}$, integrating \eqref{kern equ ij 1}  along the characteristic curve $\chi_{ij}(s;t,x,\xi)$ for $s \in (\sinterin_{ij}(t,x,\xi),t)$ yields the following integral equation:
$$
k_{ij}(t,x,\xi)
=\widetilde{k}^0_{ij}(t,x,\xi)
-\sum_{\ell=1}^n\int_{\sinterin_{ij}(t,x,\xi)}^{t}
k_{i \ell}(\chi_{ij}(s;t,x,\xi))
\widetilde{m}^{1}_{\ell j}(s,\chi_j(s;t,\xi)) \, ds.
$$

\item
For $j=i$, integrating \eqref{kern equ ij 2} along the characteristic curve $\chi_{ii}(s;t,x,\xi)$ for $s \in (t,\sinterout_{ii}(t,x,\xi))$ yields the following integral equation:
\begin{multline}\label{int equ j=i}
k_{ii}(t,x,\xi)
=\sum_{\ell=1}^{p}
k_{i, m+\ell}\left(\sinterout_{ii}(t,x,\xi),\chi_i\left(\sinterout_{ii}(t,x,\xi);t,x\right), 0\right)
\widetilde{q}^1_{\ell i}\left(\sinterout_{ii}(t,x,\xi)\right) \\
+\sum_{\ell=1}^n\int_t^{\sinterout_{ii}(t,x,\xi)}
k_{i \ell}(\chi_{ii}(s;t,x,\xi))
\widetilde{m}^{1}_{\ell i}(s,\chi_i(s;t,\xi)) \, ds.
\end{multline}

\item
For $j \in \ens{i+1,\ldots,m}$, integrating \eqref{kern equ ij extr1}  along the characteristic curve $\chi_{ij}(s;t,x,\xi)$ for $s \in (t,\sinterout_{ij}(t,x,\xi))$ yields the following integral equations:
\begin{multline*}
k_{ij}(t,x,\xi)
=\widetilde{k}^0_{ij}(t,x,\xi)
\\
+\sum_{\ell=1}^n\int_t^{\sinterout_{ij}(t,x,\xi)}
k_{i \ell}(\chi_{ij}(s;t,x,\xi))
\widetilde{m}^{1}_{\ell j}(s,\chi_j(s;t,\xi)) \, ds,
\quad (t,x,\xi) \in \clos{\SSoutp_{ij}},
\end{multline*}

and
\begin{multline}\label{int equ j>i 2}
k_{ij}(t,x,\xi)
=\sum_{\ell=1}^{p}
k_{i, m+\ell}\left(\sinterout_{ij}(t,x,\xi),\chi_i\left(\sinterout_{ij}(t,x,\xi);t,x\right), 0\right)
\widetilde{q}^1_{\ell j}\left(\sinterout_{ij}(t,x,\xi)\right) \\
+\sum_{\ell=1}^n\int_t^{\sinterout_{ij}(t,x,\xi)}
k_{i \ell}(\chi_{ij}(s;t,x,\xi))
\widetilde{m}^{1}_{\ell j}(s,\chi_j(s;t,\xi)) \, ds,
\quad (t,x,\xi) \in \clos{\SSoutm_{ij}}.
\end{multline}

\item
For $j \in \ens{m+1,\ldots,n}$, integrating \eqref{kern equ ij 3} along the characteristic curve $\chi_{ij}(s;t,x,\xi)$ for $s \in (t,\sinterout_{ij}(t,x,\xi))$ yields the following integral equation:
\begin{equation}\label{int equ j>m+1}
k_{ij}(t,x,\xi)
=\widetilde{k}^0_{ij}(t,x,\xi)
+\sum_{\ell=1}^n\int_t^{\sinterout_{ij}(t,x,\xi)}
k_{i \ell}(\chi_{ij}(s;t,x,\xi))
\widetilde{m}^{1}_{\ell j}(s,\chi_j(s;t,\xi)) \, ds.
\end{equation}

\item
We now want to plug \eqref{int equ j>m+1} into \eqref{int equ j=i} and \eqref{int equ j>i 2}, respectively.
From \eqref{int equ j>m+1} we have
\begin{multline}\label{to plug}
k_{i, m+\ell}\left(\sinterout_{ij}(t,x,\xi),\chi_i\left(\sinterout_{ij}(t,x,\xi);t,x\right), 0\right)
=\widetilde{k}^0_{i, m+\ell}\left(\sinterout_{ij}(t,x,\xi),\chi_i\left(\sinterout_{ij}(t,x,\xi);t,x\right), 0\right)
\\
+\sum_{q=1}^n\int_{\sinterout_{ij}(t,x,\xi)}^{\sinterout_{i, m+\ell}\left(\sinterout_{ij}(t,x,\xi),\chi_i\left(\sinterout_{ij}(t,x,\xi);t,x\right), 0\right)}
k_{i q}\left(\chi_{i, m+\ell}\left(s;\sinterout_{ij}(t,x,\xi),\chi_i\left(\sinterout_{ij}(t,x,\xi);t,x\right), 0\right)\right)
\\
\times
\widetilde{m}^{1}_{q,m+\ell}\left(s,
\chi_{m+\ell}\left(s;\sinterout_{ij}(t,x,\xi), 0\right)
\right) \, ds.
\end{multline}
Plugging \eqref{to plug} into \eqref{int equ j=i} and \eqref{int equ j>i 2}, we obtain, for every $j \in \ens{i,\ldots,m}$ and $(t,x,\xi) \in \clos{\SSoutm_{ij}}$,
\begin{multline*}
k_{ij}(t,x,\xi)
=
\widehat k^0_{ij}(t,x,\xi)
\\
+\sum_{\ell=1}^{p}
\Bigg(
\sum_{q=1}^n\int_{\sinterout_{ij}(t,x,\xi)}^{\sinterout_{i, m+\ell}\left(\sinterout_{ij}(t,x,\xi),\chi_i\left(\sinterout_{ij}(t,x,\xi);t,x\right), 0\right)}
k_{i q}\left(\chi_{i, m+\ell}\left(s;\sinterout_{ij}(t,x,\xi),\chi_i\left(\sinterout_{ij}(t,x,\xi);t,x\right), 0\right)\right)
\\
\times
\widetilde{m}^{1}_{q,m+\ell}\left(s,
\chi_{m+\ell}\left(s;\sinterout_{ij}(t,x,\xi), 0\right)
\right) \, ds
\Bigg)
\widetilde{q}^1_{\ell j}\left(\sinterout_{ij}(t,x,\xi)\right) \\
+\sum_{\ell=1}^n\int_t^{\sinterout_{ij}(t,x,\xi)}
k_{i \ell}(\chi_{ij}(s;t,x,\xi))
\widetilde{m}^{1}_{\ell j}(s,\chi_j(s;t,\xi)) \, ds,
\end{multline*}
where we introduced
$$
\widehat k^0_{ij}(t,x,\xi)
=
\sum_{\ell=1}^{p}\widetilde{k}^0_{i, m+\ell}\left(\sinterout_{ij}(t,x,\xi),\chi_i\left(\sinterout_{ij}(t,x,\xi);t,x\right), 0\right)
\widetilde{q}^1_{\ell j}\left(\sinterout_{ij}(t,x,\xi)\right).
$$
Note that
\begin{equation}\label{reg varphi}
\widehat k^0_{ij} \in C^0(\clos{\SSoutm_{ij}}) \cap L^\infty(\SSoutm_{ij}), \quad \forall j \in \ens{i+1,\ldots,m},
\end{equation}
and, since $\SSoutm_{ii}=\Taub$ (because of \eqref{increases}),
\begin{equation}\label{reg varphi ii}
\widehat k^0_{ii} \in C^0(\clos{\Taub}) \cap L^\infty(\Taub).
\end{equation}

\end{enumerate}

\begin{remark}
Observe that, in general, for $j \in \ens{i+1,\ldots,m}$, we have
$$\widehat k^0_{ij} \neq \widetilde{k}^0_{ij} \mbox{ on } \discsetout_{ij}.$$
This is the reason why we have to consider discontinuous kernels.
\end{remark}

\subsubsection{Solution to the integral equations}

In this subsection we show that there exists a unique solution to the system of integral equations of the previous section.
This will conclude the proof of Theorem \ref{thm exist kern}.

\paragraph{Fixed-point argument.}

As it is classical, we reformulate the existence of such a solution into the existence of a fixed-point of the mapping defined by the right-hand sides of these equations.
Let us first introduce $K^0=(k^0_{ij})_{\substack{1 \leq i \leq m \\ 1 \leq j \leq n}}$ defined by
$$
k^0_{ij}(t,x,\xi)
=
\left\{\begin{array}{ll}
\ds \widetilde{k}^0_{ij}(t,x,\xi) & \mbox{ if } j \in \ens{1,\ldots,i-1}, \\
\ds \widehat k^0_{ii}(t,x,\xi) & \mbox{ if } j=i, \\
\ds \widetilde{k}^0_{ij}(t,x,\xi) & \mbox{ if } j \in \ens{i+1,\ldots,m} \mbox{ and } (t,x,\xi) \in \SSoutp_{ij}, \\
\ds \widehat k^0_{ij}(t,x,\xi) & \mbox{ if } j \in \ens{i+1,\ldots,m} \mbox{ and } (t,x,\xi) \in \SSoutm_{ij}, \\
\ds \widetilde{k}^0_{ij}(t,x,\xi) & \mbox{ if } j \in \ens{m+1,\ldots,n}.
\end{array}\right.
$$
Thanks in particular to \eqref{reg hatvarphi}, \eqref{reg varphi} and \eqref{reg varphi ii}, we see that
$$
\begin{array}{l}
\ds k^0_{ij} \in C^0(\clos{\SSoutp_{ij}}) \cap C^0(\clos{\SSoutm_{ij}}) \cap L^{\infty}(\Taub) \quad \mbox{ if } j \in \ens{i+1,\ldots,m}, \\
\ds k^0_{ij} \in C^0(\clos{\Taub}) \cap L^{\infty}(\Taub) \quad \mbox{ otherwise.}
\end{array}
$$
It is this regularity that dictates the space in which we can work.
More precisely, let us introduce the vector space $B$ defined by
\begin{equation}\label{defB}
B=
\ens{
K=(k_{ij})_{\substack{1 \leq i \leq m \\ 1 \leq j \leq n}}, \,
\begin{array}{l}
\ds k_{ij} \in C^0(\clos{\SSoutp_{ij}}) \cap C^0(\clos{\SSoutm_{ij}}) \cap L^{\infty}(\Taub) \, \mbox{ if } j \in \ens{i+1,\ldots,m}, \\
\ds k_{ij} \in C^0(\clos{\Taub}) \cap L^{\infty}(\Taub) \,\mbox{ otherwise.}
\end{array}
}.
\end{equation}
We can check that $B$ is a Banach space when equipped with the $L^\infty$ norm.
Let us now introduce the mapping
$$\Phi:B \longrightarrow B,$$
defined, for every $K \in B$, by
$$\Phi(K)=K^0+\Phi_1(K)+\Phi_2(K),$$
where, for every $(t,x,\xi) \in \clos{\Taub}$,
\begin{multline}\label{def Phiz}
\left(\Phi_1(K)\right)_{ij}(t,x,\xi)=
\\
\left\{\begin{array}{ll}
\ds
-\sum_{\ell=1}^n \int_{\sinterin_{ij}(t,x,\xi)}^{t}
k_{i \ell}(\chi_{ij}(s;t,x,\xi))
\widetilde{m}^{1}_{\ell j}(s,\chi_j(s;t,\xi)) \, ds,
& \mbox{ if } j \in \ens{1,\ldots,i-1}, \\
\ds
\sum_{\ell=1}^n\int_t^{\sinterout_{ij}(t,x,\xi)}
k_{i \ell}(\chi_{ij}(s;t,x,\xi))
\widetilde{m}^{1}_{\ell j}(s,\chi_j(s;t,\xi)) \, ds,
& \mbox{ if } j \in \ens{i,\ldots,n},
\end{array}\right.
\end{multline}
and
\begin{multline}\label{def Psiz}
\left(\Phi_2(K)\right)_{ij}(t,x,\xi)=
\\
\sum_{\ell=1}^{p}
\Bigg(
\sum_{q=1}^n\int_{\sinterout_{ij}(t,x,\xi)}^{\sinterout_{i, m+\ell}\left(\sinterout_{ij}(t,x,\xi),\chi_i\left(\sinterout_{ij}(t,x,\xi);t,x\right), 0\right)}
k_{i q}\left(\chi_{i, m+\ell}\left(s;\sinterout_{ij}(t,x,\xi),\chi_i\left(\sinterout_{ij}(t,x,\xi);t,x\right), 0\right)\right)
\\
\times
\widetilde{m}^{1}_{q,m+\ell}\left(s,
\chi_{m+\ell}\left(s;\sinterout_{ij}(t,x,\xi), 0\right)
\right) \, ds
\Bigg)
\widetilde{q}^1_{\ell j}\left(\sinterout_{ij}(t,x,\xi)\right),
\end{multline}
if $j \in \ens{i,\ldots,m}$ and $(t,x,\xi) \in \clos{\SSoutm_{ij}}$, and $\left(\Phi_2(K)\right)_{ij}(t,x,\xi)=0$ otherwise (recall that $\SSoutm_{ii}=\Taub$).

\paragraph{Regularity of the mapping.}

First of all, we have to show that $\Phi$ is well defined, i.e. that for every $K \in B$, we have indeed
\begin{equation}\label{reg Phiz Psi}
\Phi_1(K) \in B, \qquad \Phi_2(K) \in B.
\end{equation}
This is not obvious since the function $s \mapsto \chi_{ij}(s;t,x,\xi)$ may take values in the set $\discsetout_{i \ell}$, where $k_{i\ell}$ is discontinuous (even for $j \not\in \ens{i+1,\ldots,m}$, where we expect $(\Phi_1(K))_{ij}$ to be continuous by definition of $B$).
The following result, close to Proposition \ref{prop caract exit}, shows that this may happen only at one point:
\begin{proposition}\label{prop sdisc}
Let $\ell \in \ens{i+1,\ldots,m}$ be fixed.
\begin{enumerate}[(i)]
\item
For every $j \in \ens{1,\ldots,i-1}$, for every $(t,x,\xi) \in \clos{\Taub}$, there is at most one $\sdisc_{ij\ell} \in (\sinterin_{ij}(t,x,\xi),t)$ such that $\chi_{ij}(\sdisc_{ij\ell};t,x,\xi) \in \discsetout_{i \ell}$.

\item
For every $j \in \ens{i,\ldots,n}$ with $j \neq \ell$, for every $(t,x,\xi) \in \clos{\Taub}$, there is at most one $\sdisc_{ij\ell} \in (t,\sinterout_{ij}(t,x,\xi))$ such that $\chi_{ij}(\sdisc_{ij\ell};t,x,\xi) \in \discsetout_{i \ell}$.
\end{enumerate}
\end{proposition}

This result shows in fact a stronger regularity than \eqref{reg Phiz Psi}, namely,
$$\Phi_1(K), \Phi_2(K) \in C^0(\clos{\Taub})^{m \times n} \cap L^{\infty}(\Taub)^{m \times n}.$$
The proof of Proposition \ref{prop sdisc} is postponed to Appendix \ref{app prop} for the sake of the presentation.

\paragraph{Contraction of the mapping.}

We will now prove that $\Phi^N$ is a contraction for $N \in \N^*$ large enough.
Therefore, the Banach fixed-point theorem can be applied, giving the existence (and uniqueness) of $K \in B$ such that
$$K=\Phi(K).$$
This will conclude the proof of Theorem \ref{thm exist kern}.
Now, to show that $\Phi^N$ is a contraction when $N$ is large, it is sufficient to prove the following estimate:

\begin{proposition}\label{contraction estimates}
There exists $C>0$ such that, for every $N \in \N^*$ and $K,H \in B$,
\begin{equation}\label{estim Phi}
\norm{\Phi^N(K)-\Phi^N(H)}_{L^{\infty}(\Taub)^{m \times n}}
\leq
\frac{C^N}{N!}
\norm{K-H}_{L^{\infty}(\Taub)^{m \times n}}.
\end{equation}

\end{proposition}

To establish \eqref{estim Phi} we will use the following key lemma:

\begin{lemma}\label{key lemma}
For every $i \in \ens{1,\ldots,m}$, there exist a function $\Omega_i \in C^1(\clos{\Taub}) \cap L^{\infty}(\Taub)$ and $\epsilon_0>0$ such that, for every $(t,x,\xi) \in \clos{\Taub}$, we have $\Omega_i(t,x,\xi) \geq 0$ with
\begin{equation}\label{Omega j<i}
\frac{\partial \Omega_i}{\partial t}(t,x,\xi)+\lambda_i(t,x)\frac{\partial \Omega_i}{\partial x}(t,x,\xi)+\lambda_j(t,\xi)\frac{\partial \Omega_i}{\partial \xi}(t,x,\xi)\geq \epsilon_0,
\qquad \forall j \in \ens{1,\ldots,i-1},
\end{equation}
and
\begin{equation}\label{Omega j>i}
\frac{\partial \Omega_i}{\partial t}(t,x,\xi)+\lambda_i(t,x)\frac{\partial \Omega_i}{\partial x}(t,x,\xi)+\lambda_j(t,\xi)\frac{\partial \Omega_i}{\partial \xi}(t,x,\xi)\leq -\epsilon_0,
\qquad \forall j \in \ens{i,\ldots,n}.
\end{equation}

\end{lemma}

The proof of Lemma \ref{key lemma} is postponed to Appendix \ref{app lemma} for the sake of the presentation.

\begin{remark}\label{rem key lemma}
In the time-independent case, we can take $\Omega_i(x,\xi)=\phi_i(x)-\nu\phi_i(\xi)$ (we recall that $\phi_i(x)=\int_0^x \frac{1}{-\lambda_i(y)} \, dy$) where $\nu \in [0,1)$ is any number such that $\nu>\max_{1 \leq j<i \leq m} \max_{\xi \in [0,1]} \lambda_i(\xi)/\lambda_j(\xi)$.
This function appeared for instance in \cite[Lemma 6.2]{HDMVK16} for systems with constant coefficients and in \cite[(A.32)]{HVDMK19} for systems with time-independent coefficients (see also \cite[Lemma A.4]{CVKB13} for $2 \times 2$ systems, where it is enough to take $\nu=0$ since \eqref{Omega j<i} becomes void).
\end{remark}

\begin{remark}\label{Omegai decreases}
Observe that it follows from the estimate \eqref{Omega j>i} that, for every $j \in \ens{i,\ldots,n}$,
$$s \mapsto \Omega_i\left(\chi_{ij}\left(s;t,x,\xi\right)\right) \text{ is strictly decreasing.}$$
This is the analogue to \cite[Remark 10]{HDMVK16}.
\end{remark}

We can now prove Proposition \ref{contraction estimates}:

\begin{proof}[Proof of Proposition \ref{contraction estimates}]
Let us denote by
$$R=\max \ens{\norm{\widetilde{M}^1}_{L^{\infty}(\R\times(0,1))^{n \times n}}, \norm{\widetilde{Q}^1}_{L^{\infty}(\R)^{p \times m}}}.$$

\begin{enumerate}[1)]
\item
We start with the estimate of $\norm{\Phi_1(K)-\Phi_1(H)}_{L^{\infty}(\Taub)^{m \times n}}$.
Set
$$C_1=\frac{n}{\epsilon_0}R.$$
Let $j \in \ens{1,\ldots,i-1}$.
From the definition \eqref{def Phiz} of $\Phi_1$ we see that
$$
\abs{\left(\Phi_1(K)-\Phi_1(H)\right)_{ij}(t,x,\xi)}
\leq
nR
\left(\int_{\sinterin_{ij}(t,x,\xi)}^{t} 1 \, ds\right)
\norm{K-H}_{L^{\infty}(\Taub)^{m \times n}}.
$$
Thanks to the estimate \eqref{Omega j<i} we can perform the change of variable $s \mapsto \theta(s)=\Omega_i(\chi_{ij}(s;t,x,\xi))$ and obtain
$$
\epsilon_0 \left(\int_{\sinterin_{ij}(t,x,\xi)}^{t} 1 \, ds\right)
\leq
\int_{\sinterin_{ij}(t,x,\xi)}^{t}\frac{d\theta}{ds}(s) \, ds
=\theta(t)-\theta(\sinterin_{ij}(t,x,\xi))
\leq \theta(t)
=\Omega_i(t,x,\xi).
$$

This gives the estimate
$$\abs{\left(\Phi_1(K)-\Phi_1(H)\right)_{ij}(t,x,\xi)}
\leq C_1 \Omega_i(t,x,\xi)\norm{K-H}_{L^{\infty}(\Taub)^{m \times n}}.$$
It is important to point out that the right-hand side does not depend on the second index $j$.
Computing $\Phi_1^2(H)-\Phi_1^2(K)=\Phi_1(\Phi_1(H))-\Phi_1(\Phi_1(K))$ and using the previous estimate, we obtain
\begin{multline*}
\abs{\left(\Phi_1^2(K)-\Phi_1^2(H)\right)_{ij}(t,x,\xi)}
\\
\leq
nR C_1
\left(\int_{\sinterin_{ij}(t,x,\xi)}^{t} \Omega_i(\chi_{ij}(s;t,x,\xi)) \, ds\right)
\norm{K-H}_{L^{\infty}(\Taub)^{m \times n}}.\\
\end{multline*}
Using again the change of variable $s \mapsto \theta(s)$ and \eqref{Omega j<i}, we obtain
\begin{multline*}
\epsilon_0 \left(\int_{\sinterin_{ij}(t,x,\xi)}^{t} \Omega_i(\chi_{ij}(s;t,x,\xi)) \, ds\right)
=\epsilon_0 \int_{\sinterin_{ij}(t,x,\xi)}^{t} \theta(s) \, ds
\leq
\int_{\sinterin_{ij}(t,x,\xi)}^{t} \theta(s) \frac{d\theta}{ds}(s) \, ds
\\
=\frac{\theta(t)^2}{2}-\frac{\theta(\sinterin_{ij}(t,x,\xi))^2}{2}
\leq \frac{\theta(t)^2}{2}
=\frac{\Omega_i(t,x,\xi)^2}{2}.
\end{multline*}
This gives the estimate
$$
\abs{\left(\Phi_1^2(K)-\Phi_1^2(H)\right)_{ij}(t,x,\xi)}
\leq
\frac{C_1^2\Omega_i(t,x,\xi)^2}{2}
\norm{K-H}_{L^{\infty}(\Taub)^{m \times n}}.
$$

By induction, we easily obtain that, for every $N \in \N^*$,
\begin{equation}\label{estim Phiz j<i}
\abs{\left(\Phi_1^N(K)-\Phi_1^N(H)\right)_{ij}(t,x,\xi)}
\leq
 \frac{C_1^N \Omega_i(t,x,\xi)^N}{N!}
\norm{K-H}_{L^{\infty}(\Taub)^{m \times n}}.
\end{equation}

Using the estimate \eqref{Omega j>i} instead of \eqref{Omega j<i}, we can obtain exactly the same estimate as \eqref{estim Phiz j<i} for $j \in \ens{i,\ldots,n}$.
Since $\Omega_i$ is bounded, it follows that
$$
\norm{\Phi_1^N(K)-\Phi_1^N(H)}_{L^{\infty}(\Taub)^{m \times n}}
\leq
\frac{C^N}{N!}
\norm{K-H}_{L^{\infty}(\Taub)^{m \times n}},
$$
for some $C$ independent of $N$ and $K,H$.

\item
Let us now take care of $\Phi_2(K)-\Phi_2(H)$.
The idea to estimate this term is essentially the same as before, with the extra use of the decreasing property stated in Remark \ref{Omegai decreases}.
Set
$$C_2=\frac{n}{\epsilon_0}R^2 p.$$
From the definition \eqref{def Psiz} of $\Phi_2$ we see that, for $j \in \ens{i,\ldots,m}$ and $(t,x,\xi) \in \clos{\SSoutm_{ij}}$,
\begin{multline*}
\abs{\left(\Phi_2(K)-\Phi_2(H)\right)_{ij}(t,x,\xi)}
\\
\leq
n R^2 \sum_{\ell=1}^p
\left(\int_{\sinterout_{ij}(t,x,\xi)}^{\sinterout_{i, m+\ell}\left(\sinterout_{ij}(t,x,\xi),\chi_i\left(\sinterout_{ij}(t,x,\xi);t,x\right), 0\right)} 1 \, ds\right)
\norm{K-H}_{L^{\infty}(\Taub)^{m \times n}}.
\end{multline*}
Thanks to the estimate \eqref{Omega j>i} we can perform again the change of variable
$$s \mapsto \theta(s)=\Omega_i(\chi_{ij}(s;t,x,\xi)),$$
which is decreasing since $j \geq i$ (see Remark \ref{Omegai decreases}), and obtain
\begin{multline*}
\epsilon_0 \left(\int_{\sinterout_{ij}(t,x,\xi)}^{\sinterout_{i, m+\ell}\left(\sinterout_{ij}(t,x,\xi),\chi_i\left(\sinterout_{ij}(t,x,\xi);t,x\right), 0\right)} 1 \, ds\right)
\\
\leq
\int_{\sinterout_{ij}(t,x,\xi)}^{\sinterout_{i, m+\ell}\left(\sinterout_{ij}(t,x,\xi),\chi_i\left(\sinterout_{ij}(t,x,\xi);t,x\right), 0\right)} -\frac{d\theta}{ds}(s) \, ds
\\
=-\theta\left(\sinterout_{i, m+\ell}\left(\sinterout_{ij}(t,x,\xi),\chi_i\left(\sinterout_{ij}(t,x,\xi);t,x\right), 0\right)\right)+\theta\left(\sinterout_{ij}(t,x,\xi)\right)
\leq \theta(t)
=\Omega_i(t,x,\xi).
\end{multline*}

This gives the estimate$$\abs{\left(\Phi_2(K)-\Phi_2(H)\right)_{ij}(t,x,\xi)}
\leq C_2 \Omega_i(t,x,\xi)\norm{K-H}_{L^{\infty}(\Taub)^{m \times n}}.$$
Note that this estimate is also valid if $j \not\in \ens{i,\ldots,m}$ or $(t,x,\xi) \not\in \clos{\SSoutm_{ij}}$ since $(\Phi_2(\cdot))_{ij}=0$ in this case.
Reasoning by induction as before, it is now not difficult to obtain the estimate
$$
\norm{\Phi_2^N(K)-\Phi_2^N(H)}_{L^{\infty}(\Taub)^{m \times n}}
\leq
\frac{C^N}{N!}
\norm{K-H}_{L^{\infty}(\Taub)^{m \times n}},
$$
for some $C$ independent of $N$ and $K,H$.

\end{enumerate}
\end{proof}

\subsection{On the time-periodicity of $F$}
\label{sec.per}
In this section we assume that, for some $\tau>0$,
$\Lambda$, $M$ and $Q$ are $\tau$-periodic with respect to time and show that the above construction of $F$ leads, with minor modifications, to a $F$ which is also $\tau$-periodic with respect to time.

First of all, concerning the extension of $\Lambda$ to a function of $\R^2$ (and of $M$ and $Q$ whenever needed), it is clear that one can extend $\Lambda$ to $\mathbb{R}\times [0,1]$ by just requiring the
$\tau$-periodicity with respect to time of this extension.
This extension is still denoted by $\Lambda$.
Then one extends $\Lambda$ to $\R^2$ so that this extension, still denoted by $\Lambda$,  is $\tau$-periodic with respect to time and so that the properties \eqref{Lambda diag}, \eqref{hyp speeds}, \eqref{hyp speeds bis} and \eqref{reg parameters} remain valid on $\R^2$ (see e.g. Remark \ref{rem ext Lambda}).

From the construction of $F$ (see \eqref{def F}, \eqref{def F1} and \eqref{def F2}) it is clear that $F$ is $\tau$-periodic with respect to time if so are all the matrix-valued functions involved in the several transformations of this article.
Now, in order to obtain the $\tau$-periodicity of these transformations, the minor modifications/comments are essentially the following ones.

\begin{enumerate}[1)]
\item
Concerning the diagonal transformation to remove the diagonal terms in $M$
(see Section~\ref{sect remov diag}), one simply observes that the function \eqref{sol equ phi_i} is $\tau$-periodic with respect to time if so is $m_{ii}$, thanks to the properties
\begin{equation}\label{prop tauper}
\chi_i(s+\tau;t+\tau,\xi)=\chi_i(s;t,\xi),
\qquad
\ssin_i(t+\tau,\xi)=\ssin_i(t,\xi)+\tau.
\end{equation}

\item
Concerning the kernel $H$ of the Fredholm transformation (see Section~\ref{sect exist kern T3}) one easily checks that it is indeed $\tau$-periodic with respect to time.
This follows from the uniqueness of the solution to \eqref{valuehij} and similar properties to \eqref{prop tauper}.

\item
Concerning the kernel $K$ of the Volterra transformation of the second kind (see Section~\ref{sect exist kern T2}), to construct it in such a way that it is $\tau$-periodic with respect to time, it suffices to observe that $\widetilde{m}_{\ell j}^1$, $r_{ij}$ and $q^1_{\ell j}$ become $\tau$-periodic with respect to time once $M$ and $Q$ are, and to modify the definition of the space $B$ given in \eqref{defB} by adding the condition that the $k_{ij}$, $1 \leq i \leq m$, $1 \leq j \leq n$, are $\tau$-periodic with respect to time (alternatively, one can keep \eqref{defB} and deduce from the uniqueness of the fixed point of $\Phi$ that it has to be $\tau$-periodic with respect to time).

\end{enumerate}

\section*{Acknowledgements}
All the authors would like to thank ETH Z\"{u}rich Institute for Theoretical Studies (ETH-ITS) and Institute for Mathematical Research (ETH-FIM) for their hospitalities. This work was initiated while they were visiting  there.
The second and third author would also thank Tongji University for the kind invitation.
Part of this work was also carried out there.
This project was partially supported by ANR Finite4SoS ANR-15-CE23-0007, the Natural Science Foundation of China (Nos. 11601284 and 11771336) and the Young Scholars Program of Shandong University (No. 2016WLJH52).

\appendix
\section{Background on broad solutions}\label{sect broad sol}

We recall that all the systems of this paper have the following form:
\begin{equation}\label{system general bis}
\left\{\begin{array}{l}
\ds \pt{y}(t,x)+\Lambda(t,x) \px{y}(t,x)=M(t,x) y(t,x)+G(t,x)y(t,0), \\
y_-(t,1)=\ds \int_0^1 F(t,\xi) y(t,\xi) \, d\xi, \quad y_+(t,0)=Q(t)y_-(t,0),  \\
y(t^0,x)=y^0(x),
\end{array}\right.
\end{equation}
where $M$ and $Q$ have at least the regularity \eqref{reg parameters}, $F \in L^\infty((0,+\infty) \times (0,1))^{m\times n}$ and
$$G \in C^0([0,+\infty)\times[0,1])^{n \times n} \cap L^{\infty}((0,+\infty)\times(0,1))^{n \times n}.$$

\subsection{Definition of broad solution}

Let us now introduce the notion of solution for such systems.
To this end, we have to restrict our discussion to the domain where the system \eqref{system general bis} evolves, i.e. on $(t^0,+\infty)\times(0,1)$.
For every $(t,x) \in (t^0,+\infty)\times(0,1)$, we have
$$(s,\chi_i(s;t,x)) \in (t^0,+\infty)\times(0,1), \quad \forall s \in (\ssinb_i(t^0;t,x),\ssout_i(t,x)),$$
where we introduced
$$\ssinb_i(t^0;t,x)=\max\ens{t^0,\ssin_i(t,x)}<t.$$

\noindent
Formally, writing the $i$-th equation of the system \eqref{system general bis} along the characteristic $\chi_i(s;t,x)$ for $s \in [\ssinb_i(t^0;t,x),\ssout_i(t,x)]$, and using the chain rules yields the ODE
\begin{equation}\label{sol along char}
\left\{\begin{array}{l}
\ds \dds y_i\left(s,\chi_i(s;t,x)\right)=
\sum_{j=1}^n m_{ij}\left(s,\chi_i(s;t,x)\right)y_j\left(s,\chi_i(s;t,x)\right)
+\sum_{j=1}^n g_{ij}\left(s,\chi_i(s;t,x)\right)y_j\left(s,0\right), \\
\ds y_i\left(\ssinb_i(t^0;t,x),\chi_i(\ssinb_i(t^0;t,x);t,x)\right)=b_i(y)(t,x),
 \end{array}\right.
\end{equation}
where the initial condition $b_i(y)(t,x)$ is given by the appropriate boundary or initial conditions of the system \eqref{system general bis}:
\begin{multline}\label{def bi}
b_i(y)(t,x)=
\left\{\begin{array}{ll}
\ds \sum_{j=1}^n \int_0^1 f_{ij}(\ssin_i(t,x),\xi)y_j(\ssin_i(t,x),\xi) \, d\xi & \mbox{ if } \ssin_i(t,x)>t^0 \mbox{ and } i \in \ens{1,\ldots,m}, \\
\ds \sum_{j=1}^m q_{i-m,j}(\ssin_i(t,x))y_j(\ssin_i(t,x),0)& \mbox{ if } \ssin_i(t,x)>t^0 \mbox{ and } i \in \ens{m+1,\ldots,n}, \\
\ds y^0_i(\chi_i(t^0;t,x)) & \mbox{ if } \ssin_i(t,x)<t^0.
\end{array}\right.
\end{multline}
Integrating the ODE \eqref{sol along char} over $s \in [\ssinb_i(t^0;t,x),t]$, we obtain the following system of integral equations:
\begin{multline}\label{equdef sol broad}
y_i(t,x)=b_i(y)(t,x)+\sum_{j=1}^n \int_{\ssinb_i(t^0;t,x)}^{t} m_{ij}(s,\chi_i(s;t,x))y_j(s,\chi_i(s;t,x)) \, ds
\\
+\sum_{j=1}^n \int_{\ssinb_i(t^0;t,x)}^t g_{ij}\left(s,\chi_i(s;t,x)\right)y_j\left(s,0\right) \, ds.
\end{multline}
This leads to the following notion of ``solution along the characteristics'' or ``broad solution'':

\begin{definition}\label{def broad sol}
Let $t^0 \geq 0$ and $y^0 \in L^2(0,1)^n$ be fixed.
We say that a function $y:(t^0,+\infty)\times(0,1) \longrightarrow \R^n$ is a broad solution to the system \eqref{system general bis} if
\begin{equation}\label{reg broad sol}
y \in C^0([t^0,t^0+T];L^2(0,1)^n)
\cap C^0([0,1];L^2(t^0,t^0+T)^n), \quad \forall T>0,
\end{equation}
and if the integral equation \eqref{equdef sol broad} is satisfied for every $i \in \ens{1,\ldots,n}$, for a.e. $t>t^0$ and a.e. $x \in (0,1)$.
\end{definition}

\subsection{Well-posedness}

This section is devoted to the following well-posedness result  regarding system \eqref{system general bis}:

\begin{theorem}\label{thm wp}
For every $t^0 \geq 0$ and $y^0 \in L^2(0,1)^n$, there exists a unique broad solution to \eqref{system general bis}.
Moreover, there exists $C>0$ such that, for every $T>0$, $t^0 \geq 0$ and $y^0 \in L^2(0,1)^n$, the corresponding broad solution $y$ satisfies
\begin{equation}\label{dep cont}
\norm{y}_{C^0([t^0,t^0+T];L^2(0,1)^n)}
+\norm{y}_{C^0([0,1];L^2(t^0,t^0+T)^n)}
\leq Ce^{CT} \norm{y^0}_{L^2(0,1)^n}.
\end{equation}
\end{theorem}

\begin{remark}\label{rem FTGA implies US}
It follows from the uniformity of the constant $Ce^{CT}$ with respect to the initial time $t^0$ in the estimate \eqref{dep cont} that, for systems of the form \eqref{system general bis}, the uniform stability property \eqref{FT cond} is a consequence of the finite-time global attractor property \eqref{attracT} (simply take $\delta>0$ such that $Ce^{CT}\delta \leq \epsilon$).
\end{remark}

Let us first point out that this well-posedness result for our initial system \eqref{syst init} for the particular $F$ that we have constructed in Section \ref{sect syst transfo} follows in fact from the well-posedness result for the final target system of Proposition \ref{prop transfo 3} (easier to establish), since we have shown that both systems are equivalent by means of several invertible transformations.
However, it is still important to have such a well-posedness result for any $F$ within the class studied, which is a result that also has its own interest.
We will provide a complete proof since, to the best of our knowledge, there are no references that show the well-posedness for the initial-boundary value problem \eqref{system general bis} with non-local terms $G(t,x)y(t,0)$, with weak regularity \eqref{reg broad sol} and with uniform estimate \eqref{dep cont}.

\begin{proof}[Proof of Theorem \ref{thm wp}]

We first remark that it is enough to prove the theorem for $\norm{Q}_{L^{\infty}}$ small enough, say
\begin{equation}\label{cond on Q}
\norm{Q}_{L^{\infty}} \leq \alpha,
\end{equation}
where $\alpha>0$ does not depend on $T,t^0,y^0$ nor on $M,G,F$.
This follows from the following change of variable:
$$y=D\tilde{y},
\qquad D=\begin{pmatrix} \frac{\alpha}{\norm{Q}_{L^\infty}+\alpha} \Id_{\R^m} & 0 \\ 0 & \Id_{\R^p} \end{pmatrix},
$$
where $\tilde{y}$ is the solution to the system $(\tilde{M},\tilde{G},\tilde{F},\tilde{Q})$ with
$$
\tilde{M}=D^{-1}MD,
\quad
\tilde{G}=D^{-1}GD,
\quad
\tilde{F}=\left(\frac{\alpha}{\norm{Q}_{L^\infty}+\alpha}\right)^{-1}FD,
\quad
\tilde{Q}=\frac{\alpha}{\norm{Q}_{L^\infty}+\alpha} Q.
$$

Let us now show how to prove the theorem under the smallness condition \eqref{cond on Q} with the Banach fixed point theorem ($\alpha>0$ will be fixed adequately below).
Let $T>0$, $t^0 \geq 0$ and $y^0 \in L^2(0,1)^n$ be fixed for the remainder of the proof.
It is clear that a function $y:(t^0,+\infty)\times(0,1) \longrightarrow \R^n$ satisfies the integral equation \eqref{equdef sol broad} if, and only if, it is a fixed point of the map $\mathcal{F}:B \longrightarrow B$, where
$$B=C^0([t^0,t^0+T];L^2(0,1)^n) \cap C^0([0,1];L^2(t^0,t^0+T)^n),$$
and $(\mathcal{F}(y))_i(t,x)$ is given by the expression on the right-hand side of \eqref{equdef sol broad}.
It can be checked that $\mathcal{F}$ indeed maps $B$ into itself (actually, by computations similar to the upcoming ones).
Let us now make $B$ a Banach space by equipping it with the following weighted norm:
$$\norm{y}_B=\norm{y}_{B_1}+\norm{y}_{B_2},$$
where
$$
\norm{y}_{B_1}=\max_{t \in [t^0,t^0+T]} e^{-\frac{L_1}{2}(t-t^0)} \sqrt{\int_0^1 \sum_{i=1}^n \abs{y_i(t,x)}^2 e^{-L_2 x} \, dx},
$$
and
$$
\norm{y}_{B_2}=\max_{x \in [0,1]} e^{\frac{L_2}{2}(1-x)} \sqrt{\int_{t^0}^{t^0+T} \sum_{i=1}^n \abs{y_i(t,x)}^2 e^{-L_1(t-t^0)} \, dt},
$$
where $L_1,L_2>0$ are constants independent of $T,t^0$ and $y^0$ that will be fixed below.
Our goal is to show that, for $L_1,L_2>0$ large enough,
\begin{equation}\label{F contraction}
\norm{\mathcal{F}(y^1)-\mathcal{F}(y^2)}_B
\leq \frac{1}{2} \norm{y^1-y^2}_B, \quad \forall y^1,y^2 \in B.
\end{equation}
It is then not difficult to check that the fixed point of $\mathcal{F}$ satisfies the estimate \eqref{dep cont}.
Indeed, using \eqref{F contraction}, we easily see that the fixed point $y$ of $\mathcal{F}$ will satisfy
$$\frac{1}{2}\norm{y}_B \leq \norm{\mathcal{F}(0)}_B,$$
and some straightforward computations show that
$$
\begin{array}{c}
\ds
\norm{y}_{C^0([t^0,t^0+T];L^2(0,1)^n)}^2
\leq e^{L_2} e^{L_1 T} \norm{y}_{B_1}^2,
\qquad
\norm{y}_{C^0([0,1];L^2(t^0,t^0+T)^n)}^2
\leq e^{L_1 T} \norm{y}_{B_2}^2,
\\
\ds
\norm{\mathcal{F}(0)}_B^2 \leq 2\left(1+\frac{e^{L_2}}{\epsilon}\right)e^{\norm{\px{\Lambda}}_{L^\infty} T}\norm{y^0}_{L^2(0,1)^n}^2.
\end{array}
$$

\noindent
Let us now establish \eqref{F contraction}.
We introduce
$$y=y^1-y^2,$$
so that $\mathcal{F}(y^1)-\mathcal{F}(y^2)$ is equal to the right-hand side of \eqref{equdef sol broad} with $y^0=0$.
We have to estimate four types of terms in each $\norm{\cdot}_{B_i}$-norm ($i=1,2$).
For convenience, we denote by
$$
R_1=\max \ens{
\norm{\Lambda}_{L^{\infty}},
\norm{\px{\Lambda}}_{L^{\infty}}
},
\quad
R_2=\max \ens{
\norm{M}_{L^{\infty}},
\norm{G}_{L^{\infty}},
\norm{F}_{L^{\infty}}
}.
$$
We recall that it is crucial that $\alpha$ does not depend on $R_2$.

\paragraph{Estimate of the $\norm{\cdot}_{B_1}$-norm.}

Let $t \in [t^0,t^0+T]$ be fixed.
Let $I=\ens{x \in (0,1), \quad \ssin_i(t,x)>t^0}$.

\begin{enumerate}[1)]
\item
Let $i \in \ens{1,\ldots,m}$.
For a.e. $x \in I$, using Cauchy-Schwarz inequality, we have
$$
\abs{\sum_{j=1}^n \int_0^1 f_{ij}(\ssin_i(t,x),\xi)y_j(\ssin_i(t,x),\xi) \, d\xi}^2
\leq n R_2^2 e^{L_2} \norm{y}_{B_1}^2 e^{L_1(\ssin_i(t,x)-t^0)}.
$$

Using a finer version of \eqref{tentredeux}, namely,
\begin{equation}\label{tentredeux finer}
\frac{1-x}{R_1} \leq t-\ssin_i(t,x),
\end{equation}
(obtained similarly to \eqref{bounds sin sout}) we obtain the estimate
$$
\int_{I} \abs{\sum_{j=1}^n \int_0^1 f_{ij}(\ssin_i(t,x),\xi)y_j(\ssin_i(t,x),\xi) \, d\xi}^2 e^{-L_2 x} \, dx
\leq \left(n R_2^2\frac{1}{\frac{L_1}{R_1}-L_2}\right) e^{L_1(t-t^0)} \norm{y}_{B_1}^2,
$$
provided that
\begin{equation}\label{cond L1L2}
\frac{L_1}{R_1}-L_2>0.
\end{equation}

\item
Let $i \in \ens{m+1,\ldots,n}$.
Using \eqref{tentredeux}, we have
\begin{multline*}
\int_{I}
\abs{\sum_{j=1}^m q_{i-m,j}(\ssin_i(t,x))y_j(\ssin_i(t,x),0)}^2
e^{-L_2 x} \, dx
\\
\leq m \alpha^2 e^{L_1(t-t^0)}
\sum_{j=1}^m
\int_{I}
\abs{y_j(\ssin_i(t,x),0)}^2
e^{-L_1(\ssin_i(t,x)-t^0)}
\, dx
.
\end{multline*}

Doing the change of variables $\sigma=\ssin_i(t,x)$ and using the estimate  (see \eqref{def sin sout}, \eqref{deriv chi} and \eqref{bounds sin sout})
$$
\px{\ssin_i(t,x)}
=\frac{-e^{-\int_{\ssin_i(t,x)}^t \px{\lambda_i}(\theta,\chi_i(\theta;t,x)) \, d\theta}}{\lambda_i(\ssin_i(t,x),0)}
\leq -\frac{1}{R_1} e^{-\frac{R_1}{\epsilon}},
$$
we obtain
$$
\int_{I}
\abs{\sum_{j=1}^m q_{i-m,j}(\ssin_i(t,x))y_j(\ssin_i(t,x),0)}^2
e^{-L_2 x} \, dx
\leq
\left(m \alpha^2 R_1 e^{\frac{R_1}{\epsilon}}e^{-L_2} \right)
e^{L_1(t-t^0)}
\norm{y}_{B_2}^2.
$$

\item
For the next term, we have
\begin{multline*}
\int_0^1 \sum_{i=1}^n
\abs{\sum_{j=1}^n \int_{\ssinb_i(t^0;t,x)}^{t} m_{ij}(s,\chi_i(s;t,x))y_j(s,\chi_i(s;t,x)) \, ds}^2 e^{-L_2 x} \, dx
\\
\leq
nR_2^2\frac{1}{\epsilon}\sum_{i=1}^n
\int_0^1
\int_{\ssinb_i(t^0;t,x)}^{t} \sum_{j=1}^n \abs{y_j(s,\chi_i(s;t,x))}^2 \, ds \, dx.
\end{multline*}

Using the change of variable $(\sigma,\xi)=(s,\chi_i(s;t,x))$, whose Jacobian determinant is (see \eqref{deriv chi})
$$\det \begin{pmatrix} 1 & 0 \\ \lambda_i(s,\chi_i(s;t,x)) & \px{\chi_i}(s;t,x) \end{pmatrix}
=e^{-\int_s^t \px{\lambda_i}(\theta,\chi_i(\theta;t,x)) \, d\theta}
\geq e^{-\frac{R_1}{\epsilon}}, \quad \forall s \in (\ssin_i(t,x),t),
$$
we obtain
\begin{multline*}
\int_0^1 \sum_{i=1}^n
\abs{\sum_{j=1}^n \int_{\ssinb_i(t^0;t,x)}^{t} m_{ij}(s,\chi_i(s;t,x))y_j(s,\chi_i(s;t,x)) \, ds}^2 e^{-L_2 x} \, dx
\\
\leq
\left(
n^2 R_2^2\frac{1}{\epsilon}e^{\frac{R_1}{\epsilon}}e^{L_2}\frac{1}{L_1}
\right)
e^{L_1(t-t^0)}
\norm{y}_{B_1}^2.
\end{multline*}

\item
Finally, the estimate of the remaining term is easy:
\begin{multline*}
\int_0^1 \sum_{i=1}^n\abs{\sum_{j=1}^n \int_{\ssinb_i(t^0;t,x)}^t g_{ij}\left(s,\chi_i(s;t,x)\right)y_j\left(s,0\right) \, ds}^2 e^{-L_2 x} \, dx
\\
\leq
\left(
n^2R_2^2\frac{1}{\epsilon}e^{-L_2}
\right)
e^{L_1(t-t^0)} \norm{y}_{B_2}^2.
\end{multline*}
\end{enumerate}

\noindent
In summary, we have established the following estimate (provided that \eqref{cond L1L2} holds):
\begin{multline}\label{estim norm B1}
\norm{\mathcal{F}(y^1)-\mathcal{F}(y^2)}_{B_1}^2 \leq
3
\left(
m n R_2^2\frac{1}{\frac{L_1}{R_1}-L_2}
+n^2 R_2^2\frac{1}{\epsilon}e^{\frac{R_1}{\epsilon}}e^{L_2}\frac{1}{L_1}
\right)\norm{y}_{B_1}^2
\\
+3\left(
(n-m)m \alpha^2 R_1 e^{\frac{R_1}{\epsilon}}e^{-L_2}
+n^2R_2^2\frac{1}{\epsilon}e^{-L_2}
\right)\norm{y}_{B_2}^2
.
\end{multline}

\paragraph{Estimate of the $\norm{\cdot}_{B_2}$-norm.}

Let $x \in [0,1]$ be fixed.
Let $J=\ens{t \in (t^0,t^0+T), \quad \ssin_i(t,x)>t^0}$.

\begin{enumerate}[1)]
\item
Let $i \in \ens{1,\ldots,m}$.
We have
\begin{multline*}
\int_J \abs{\sum_{j=1}^n \int_0^1 f_{ij}(\ssin_i(t,x),\xi)y_j(\ssin_i(t,x),\xi) \, d\xi}^2
e^{-L_1(t-t^0)} \, dt
\\
\leq n R_2^2 \int_0^1 \left(
\int_J \sum_{j=1}^n \abs{y_j(\ssin_i(t,x),\xi)}^2 e^{-L_1(\ssin_i(t,x)-t^0)} e^{L_1(\ssin_i(t,x)-t)}\, dt
\right) \, d\xi.
\end{multline*}

Using once again \eqref{tentredeux finer}, \eqref{cond L1L2}, performing the change of variable $\sigma=\ssin_i(t,x)$, and using the estimate  (see \eqref{def sin sout}, \eqref{deriv chi} and \eqref{bounds sin sout})
\begin{equation}\label{estim jacob wrt t}
\pt{\ssin_i(t,x)}
=\frac{\lambda_i(t,x)e^{-\int_{\ssin_i(t,x)}^t \px{\lambda_i}(\theta,\chi_i(\theta;t,x)) \, d\theta}}{\lambda_i(\ssin_i(t,x),1)}
\geq \frac{\epsilon}{R_1} e^{-\frac{R_1}{\epsilon}},
\end{equation}

we obtain the estimate
$$
\int_J \abs{\sum_{j=1}^n \int_0^1 f_{ij}(\ssin_i(t,x),\xi)y_j(\ssin_i(t,x),\xi) \, d\xi}^2
e^{-L_1(t-t^0)} \, dt
\leq \left(
n \frac{R_1}{\epsilon} R_2^2 e^{\frac{R_1}{\epsilon}}\frac{1}{L_2}
\right)e^{-L_2(1-x)}\norm{y}_{B_2}^2.
$$

\item
The next estimate is where we will need the smallness assumption on $Q$.
Let $i \in \ens{m+1,\ldots,n}$.
Using \eqref{tentredeux} and the change of variables $\sigma=\ssin_i(t,x)$ (recall the estimate \eqref{estim jacob wrt t}), we obtain
$$
\int_{J}
\abs{\sum_{j=1}^m q_{i-m,j}(\ssin_i(t,x))y_j(\ssin_i(t,x),0)}^2
e^{-L_1(t-t^0)} \, dt
\leq
\left(
m\alpha^2\frac{R_1}{\epsilon}e^{\frac{R_1}{\epsilon}}
\right) e^{-L_2(1-x)} \norm{y}_{B_1}^2.
$$

\item
For $i \in \ens{1,\ldots,m}$, using the estimate
$$-R_1(t-s) \leq x-\chi_i(s;t,x),$$
we have
\begin{multline*}
\abs{\sum_{j=1}^n \int_{\ssinb_i(t^0;t,x)}^{t} m_{ij}(s,\chi_i(s;t,x))y_j(s,\chi_i(s;t,x)) \, ds}^2
e^{-L_1(t-t^0)}
\\
\leq
nR_2^2\frac{1}{\epsilon}
\int_{\ssinb_i(t^0;t,x)}^{t} \sum_{j=1}^n \abs{y_j(s,\chi_i(s;t,x))}^2 e^{-L_1(s-t^0)} e^{-\frac{L_1}{R_1}(\chi_i(s;t,x)-x)} \, ds.
\end{multline*}
Integrating and using the change of variable $(\sigma,\xi)=(s,\chi_i(s;t,x))$, whose Jacobian determinant is uniformly estimated for $s \in (\ssin_i(t,x),t)$ by (see \eqref{deriv chi} and \eqref{bounds sin sout})
$$\abs{\det \begin{pmatrix} 1 & 0 \\ \lambda_i(s,\chi_i(s;t,x)) & \pt{\chi_i}(s;t,x) \end{pmatrix}}
=\abs{\lambda_i(t,x)}e^{-\int_s^t \px{\lambda_i}(\theta,\chi_i(\theta;t,x)) \, d\theta}
\geq \epsilon e^{-\frac{R_1}{\epsilon}},
$$
we obtain (using also that $x \leq \chi_i(s;t,x) \leq 1$)
\begin{multline*}
\int_{t^0}^{t^0+T}\sum_{i=1}^m
\abs{\sum_{j=1}^n \int_{\ssinb_i(t^0;t,x)}^{t} m_{ij}(s,\chi_i(s;t,x))y_j(s,\chi_i(s;t,x)) \, ds}^2
e^{-L_1(t-t^0)} \, dt
\\
\leq
mn R_2^2 \frac{1}{\epsilon^2}e^{\frac{R_1}{\epsilon}}
\int_x^1
\left(\int_{t^0}^{t^0+T} \sum_{j=1}^n \abs{y_j(\sigma,\xi)}^2
e^{-L_1(\sigma-t^0)} \, d\sigma \right)
e^{-\frac{L_1}{R_1}(\xi-x)}
\, d\xi
\\
\leq
\left(
mn R_2^2 \frac{1}{\epsilon^2}e^{\frac{R_1}{\epsilon}}\frac{1}{\frac{L_1}{R_1}-L_2}
\right)
e^{-L_2(1-x)}\norm{y}_{B_2}^2,
\end{multline*}
provided that \eqref{cond L1L2} holds.
A similar reasoning shows that
\begin{multline*}
\int_{t^0}^{t^0+T}\sum_{i=m+1}^n
\abs{\sum_{j=1}^n \int_{\ssinb_i(t^0;t,x)}^{t} m_{ij}(s,\chi_i(s;t,x))y_j(s,\chi_i(s;t,x)) \, ds}^2
e^{-L_1(t-t^0)} \, dt
\\
\leq
\left(
(n-m)n R_2^2 \frac{1}{\epsilon^2}e^{\frac{R_1}{\epsilon}}\frac{1}{L_2+\frac{L_1}{R_1}}
\right)
e^{-L_2(1-x)}\norm{y}_{B_2}^2.
\end{multline*}

\item
For the remaining term, using a similar reasoning to the one used in the previous step, we obtain
\begin{multline*}
\int_{t^0}^{t^0+T}\sum_{i=1}^n
\abs{\sum_{j=1}^n \int_{\ssinb_i(t^0;t,x)}^{t} g_{ij}(s,\chi_i(s;t,x))y_j(s,0) \, ds}^2
e^{-L_1(t-t^0)} \, dt
\\
\leq
\left(
n^2 R_1 R_2^2 \frac{1}{\epsilon^2}e^{\frac{R_1}{\epsilon}}\frac{1}{L_1}
\right)
e^{-L_2(1-x)}\norm{y}_{B_2}^2.
\end{multline*}

\end{enumerate}

\noindent
In summary, we have established the following estimate (provided that \eqref{cond L1L2} holds):
\begin{multline}\label{estim norm B2}
\norm{\mathcal{F}(y^1)-\mathcal{F}(y^2)}_{B_2}^2 \leq
3
\left(
(n-m)m\alpha^2\frac{R_1}{\epsilon}e^{\frac{R_1}{\epsilon}}
\right)\norm{y}_{B_1}^2
\\
+3\Bigg(
mn \frac{R_1}{\epsilon} R_2^2 e^{\frac{R_1}{\epsilon}}\frac{1}{L_2}
+mn R_2^2 \frac{1}{\epsilon^2}e^{\frac{R_1}{\epsilon}}\frac{1}{\frac{L_1}{R_1}-L_2}	+(n-m)n R_2^2 \frac{1}{\epsilon^2}e^{\frac{R_1}{\epsilon}}\frac{1}{L_2+\frac{L_1}{R_1}}
\\
+n^2 R_1 R_2^2 \frac{1}{\epsilon^2}e^{\frac{R_1}{\epsilon}}\frac{1}{L_1}
\Bigg)\norm{y}_{B_2}^2
.
\end{multline}

Consequently, we see from \eqref{estim norm B1} and \eqref{estim norm B2} that $\mathcal{F}$ indeed satisfies the contraction property \eqref{F contraction} if $\alpha$ is small enough (depending only on $n-m, m, R_1$ and $\epsilon$) and if we fix $L_2>0$ and then $L_1>0$ large enough.
This concludes the proof of Theorem \ref{thm wp}.

\end{proof}

\begin{remark}\label{rem weak sol}
It can be shown that the broad solution is also the classical solution if the data of the system are smooth enough.
It then follows by standard approximation arguments that the broad solution is also the so-called weak solution.
We recall that the notion of weak solution for \eqref{system general bis} is obtained by multiplying \eqref{system general bis} by a smooth function and integrating by parts, that is, a function $y:(t^0,+\infty)\times(0,1) \longrightarrow \R^n$ is a weak solution to \eqref{system general bis} if $y \in C^0([t^0,+\infty);L^2(0,1)^n)$ and if it satisfies:
\begin{multline}\label{equ def sol}
\int_0^1 y(t^0+T,x) \cdot \varphi(t^0+T,x) \, dx
-\int_0^1 y^0(x) \cdot \varphi(t^0,x) \, dx \\
+\int_{t^0}^{t^0+T} \int_0^1 y(t,x) \cdot
\left(
-\pt{\varphi}(t,x)
-\Lambda(t,x)\px{\varphi}(t,x)
-\left(\px{\Lambda}(t,x)+\tr{M(t,x)}\right)\varphi(t,x)
\right) \, dx dt \\
+\int_{t^0}^{t^0+T} \int_0^1 y(t,\xi) \cdot
\tr{F(t,\xi)} \Lambda_{--}(t,1)\varphi_-(t,1)
 \, d\xi dt
=0,
\end{multline}
for every $T>0$ and every $\varphi \in C^1([t^0,t^0+T]\times[0,1])^n$ such that, for every $t \in [t^0,t^0+T]$,
$$
\begin{array}{c}
\ds \varphi_+(t,1)=0,
\\
\ds \varphi_-(t,0)
=-\Lambda_{--}(t,0)^{-1}\left(\tr{Q(t)}\Lambda_{++}(t,0)\varphi_+(t,0)
+\begin{pmatrix} \Id_{\R^m} & \tr{Q(t)} \end{pmatrix} \int_0^1 \tr{G(t,x)} \varphi(t,x) \, dx
\right).
\end{array}
$$
In \eqref{equ def sol}, we denoted by $\tr{A}$ the transpose of a matrix $A$ and $v_1 \cdot v_2$ denotes the canonical scalar product between two vectors $v_1,v_2$ of $\R^n$.
\end{remark}

\subsection{Justification of the formal computations}\label{sect justification}

In this section, we finally rigorously prove that the transformations that we used all along this paper are preserving broad solutions.
We show how it works only for the Fredholm transformation of Section \ref{sect stabi S2} (because it is simpler to present) but the reasoning is general and can be used for the Volterra transformation of Section \ref{reduc to triang} as well.
More precisely, the goal of this section is to prove the following result:

\begin{proposition}\label{prop broad to broad}
Let $H_{--}=(h_{ij})_{1 \leq i,j \leq m}$, where $h_{ij}$ is the solution to the differential equation \eqref{ODE kerh}.
Let $t^0 \geq 0$ be fixed.
Let $z^0 \in L^2(0,1)^n$ and let $z$ be the broad solution to
$$
\left\{\begin{array}{l}
\ds \pt{z}(t,x)+\Lambda(t,x) \px{z}(t,x)=G^3(t,x)z(t,0), \\
z_-(t,1)=0, \quad z_+(t,0)=Q^1(t)z_-(t,0),  \\
z(t^0,x)=z^0(x).
\end{array}\right.
$$
Then, the function $\gamma$ defined by the Fredholm transformation \eqref{transfo fred} is the broad solution to
$$
\left\{\begin{array}{l}
\ds \pt{\gamma}(t,x)+\Lambda(t,x) \px{\gamma}(t,x)=G^2(t,x)\gamma(t,0), \\
\gamma_-(t,1)=\ds \int_0^1 F^2(t,\xi) \gamma(t,\xi) \, d\xi, \quad \gamma_+(t,0)=Q^1(t)\gamma_-(t,0),  \\
\gamma(t^0,x)=\gamma^0(x),
\end{array}\right.
$$
where $\gamma^0(x)=z^0(x)-\int_0^1 H(t^0,x,\xi)z^0(\xi) \,d\xi$.
\end{proposition}

We recall that $H$ is given by \eqref{def ker H}, $F^2$ is the solution of \eqref{def F2}, $Q^1$ is provided by Proposition \ref{prop transfo 1}, $G^2$ is provided by Proposition \ref{prop transfo 2} and, finally, $G^3$ is given by \eqref{def G3}.

\begin{remark}
Obviously, we could use the explicit expression \eqref{kij Fred}-\eqref{explicit bij} of the solution $H$ to simplify the forthcoming arguments but we choose not to do so and to only use the differential equation \eqref{ODE kerh} in order to give a general procedure that can also be used to justify the formal computations of Section \ref{reduc to triang} as well.
\end{remark}

A similar result to Proposition \ref{prop broad to broad} can be found in \cite[Proposition 3.5]{CN19}.
Here we propose a different and self-contained proof, based on the following characterization of broad solutions:

\begin{lemma}\label{prop charact broad sol}
A function $y:(t^0,+\infty)\times(0,1) \longrightarrow \R^n$ is the broad solution to \eqref{system general bis} if, and only if, $y$ has the regularity \eqref{reg broad sol} and, for every $i \in \ens{1,\ldots,n}$, for a.e. $t>t^0$ and a.e. $x \in (0,1)$, the function $s \mapsto y_i\left(s,\chi_i(s;t,x)\right)$ belongs to $H^1(\ssinb_i(t^0;t,x),\ssout_i(t,x))$ and it satisfies the ODE \eqref{sol along char}.
\end{lemma}

The proof of Lemma \ref{prop charact broad sol} is not difficult, it simply relies on the properties \eqref{chi invariance} and \eqref{invariance}.

\begin{proof}[Proof of Proposition \ref{prop broad to broad}]
~
\begin{enumerate}[1)]

\item
The required regularity
$$
\gamma \in C^0([t^0,t^0+T];L^2(0,1)^n)
\cap C^0([0,1];L^2(t^0,t^0+T)^n), \quad \forall\, T>0,
$$
is clear since $z$ also has this regularity and $
(t,x) \mapsto \int_0^1 H_{--}(t,x,\xi)z_-(t,\xi) \,d\xi$ is continuous (see e.g. Remark \ref{rem broad reg int term}).

\item
The initial condition in the ODE formulation
$$\gamma_i\left(\ssinb_i(t^0;t,x),\chi_i(\ssinb_i(t^0;t,x);t,x)\right)=b_i(\gamma)(t,x)$$
is not difficult to check by using the boundary condition $z_-(t,1)=0$ with the definition \eqref{def F2} of $F^2$ and Fubini's theorem (case $\ssin_i(t,x)>t^0$ and $i \in \ens{1,\ldots,m}$), the condition $H_{--}(t,0,\xi)=0$ (case $\ssin_i(t,x)>t^0$ and $i \in \ens{m+1,\ldots,n}$) and the definition of $\gamma^0$ (case $\ssin_i(t,x)<t^0$).

\item
It remains to check that, for every $i \in \ens{1,\ldots,n}$, for a.e. $t>t^0$ and $x \in (0,1)$, the function $s \mapsto \gamma_i\left(s,\chi_i(s;t,x)\right)$ belongs to $H^1(\ssinb_i(t^0;t,x),\ssout_i(t,x))$ with
\begin{equation}\label{dds gamma}
\ds \dds \gamma_i\left(s,\chi_i(s;t,x)\right)=
\sum_{j=1}^m g^2_{ij}\left(s,\chi_i(s;t,x)\right)\gamma_j\left(s,0\right).
\end{equation}

By definition \eqref{transfo fred} of $\gamma$, we have
$$\gamma_i\left(s,\chi_i(s;t,x)\right)=z_i\left(s,\chi_i(s;t,x)\right)-\sum_{j=1}^m\int_0^1 h_{ij}(s,\chi_i(s;t,x),\xi)z_j(s,\xi) \, d\xi.$$

For $i \in \ens{m+1,\ldots,n}$, the identity \eqref{dds gamma} easily follows from the equation satisfied by $z_i$, the relation $z(\cdot,0)=\gamma(\cdot,0)$, and the fact that $h_{ij}=0$ for such indices (recall \eqref{def ker H}).

Let us now assume that $i \in \ens{1,\ldots,m}$.
The equation satisfied by $z_i$ then gives
$$\dds z_i\left(s,\chi_i(s;t,x)\right)=0, \quad \forall i \in \ens{1,\ldots,m}.$$

On the other hand, since we know some information of $h_{ij}$ along the characteristic curve $s \mapsto \chi_{ij}(s;t,x,\theta)=(s,\chi_i(s;t,x),\chi_j(s;t,\theta))$, we would like to perform the change of variable
$$\xi=\chi_j(s;t,\theta).$$

Thanks to \eqref{chi increases} and the implicit function theorem there exists $\theta_j \in C^1(\R^3)$ such that, for every $(s,t,\xi) \in \R^3$, we have
\begin{equation}\label{def thetaj}
\xi=\chi_j(s;t,\theta_j(s;t,\xi)),
\qquad \pxi{\theta_j}(s;t,\xi)>0.
\end{equation}

Using this change of variable, we have
$$
\gamma_i\left(s,\chi_i(s;t,x)\right)=z_i\left(s,\chi_i(s;t,x)\right)
-\sum_{j=1}^m \int_{a_j(s)}^{b_j(s)} \eta_{ij}(s,\theta) \, d\theta,
$$
where
$$a_j(s)=\theta_j(s;t,0),
\qquad
b_j(s)=\theta_j(s;t,1),$$
$$
\eta_{ij}(s,\theta)=
h_{ij}(\chi_{ij}(s;t,x,\theta))z_j(s,\chi_j(s;t,\theta)) \px{\chi_j}(s;t,\theta).
$$

We would like to use the formula
$$
\dds \left(\int_{a_j(s)}^{b_j(s)} \eta_{ij}(s,\theta) \, d\theta\right)=
b_j'(s)\eta_{ij}(s,b_j(s))-a_j'(s)\eta_{ij}(s,a_j(s))+\int_{a_j(s)}^{b_j(s)} \pas{\eta_{ij}}(s,\theta) \, d\theta.
$$

Clearly, $a_j, b_j \in C^1(\R)$.
Differentiating the relation $\xi=\chi_j(s;t,\theta_j(s;t,\xi))$ with respect to $s$ we obtain
$$
a_j'(s)=\frac{-\lambda_j(s,0)}{\px{\chi_j}(s;t,\theta_j(s;t,0))}.
$$
On the other hand, using \eqref{ODE kerh} with $\xi=0$, \eqref{def thetaj} and the boundary condition $z_-(\cdot,1)=0$, we have
$$
\eta_{ij}(s,a_j(s))
=-\frac{g^2_{ij}(s,\chi_i(s;t,x))}{\lambda_j(s,0)}
z_j(s,0) \px{\chi_j}(s;t,\theta_j(s;t,0)),
\qquad
\eta_{ij}(s,b_j(s))=0.
$$

Using the ODEs satisfied  along the characteristics by $h_{ij}$ (see \eqref{ODE kerh}) and $z_j$, and using the relation (see \eqref{deriv chi})
$$
\frac{\partial^2 \chi_j}{\partial s \partial x}(s;t,\theta)=\frac{\partial\lambda_{j}}{\partial x}(s,\chi_j(s;t,\theta))\px{\chi_j}(s;t,\theta),
$$
we can check that $\eta_{ij}$ has weak derivative with respect to $s$ which is equal to zero:
$$\pas{\eta_{ij}}(s,\theta)=0.$$

It follows from all the previous computations and the relation $z(\cdot,0)=\gamma(\cdot,0)$ that
$$
\dds \gamma_i\left(s,\chi_i(s;t,x)\right)
=\sum_{j=1}^m a_j'(s)\eta_{ij}(s,a_j(s))
=\sum_{j=1}^m g^2_{ij}(s,\chi_i(s;t,x))\gamma_j(s,0).
$$

\end{enumerate}

\end{proof}

\section{Constructions of $\sinterin_{ij}, \sinterout_{ij}$ and $\sdisc_{ij\ell}$}\label{app prop}

In this appendix, we give a proof of the existence $\sinterin_{ij}, \sinterout_{ij}$ and $\sdisc_{ij\ell}$ satisfying the properties stated in Proposition \ref{prop caract exit} and Proposition \ref{prop sdisc}.
We will make use of the following simple lemma:

\begin{lemma}\label{lem TVI}
Let $f \in C^1([a,b])$ ($a<b$) satisfy the following property:
\begin{equation}\label{prop lem TVI}
\forall s \in [a,b],
\quad f(s)=0 \quad \Longrightarrow \quad f'(s)<0.
\end{equation}
Then, there exists a unique $c \in [a,b]$ such that
$$
f(s)>0, \quad \forall s \in (a,c), \quad f(s)<0, \quad \forall s \in (c,b).
$$
Moreover, $c$ has the properties listed in Table \ref{tableau} (an $\emptyset$ means that such a situation can not occur).

\begin{table}[h]
\centering
\begin{tabular}{|c|c|c|c|}
\hline
& $f(b)>0$ & $f(b)=0$ & $f(b)<0$ \\
\hline
$f(a)>0$ & $c=b$ & $c=b$ & $f(c)=0$ \\
\hline
$f(a)=0$ & $\emptyset$ & $\emptyset$ & $c=a$ \\
\hline
$f(a)<0$ & $\emptyset$ & $\emptyset$ & $c=a$ \\
\hline
\end{tabular}
\caption{\label{tableau} Properties of $c$}
\end{table}

\end{lemma}

\begin{proof}[Proof of Proposition \ref{prop caract exit}]
We recall that $i \in \ens{1,\ldots,m}$ and we refer to Figures \ref{figure prop caract 1}, \ref{figure prop caract 2}, \ref{figure prop caract 3} and \ref{figure prop caract 4} for a clarification of the geometric situation (at a fixed $t$).
We only focus on the existence part since the uniqueness readily follows from the properties that have to be satisfied.

\begin{enumerate}[1)]
\item
Assume that $j \in \ens{1,\ldots,i-1}$.
For every $(t,x,\xi) \in \clos{\Taub}$ such that $x<1$, we introduce the $C^1$ function
$$f:s \in \left[\max\ens{\ssin_i(t,x),\ssin_j(t,\xi)},t\right] \mapsto \chi_j(s;t,\xi)-\chi_i(s;t,x).$$
Note that the interval has a non empty interior since $x<1$ and $\xi \leq x<1$ (see \eqref{tentredeux}-\eqref{sin is zero}).
This function clearly satisfies the property \eqref{prop lem TVI} thanks to the ODE \eqref{ODE xi} and the assumption \eqref{hyp speeds} since $j<i$.
Consequently, Lemma \ref{lem TVI} applies and gives the existence of $\sinterin_{ij}(t,x,\xi)$ with
$$\max\ens{\ssin_i(t,x),\ssin_j(t,\xi)} \leq \sinterin_{ij}(t,x,\xi) \leq t,$$
and such that
$$\chi_j(s;t,\xi)<\chi_i(s;t,x), \quad \forall s \in \left(\sinterin_{ij}(t,x,\xi),t\right).$$
Clearly, $(t,x,\xi) \mapsto t-\sinterin_{ij}(t,x,\xi) \in L^{\infty}(\Taub)$ thanks to \eqref{bounds sin sout}.
Moreover, it follows from Table \ref{tableau} that
\begin{equation}\label{def 1 sij}
\left\{\begin{array}{ll}
\sinterin_{ij}(t,x,\xi)=t & \mbox{ if } \ssin_i(t,x)<\ssin_j(t,\xi) \mbox{ and } \xi=x , \\
f(\sinterin_{ij}(t,x,\xi))=0 & \mbox{ if } \ssin_i(t,x)<\ssin_j(t,\xi) \mbox{ and } \xi<x,
\\
\sinterin_{ij}(t,x,\xi)=\ssin_i(t,x) & \mbox{ if } \ssin_i(t,x)=\ssin_j(t,\xi) \mbox{ and } \xi<x,
\\
\sinterin_{ij}(t,x,\xi)=\ssin_i(t,x) & \mbox{ if } \ssin_i(t,x)>\ssin_j(t,\xi) \mbox{ and } \xi<x.
\end{array}\right.
\end{equation}
Let us now complete the definition of $\sinterin_{ij}$ on the remaining parts of $\clos{\Taub}$.
The missing case in \eqref{def 1 sij} is when $\ssin_i(t,x) \geq \ssin_j(t,\xi)$ and $\xi=x$.
However, unless $x=1$, these conditions are not compatible since $\ssin_i(t,x)<\ssin_j(t,x)$ for $j<i \leq m$ (see \eqref{comparison ssout}).
Consequently, it only remains to define $\sinterin_{ij}$ in the part where $x=1$, which we do now by setting
\begin{equation}\label{def 2 sij}
\sinterin_{ij}(t,1,\xi)=t.
\end{equation}
We can check that $\sinterin_{ij}$ defined by \eqref{def 1 sij}-\eqref{def 2 sij} belongs to $C^0(\clos{\Taub})$ (for the second case in \eqref{def 1 sij} this follows from the implicit function theorem).
Therefore, such a $\sinterin_{ij}$ clearly satisfies all the properties claimed in the statement of item \ref{exit 1} of Proposition \ref{prop caract exit}.

\item
Assume that  $j=i$.
We will show that, in this case, we can simply take
$$\sinterout_{ii}(t,x,\xi)=\ssout_i(t,\xi).$$
Clearly, $\sinterout_{ii} \in C^0(\clos{\Taub})$ with $(t,x,\xi) \mapsto \sinterout_{ii}(t,x,\xi)-t \in L^{\infty}(\Taub)$ thanks to \eqref{bounds sin sout} and $\sinterout_{ii}(t,x,\xi)>t$ as long as $\xi>0$ (see \eqref{tentredeux}-\eqref{sin is zero}).
Let us now observe that, for $\xi<x$, we have from \eqref{chi increases}:
$$\chi_i(s;t,\xi)<\chi_i(s;t,x), \quad \forall s \in \R,$$
and $\chi_i(s;t,\xi)>0$ for $s \in (t,\sinterout_{ii}(t,x,\xi))$ since $\ssout_i(t,\xi)< \ssout_i(t,x)$ by \eqref{increases} (recall that $j=i \in \ens{1,\ldots,m}$).

\item
The proof for the case $j \in \ens{i+1,\ldots,m}$ is similar to the proof of part 1) by considering, for each $(t,x,\xi) \in \clos{\Taub}$ such that $\xi,x>0$, the function
\begin{equation}\label{def-F}
f:s \in \left[t,\min\ens{\ssout_i(t,x),\ssout_j(t,\xi)}\right] \mapsto \chi_i(s;t,x)-\chi_j(s;t,\xi).
\end{equation}

\item
The proof for the case $j \in \ens{m+1,\ldots,n}$ is also similar to the proof of part 1) by considering, for each $(t,x,\xi) \in \clos{\Taub}$ such that $0 \leq \xi<x \leq 1$, the function $f$ defined again by \eqref{def-F}.
\end{enumerate}
\end{proof}

\begin{proof}[Proof of Proposition \ref{prop sdisc}]
The difference with the proof of Proposition \ref{prop caract exit} is that we do not need to neither track the regularity of the point where the function $f$ vanishes nor its sign on the left and right of this zero.
It is a straightforward consequence of Lemma \ref{lem TVI} applied to the following functions (it is enough to consider non empty intervals):
\begin{enumerate}[1)]
\item
For $j \in \ens{1,\ldots,i-1}$, we use
$$f:s \in \left[\sinterin_{ij}(t,x,\xi),t\right] \mapsto \chi_j(s;t,\xi)-\psi_{i \ell}\left(s,\chi_i(s;t,x)\right).$$
Using the ODE \eqref{ODE xi} satisfied by $\chi_j$ and using the equation \eqref{equ psi} satisfied by $\psi_{i\ell}$, we have
$$
f'(s)=\lambda_j\left(s,\chi_j(s;t,\xi)\right)
-\lambda_{\ell}\left(s,\psi_{i\ell}\left(s,\chi_i(s;t,x)\right)\right).
$$
Since $j<\ell$, this shows that such a $f$ satisfies the property \eqref{prop lem TVI} of Lemma \ref{lem TVI}.

\item
For $j \in \ens{i,\ldots,\ell-1}$, we use
\begin{equation}\label{def F(s)}
f:s \in \left[t,\sinterout_{ij}(t,x,\xi)\right] \mapsto \chi_j(s;t,\xi)-\psi_{i \ell}\left(s,\chi_i(s;t,x)\right).
\end{equation}

\item
For $j \in \ens{\ell+1,\ldots,m}$, we use the function $-f$, where $f$ is given by \eqref{def F(s)}.

\item
For $j \in \ens{m+1,\ldots,n}$, we use the same function $f$ given by \eqref{def F(s)} (in fact, the result then directly follows from the intermediate value theorem).
\end{enumerate}

\end{proof}

\section{Construction of $\Omega_i$}\label{app lemma}

This appendix is devoted to the proof of Lemma \ref{key lemma}, that is to the existence of the key change of variable needed in the proof of Proposition \ref{contraction estimates}.
We recall that $i \in \ens{1,\ldots,m}$.

\begin{enumerate}[1)]
\item
Inspired by the time-independent case (see Remark \ref{contraction estimates}), we look for $\Omega_i$ in the following form:
$$\Omega_i(t,x,\xi)=\omega^1_i(t,x)-\omega^{\nu}_i(t,\xi),$$
where, at each fixed $\nu \in (0,1]$, $\omega^{\nu}_i(\cdot,\cdot)$ is the solution to the following linear hyperbolic equation:
\begin{equation}\label{def omegai}
\left\{\begin{array}{l}
\ds \pt{\omega^{\nu}_i}(t,x)+\frac{\lambda_i(t,x)}{\nu}\px{\omega^{\nu}_i}(t,x)=0,\\
\omega^{\nu}_i(t,0)=t,
\end{array}\right.
\qquad t \in \R, \, x \in [0,1].
\end{equation}
The solution of \eqref{def omegai} is explicit:
\begin{equation}\label{def omegai ex}
\omega^{\nu}_i(t,x)=\omega^{\nu}_i(\ssoutnu_i(t,x),0)=\ssoutnu_i(t,x),
\end{equation}
where $\ssoutnu_i(t,x) \geq t$ (with $\ssoutnu_i(t,x)=t \Longleftrightarrow x=0$) is the unique number such that
\begin{equation}\label{def soutnu}
\chi^{\nu}_i(\ssoutnu_i(t,x);t,x)=0,
\end{equation}
where $s \mapsto \chi^{\nu}_i(s;t,x)$ is the solution to the ODE
\begin{equation}\label{ODE xinu}
\left\{\begin{array}{l}
\ds \pas{\chi^{\nu}_i}(s;t,x)=\frac{1}{\nu}\lambda_i(s,\chi^{\nu}_i(s;t,x)), \quad \forall s \in \R,\\
\chi^{\nu}_i(t;t,x)=x.
\end{array}\right.
\end{equation}
We can check that the map $(t,x,\nu) \mapsto \omega^{\nu}_i(t,x)$ belongs to $C^1(\R\times[0,1]\times(0,1])$.

\item
We now prove that there exists $\delta>0$ such that, for every $t \in \R$, $x \in [0,1]$ and $\nu \in (0,1]$,
\begin{equation}\label{lbounds omegai}
\pt{\omega^{\nu}_i}(t,x) \geq \epsilon\delta,
\qquad \px{\omega^{\nu}_i}(t,x) \geq \nu\delta,
\qquad \pnu{\omega^{\nu}_i}(t,x) \geq 0.
\end{equation}

Using the equation \eqref{def omegai} and the assumption \eqref{hyp speeds}, it is clear that the estimate for $\partial \omega^{\nu}_i/\partial t$ follows from the estimate of $\partial \omega^{\nu}_i/\partial x$.
Note from \eqref{def omegai ex} that $\partial \omega^{\nu}_i/\partial x=\partial \ssoutnu_i/\partial x$.
Taking the derivative of \eqref{def soutnu} with respect to $x$, we obtain
$$
\frac{1}{\nu}\lambda_i(\ssoutnu_i(t,x),\chi^{\nu}_i(\ssoutnu_i(t,x);t,x))
\px{\ssoutnu_i}(t,x)
+\px{\chi^{\nu}_i}(\ssoutnu_i(t,x);t,x)
=0.
$$
Since $\lambda_i \in L^{\infty}(\R\times(0,1))$, we have to bound $\px{\chi^{\nu}_i}(\ssoutnu_i(t,x);t,x)$ from below by a positive constant that does not depend on $t,x$ and $\nu$.
From \eqref{ODE xinu} we can show that
$$\px{\chi^{\nu}_i}(s;t,x)=e^{\frac{1}{\nu}\int_t^s \px{\lambda_i}(\theta,\chi^{\nu}_i(\theta;t,x)) \, d\theta},$$
so that
$$\px{\chi^{\nu}_i}(\ssoutnu_i(t,x);t,x) \geq e^{\frac{1}{\nu} (\ssoutnu_i(t,x)-t)  \inf_{\R\times[0,1]} \px{\lambda_i}}.$$
This establishes the desired lower bound since $\px{\lambda_i} \in L^{\infty}(\R\times(0,1))$ and $0 \leq \ssoutnu_i(t,x)-t \leq \frac{x}{\epsilon} \nu$ (the proof is similar to the one of \eqref{bounds sin sout}).
Note that it follows as well from this estimate that $\Omega_i \in L^{\infty}(\R\times(0,1)\times(0,1))$.

To prove the remaining estimate in \eqref{lbounds omegai}, we denote by $\gamma^{\nu}=\partial \omega^{\nu}_i/\partial \nu$ and observe that it satisfies
$$
\left\{\begin{array}{l}
\ds \pt{\gamma^{\nu}}(t,x)+\frac{\lambda_i(t,x)}{\nu}\px{\gamma^{\nu}}(t,x)
=\frac{\lambda_i(t,x)}{\nu^2} \px{\omega^{\nu}_i}(t,x) \leq 0,\\
\gamma^{\nu}(t,0)=0,
\end{array}\right.
\qquad t \in \R, \, x \in [0,1].
$$
It immediately follows that $\gamma^{\nu} \geq 0$.

\item
Let us now check the estimates \eqref{Omega j<i} and \eqref{Omega j>i}.
We have
\begin{multline*}
\pt{\Omega_i}(t,x,\nu)+\lambda_i(t,x)\px{\Omega_i}(t,x,\nu)+\lambda_j(t,\xi)\pxi{\Omega_i}(t,x,\nu)
\\
=\pt{\omega^1_i}(t,x)+\lambda_i(t,x)\px{\omega^1_i}(t,x)
-\pt{\omega^{\nu}_i}(t,\xi)-\lambda_j(t,\xi)\px{\omega^{\nu}_i}(t,\xi)
\\
=-\pt{\omega^{\nu}_i}(t,\xi)\left(1-\nu \frac{\lambda_j(t,\xi)}{\lambda_i(t,\xi)}\right).
\end{multline*}

Since $\lambda_j/\lambda_i \leq 1$ for $i \leq j$ and $i \leq m$, we see that the estimate \eqref{Omega j>i} is obtained by simply taking
\begin{equation}\label{cond nu 1}
0<\epsilon_0 \leq \epsilon\delta(1-\nu),
\qquad
0<\nu<1.
\end{equation}

On the other hand, let us introduce
$$r=\max_{1 \leq j<i} \sup_{\substack{t \in \R \\ \xi \in [0,1]}} \frac{\lambda_i(t,\xi)}{\lambda_j(t,\xi)}.$$
Clearly, $0< r \leq 1$.
In fact, $r<1$ since from \eqref{hyp speeds bis} we have, for $j<i \leq m$,
$$\sup_{\substack{t \in \R \\ \xi \in [0,1]}} \frac{\lambda_i(t,\xi)}{\lambda_j(t,\xi)} \leq 1-\frac{\epsilon}{\norm{\lambda_j}_{L^{\infty}}}.$$
The estimate \eqref{Omega j<i} now follows from \eqref{lbounds omegai} by taking
\begin{equation}\label{cond nu 2}
0<\epsilon_0 \leq \epsilon\delta\left(\frac{\nu}{r}-1\right),
\qquad
r<\nu \leq 1.
\end{equation}
Note that the conditions \eqref{cond nu 1} and \eqref{cond nu 2} are compatible by taking $\nu$ close enough to $1$.

\item
It remains to check that $\Omega_i \geq 0$ on $\clos{\Taub}$.
Since both functions $\nu \mapsto \omega^{\nu}_i(t,x)$ and $x \mapsto \omega^{\nu}_i(t,x)$ are nondecreasing by \eqref{lbounds omegai} and $\xi \leq x$, we have
$$\omega^1_i(t,x) \geq \omega^{\nu}_i(t,x) \geq \omega^{\nu}_i(t,\xi).$$

\qed

\end{enumerate}

\bibliographystyle{amsalpha}
\bibliography{biblio}

\end{document}